\documentclass[12pt,english]{elsarticle}
\usepackage{amsmath}
\usepackage{amsthm}
\usepackage{amssymb}
\usepackage{geometry}
\makeatletter
\def\ps@pprintTitle{%
 \let\@oddhead\@empty
 \let\@evenhead\@empty
 \def\@oddfoot{\footnotesize\itshape
       Published Open Access in Journal of Differential Equations \hfill \today}%
 \let\@evenfoot\@oddfoot}
\makeatother

\makeatletter
\theoremstyle{plain}
\newtheorem{thm}{\protect\theoremname}
\theoremstyle{plain}
\newtheorem{prop}[thm]{\protect\propositionname}
\theoremstyle{remark}
\newtheorem{rem}[thm]{\protect\remarkname}
\theoremstyle{plain}
\newtheorem{lem}[thm]{\protect\lemmaname}
\theoremstyle{definition}
\newtheorem{defn}[thm]{\protect\definitionname}
\theoremstyle{definition}
\newtheorem{example}[thm]{\protect\examplename}

\usepackage{amscd}
\usepackage{xspace}
\usepackage{mathrsfs}
\usepackage[colorlinks=true, allcolors=blue]{hyperref}

\DeclareFontFamily{OML}{cyr}{}
\DeclareFontShape{OML}{cyr}{m}{n}{
   <5> <6> <7> <8> <9> gen * wncyr
   <10> <10.95> <12> <14.4> <17.28> <20.74> <24.88> wncyr10
  }{}
\DeclareSymbolFont{rusletters}{OML}{cyr}{m}{n}
\DeclareSymbolFontAlphabet{\rusmath}{rusletters}
\DeclareMathSymbol\re{\rusmath}{rusletters}{"03}

\makeatother

\usepackage[english]{babel}
\providecommand{\definitionname}{Definition}
\providecommand{\examplename}{Example}
\providecommand{\lemmaname}{Lemma}
\providecommand{\propositionname}{Proposition}
\providecommand{\remarkname}{Remark}
\providecommand{\theoremname}{Theorem}

\begin{document}

\begin{frontmatter}

\title{\vspace{-2cm} 
    {\footnotesize \textnormal{Published Open Access in \textit{Journal of Differential Equations}. 
    \href{https://doi.org/10.1016/j.jde.2026.114334}{Download the official version here.}}} \\ [1ex]
    Nonlocal pseudosymmetries and Bäcklund transformations as $\mathcal{C}$-morphisms}
    
\author[rvt]{Diego Catalano Ferraioli\fnref{fn1}\corref{cor1}}
\ead{diego.catalano@ufba.br}
\author[focal]{Tarcísio Castro Silva\fnref{fn2}}
\fntext[fn1]{Partially supported by CNPq grant 405427/2025-5.}
\fntext[fn2]{Partially supported by CNPq grant 405427/2025-5 and by FEMAT.}
\cortext[cor1]{Corresponding author}
\address[rvt]{Department of Mathematics, Universidade Federal da Bahia, Campus
de Ondina, Av. Milton Santos, S/N, Ondina - CEP 40.170.110 - Salvador,
BA - Brazil, e-mail: diego.catalano@ufba.br.}
\address[focal]{Department of Mathematics, Universidade de Brasilia, Brasilia DF
70910-900, Brazil, e-mail: tarcisiocastro@unb.br. }
\begin{abstract}
In this paper, we show how factorisation with respect to nonlocal
pseudosymmetries enables us to obtain Bäcklund transformations, which
are described as nonlocal $\mathcal{C}$-morphisms of differential
equations. It follows that, in this approach, Bäcklund transformations
are determined by basic invariants of the exploited nonlocal pseudosymmetries.
As illustrated in several representative examples, including one involving
a new integrable equation, our resulting framework enables a very
general structural approach, distinct from the case-by-case ones often
adopted in the literature---especially in the analysis of cases where
the sought Bäcklund transformations also affect the independent variables.
\end{abstract}
\end{frontmatter}{}

\medskip{}

%s1 #&#
\section{Introduction}
\label{sec1}

B\"{a}cklund transformations, although originally arising within the realm
of classical differential geometry, have come to play a significant role
in the modern analysis of nonlinear phenomena. Indeed, this type of transformation
is named after the mathematician traditionally associated with a theorem
that studies the transformation of surfaces with constant negative curvature
in $\mathbb{R}^{3}$, realizing them as focal surfaces of a pseudospherical
congruence of lines. In practice, starting from a surface with constant
negative Gaussian curvature, B\"{a}cklund's theorem allows for the construction
of a one-parameter family of surfaces by integrating only a system of ordinary
differential equations (ODEs). Since such surfaces correspond, from an
analytical point of view, to solutions of the sine-Gordon equation
$z_{xt}=\sin z$, B\"{a}cklund's theorem provides a method for generating
new solutions of the sine-Gordon equation from a given one, again through
the integration of a system of ODEs. In doing so, this transformation acts
as an automorphism of the sine-Gordon solution space --- a property that
justifies the term auto-B\"{a}cklund transformation.

In its modern formulation, an auto-B\"{a}cklund transformation for a differential
equation $\mathcal{E}$ refers to a map that associates, to any given solution
of $\mathcal{E}$, a family of new solutions of $\mathcal{E}$, typically
by integrating a system of ODEs. Then, more generally, a B\"{a}cklund transformation
is understood as a transformation that, via the integration of a system
of ordinary differential equations, maps solutions of one differential
equation $\mathcal{E}$ to solutions of another equation
$\mathcal{E}'$.

Given a nonlinear differential equation, the problem of determining the
existence of its auto-B\"{a}cklund transformations---or, more generally,
of B\"{a}cklund transformations between this equation and other equations
of interest---is an important problem that has attracted the attention
of many mathematicians (see, for instance
\cite{And-Ibr,Lab74,Lamb80,Nimmo,Pirani1,Rogers-Schief,Zvyagin} and references
therein). However, it is a difficult problem which is still open in its
full generality, despite having been studied for a long time and, in recent
years, interesting progress having been made in some particular cases (see,
for instance, the works
\cite{AndFel2015,AndFel2016,Clelland2002,Clelland2009,Clelland2018,Yuhao1,Yuhao2}).

On the other hand, a notion closely related to that of a B\"{a}cklund transformation
is that of a differential substitution. In fact, a differential substitution
from an equation $\mathcal{E}$ to an equation $\mathcal{E}'$ is a transformation
that maps solutions of $\mathcal{E}$ to solutions of $\mathcal{E}'$, without
requiring the integration of a system of ordinary differential equations.
Classical examples of differential substitutions include the Cole--Hopf
transformation, which maps solutions of the heat equation to solutions
of Burgers equation, and the Miura transformation, which maps solutions
of the modified Korteweg-de Vries (mKdV) equation to solutions of the Korteweg-de
Vries (KdV) equation.

As with B\"{a}cklund transformations, given a differential equation
$\mathcal{E}$, it is often of great interest to know whether a differential
substitution exists that maps its solutions into those of another equation
$\mathcal{E}'$. Regarding this issue, in the context of evolutionary equations,
Sokolov took a significant step forward in \cite{Sokolov}, highlighting
a particularly important relationship between factorisation via pseudosymmetries
and differential substitutions. However, although Sokolov underscored the
importance of pseudosymmetries through this discovery, these have been
surprisingly underexplored in the literature. Indeed, aside from the works
of Sokolov himself, pseudosymmetries have been little studied in their
full generality. In a simplified form, they have been used in several works
mainly to reduce and integrate ordinary differential equations (see, for
example, \cite{Diego-lambda,Diego-Paola,Diego-Giuseppe,Muriel-Romero} and
the references therein). While for the particular case of integrable pseudosymmetries,
some possible uses have been discussed in the study of differential coverings
\cite{Chet,Chet-2} (see \cite{Kras-Vin} for the notion of differentiable
covering) or as a tool to explore the effectiveness and generating power
of an already known B\"{a}cklund transformation \cite{Marvan}. Therefore,
no previous work on using pseudosymmetries addresses the important problem
of determining B\"{a}cklund transformations.

A particularly convenient approach to B\"{a}cklund transformations is based
on the notion of $\mathcal{C }$-morphism. Indeed, all B\"{a}cklund
transformations described in the literature so far can be treated as
$\mathcal{C }$-morphisms connecting differential equations.
After discussing this approach to B\"{a}cklund transformations, this paper
presents a previously unexplored connection between B\"{a}cklund transformations
and a type of pseudosymmetry, hereafter referred to as nonlocal pseudosymmetries.
Indeed, we demonstrate how factorisation by nonlocal pseudosymmetries enables
the derivation of B\"{a}cklund transformations, determined by the basic
invariants of the exploited nonlocal pseudosymmetries. This is done in
the case of equations admitting zero-curvature representations (ZCRs),
where it is natural to consider the nonlocal pseudosymmetries defined in
a kind of Riccati-type differentiable coverings, which are shown to be
determined by such ZCRs.

As illustrated by the application to several representative examples, including
one involving a novel integrable equation, our framework allows for a very
general structural approach, distinct from the case-by-case ones often
adopted in the literature---especially in the analysis of cases where the
sought B\"{a}cklund transformations also affect the independent variables.

The paper is organised as follows. Section~\ref{sec:Preliminaries} is a section of preliminaries,
devoted to some general aspects of the geometry of differential equations
and a discussion of more specific notions that will play a particularly
significant role in the remainder of the paper. This section is divided
into the following four subsections: the first is devoted to a brief introduction
to differential equations as submanifolds of jet spaces; the second deals
with $\mathcal{C}$-morphisms; the third, after a short review of the notion
of symmetry, provides a discussion of the notion of pseudosymmetry together
with a brief description of how this notion leads to a factorisation procedure;
finally the fourth subsection reviews the theory of differential coverings.
In particular, some significant examples are included in this section to
illustrate the central notions discussed here. Next, starting from ZCRs
of $\mathcal{E}$, in Section~\ref{subsec:Riccati-type-differentiable-cove} we show how to determine Riccati-type differential
coverings of $\mathcal{E}$ and corresponding nonlocal conservation laws.
This type of covering, with the related conservation laws, is particularly
useful for determining possible nonlocal pseudosymmetries of
$\mathcal{E}$. Finally, Section~\ref{sec:last} shows how factorisation using nonlocal
pseudosymmetries of an equation $\mathcal{E}$ can provide B\"{a}cklund
transformations for $\mathcal{E}$. This approach to determining B\"{a}cklund
transformations is illustrated here through several representative examples,
including one involving a novel integrable equation, which demonstrate
not only the method's degree of generality but also its applicability.

Throughout the paper the following main notations and conventions are adopted:
\begin{itemize}
\item[{-}] all objects in the paper, e.g., manifolds, mappings, functions,
vector fields, etc, are supposed to be smooth;
\item[{-}] in addition to vector fields, intended as derivations of the
algebras of smooth functions, we will also consider
\textsl{relative
vector fields} (also referred to in the literature as vector fields along
maps) \cite{BG,JN} that naturally occur when considering projections of
vector fields in a fiber bundle;
\item[{-}] for jet bundles $J^{k}(\pi )$ and $J^{\infty}(\pi )$ of a fiber
bundle $\pi :E\rightarrow M$ (in particular for total derivatives, Cartan
distributions, prolongations, symmetries, horizontal and vertical differentials,
conservation laws, zero-curvature representations, and differential coverings)
we will use the notations and conventions revised in Subsections \ref{subsec:Jet-spaces-and}, \ref{subsec:Symmetries-of-differential} and \ref{subsec:Cov-BT};
\item[{-}] $(x_{i},u_{\sigma}^{j})$ usually denote the
\textit{canonical coordinates} in a $k$-order jet bundle
$J^{k}(\pi )$;
\item[{-}] the algebra of smooth functions on a manifold $N$ is usually
denoted by $C^{\infty}(N)$, whereas $\mathcal{F}_{k}(\pi )$ and
$\mathcal{F}(\pi )$ will denote the algebra of
\textit{smooth functions on} $J^{k}(\pi )$ and $J^{\infty}(\pi )$, respectively;
\item[{-}] by considering a system of differential equations
$\mathcal{E}\subset J^{k}(\pi )$, we usually describe it as
$\mathcal{E}=\{F^{1}(x,u_{\sigma})=0,...,F^{r}(x,u_{\sigma})=0\}$;
\item[{-}] when $\mathcal{E},\mathcal{V}\subset J^{k}(\pi )$ are such that
$\mathcal{V}$ is a $q$-dimensional differentiable extension of
$\mathcal{E}$, i.e., $\mathcal{V}$ defines a $q$-dimensional differentiable
covering
$\mathcal{V}^{(\infty )}\rightarrow \mathcal{E}^{(\infty )}$, we will usually
distinguish the fiber coordinates $\{u^{1},...,u^{m}\}$ of $\pi $ into
the two subsets $\{z^{1},...,z^{m-q}\}$ (also referred to as
\textit{local variables}) and $\{v^{1},...,v^{q}\}$ (also referred to as
\textit{nonlocal variables}), such that
$\mathcal{E}=\{F^{1}(x,z_{\sigma})=0,...,F^{r}(x,z_{\sigma})=0\}$ and
$\mathcal{V}=\{F^{h}(x,z_{\sigma})=0,\,v_{i}^{s}=X_{i}^{s}(x,z_{
\sigma},v)\}$, where $h=1,...,r$, $s=1,...q$ and $i$ enumerates the independent
variables $x_{1},...,x_{n}$;
\item[{-}] for Riccati-type differentiable coverings of an equation
$\mathcal{E}$, we will usually denote nonlocal variables by
$\{\rho ^{j}\}$ and the differentiable extension of $\mathcal{E}$ by
$\mathcal{Y}$;
\item[{-}] by a ``basic system of invariants'', we mean a system of invariants
that generates all invariants through prolongations and functional combinations,
whereas by a ``complete system of invariants'', we mean a system of invariants
that generates all invariants only through functional combinations.
\end{itemize}

%s2 #&#
\section{Preliminaries}
%%LEAP%%%\label{sec2}
\label{sec:Preliminaries}

We assume the reader is familiar with the geometric theory of differential
equations. However, to make the main results accessible to a wide range
of readers, we collected here some notations and basic facts of this theory
used throughout the paper. We also introduce lesser-known topics that will
be key in the forthcoming sections. The reader is referred to
\cite{B_KrV,KLV,olver,Stormark} and
\cite{And-Ibr,DiegoLuiz,Dun,Ibragimov,Marvan1,Tesi_Luiz,Sokolov}, as well
as references therein, for further details.

%s2.1 #&#
\subsection{Differential equations as submanifolds of jet spaces}
%%LEAP%%%\label{sec2.1}
\label{subsec:Jet-spaces-and}

Consider a fiber bundle $\pi :E\rightarrow M$, with $dim\,M=n$ and
$dim\,E=n+m$. Given a smooth (local) section $s$ of $\pi $ at
$a\in M$, for any $k\in \mathbb{N}$, we denote by $[s]_{a}^{k}$ the
$k$-th order jet of $s$ at $a$; by definition $[s]_{a}^{k}$ is the equivalence
class of (local) smooth sections of $\pi $ that are $k$-order tangent (or,
have $k$-fold contact) to $s$ at $a$. The space $J^{k}(\pi )$ of $k$-th
order jets of sections of $\pi $ is naturally equipped with a differentiable
manifold structure, induced by bundle atlas of $\pi $. If
$\{x_{1},...,x_{n}\}$ are local coordinates on $M$ and
$\{u^{1},...,u^{m}\}$ local fiber coordinates of $\pi $, the induced
\textit{canonical coordinates} on $J^{k}(\pi )$ will be denoted by
$\{x_{i},u_{\sigma}^{j}\}$, where $i\in \{1,...,n\}$,
$j\in \{1,...,m\}$ and $\sigma =(\sigma _{1},...,\sigma _{n})$ is a multi-index
of order $|\sigma |=\sigma _{1}+...+\sigma _{n}$ such that
$0\leq |\sigma |\leq k$; by definition, if $\theta =[s]_{a}^{k}$, then
$x_{i}(\theta ):=x_{i}(a)$ and
$u_{\sigma}^{j}(\theta ):=
\frac{\partial ^{|\sigma |}s^{j}}{\partial x_{1}^{\sigma _{1}}...\partial x_{n}^{\sigma _{n}}}(a)$.
The manifold $J^{k}(\pi )$ is usually referred to as the $k$\textit{-order
jet bundle} of sections of $\pi $, since for any $k\in \mathbb{N}$ the
natural projection $\pi _{k}:J^{k}(\pi )\rightarrow M$,
$[s]_{a}^{k}\mapsto a$, is a fiber bundle; in canonical coordinates,
$\pi _{k}$ has the form $(x_{i},u_{\sigma}^{j})\mapsto (x_{i})$. In particular,
$J^{0}(\pi )$ can be identified with $E$ and $\pi _{0}$ with $\pi $.

Throughout the paper, by $\mathcal{F}_{k}(\pi )$ we will denote the algebra
of \textit{smooth functions on} $J^{k}(\pi )$, $k\geq 1$. Moreover, to emphasize
that $f\in \mathcal{F}_{k}(\pi )$ depends only on
$(x_{i},u_{\sigma}^{j})$, with $|\sigma |\leq k$, we sometimes use the
notation $f=f(x,u^{(k)})$. Also, for lower order $k$, notations like
$u_{x_{i}}^{j},u_{x_{i}x_{j}}^{j},...$ will usually be preferred to multi-index
notation in concrete computations.

Now, since for any $h>k$ the \textit{natural projections}
$\pi _{h,k}:J^{h}(\pi )\rightarrow J^{k}(\pi )$,
$[s]_{a}^{h}\rightarrow [s]_{a}^{k}$, are fiber bundles, one can also define
the \textit{infinite jet space} $J^{\infty}(\pi )$ as the inverse limit
of the sequence
$M\overset{\pi}{\longleftarrow}J^{0}(\pi )
\overset{\pi _{1,0}}{\longleftarrow}...
\overset{\pi _{k,k-1}}{\longleftarrow}J^{k}(\pi )
\overset{\pi _{k+1,k}}{\longleftarrow}...$. By definition,
$J^{\infty}(\pi )$ is the space of sequences
$\theta =\{\theta _{i}\}_{i\in \mathbb{N}}$ with
$\theta _{i}\in J^{i}(\pi )$ and such that
$\pi _{h,k}(\theta _{h})=\theta _{k}$, for all $h>k$. Despite it is not
a finite dimensional manifold, one can still introduce a differential calculus
on $J^{\infty}(\pi )$ by making use of standard constructions of differential
calculus over commutative algebras \cite{B_KrV,KLV}. For instance, one
can define the algebra $\mathcal{F}(\pi )$ of
\textit{smooth functions on} $J^{\infty}(\pi )$ as the filtered algebra
given by direct limit of the sequence of inclusions
$C^{\infty}(M)\overset{\pi ^{*}}{\longrightarrow}\mathcal{F}_{0}(\pi )
\overset{\pi _{1,0}^{*}}{\longrightarrow}...
\overset{\pi _{k,k-1}^{*}}{\longrightarrow}\mathcal{F}_{k}(\pi )
\overset{\pi _{k+1,k}^{*}}{\longrightarrow}...$. Analogously, one can
define the exterior algebra $\Lambda ^{*}(\pi )$ of
\textit{differential forms on} $J^{\infty}(\pi )$ as the filtered exterior
algebra provided by the direct limit of the sequence of inclusions
$\Lambda ^{*}(M)\overset{\pi ^{*}}{\longrightarrow}\Lambda ^{*}\left (J^{0}(
\pi )\right )\overset{\pi _{1,0}^{*}}{\longrightarrow}...
\overset{\pi _{k,k-1}^{*}}{\longrightarrow}\Lambda ^{*}\left (J^{k}(
\pi )\right )\overset{\pi _{k+1,k}^{*}}{\longrightarrow}...$. Thus, being
direct limits, any smooth function or form on an infinite jet space is
nothing but a smooth function or form on some finite order jet space. Therefore,
the \textit{exterior differential} $d$ naturally extends to differentiable
forms on $J^{\infty}(\pi )$. On the other hand, using
$\mathcal{F}(\pi )$, one can also think about \textit{vector fields} on
$J^{\infty}(\pi )$ as derivations of $\mathcal{F}(\pi )$. Throughout the
paper the $\mathcal{F}(\pi )$-module of vector fields on
$J^{\infty}(\pi )$ will be denoted by $\mathcal{D}(\pi )$. In canonical
coordinates these vector fields can be identified with formal series
$Z=\sum _{i}\alpha _{i}\partial _{x_{i}}+\sum _{\sigma}\sum _{j}
\beta _{\sigma}^{j}\partial _{u_{\sigma}^{j}}$, with
$\alpha _{i},\beta _{\sigma}^{j}\in \mathcal{F}(\pi )$. In particular,
one can say that a vector field $Z$ has \textit{filtration degree}
$r$ when it is the smallest $r\in \mathbb{N}$ such that
$Z(\mathcal{F}_{k}(\pi ))\subseteq \mathcal{F}_{k+r}(\pi )$,
$\forall k\in \mathbb{N}$, $k\geq 1$. Since in general vector fields on
$J^{\infty}(\pi )$ do not have an associated flow, a particularly important
case is that of vector fields with zero filtration degree, because any
such field $Z$ admits a flow that can be seen as an inverse limit of a
sequence of flows on finite order jet spaces. Also the
\textit{Lie derivative} of functions, vector fields or forms on
$J^{\infty}(\pi )$ can be defined in a completely algebraic way. For instance,
the Lie derivative of a function $f\in \mathcal{F}(\pi )$ along a vector
field $Z\in \mathcal{D}(\pi )$ is $L_{Z}(f):=Z(f)$, and the Lie derivative
of $Y\in \mathcal{D}(\pi )$ along $Z$ is
$L_{Z}Y:=\left [Z,Y\right ]=Z\circ Y-Y\circ Z$. Whereas, the Lie derivative
of a form $\omega \in \Lambda ^{*}(\pi )$ along $Z$ is defined as
$L_{Z}\omega :=i_{Z}(d\omega )+d(i_{Z}\omega )$, where $i_{Z}$ denotes
the \textit{insertion operator}
$i_{Z}:\Lambda ^{h}(\pi )\longrightarrow \Lambda ^{h-1}(\pi )$.

Jets spaces are naturally equipped with a tangent distribution which is
referred to as \textit{Cartan distribution}, or contact distribution. Indeed,
if $s$ is a (local) section of $\pi $, then the $k$\textsl{-th
order jet prolongation} $s^{(k)}$ of $s$ is the (local) section of
$\pi _{k}$ defined by $s^{(k)}(a)=[s]_{a}^{k}$, for any $a$ in the domain
of $s$. Then the \textit{Cartan distribution}
$\mathcal{C}^{k}(\pi )=\cup _{\theta \in J^{k}(\pi )}\mathcal{C}_{
\theta}^{k}(\pi )$ on the $k$-th order jet space $J^{k}(\pi )$ can be point-wise
defined by the spans $\mathcal{C}_{\theta}^{k}(\pi )$ of the tangent planes
at $\theta =[s]_{a}^{k}$ to the graphs of $k$-th order jet prolongations
$s'{}^{(k)}$ of sections $s'$ such that $[s']_{a}^{k}=[s]_{a}^{k}$. In
terms of canonical coordinates, the Cartan distribution
$\mathcal{C}^{k}(\pi )$ is described by the annihilator of the Pfaffian
system
$\{\omega _{\sigma}^{j}:\;0\leq |\sigma |\leq k-1,\;j=1,...,m\}$, with
$\omega _{\sigma}^{j}=du_{\sigma}^{j}-\sum _{i}u_{\sigma +1_{i}}^{j}dx_{i}$
denoting the so called Cartan forms. Dually, $\mathcal{C}^{k}(\pi )$ can
also be described as the distribution generated by the system of vector
fields
$\{\partial _{u_{\sigma}^{j}},\,D_{i}^{(k)}:\;|\sigma |=k,\;j=1,...,m,
\;i=1,...,n\}$, with
$D_{i}^{(k)}:=\partial _{x_{i}}+\sum _{|\sigma |\leq k-1}{
\displaystyle u_{\sigma +1_{i}}^{j}\partial _{u_{\sigma}^{j}}}$ denoting
the $k$-th order \textit{truncated total derivative}s.

Then by taking the inverse limit of the sequence of surjections
$\mathcal{C}^{1}(\pi )\overset{\pi _{2,1*}}{\longleftarrow}
\mathcal{C}^{2}(\pi )\overset{\pi _{3,2*}}{\longleftarrow}...
\overset{\pi _{k,k-1*}}{\longleftarrow}\mathcal{C}^{k}(\pi )
\overset{\pi _{k+1,k*}}{\longleftarrow}...$, one defines the
\textit{Cartan distribution} $\mathcal{C}(\pi )$ of
$J^{\infty}(\pi )$. One can see $\mathcal{C}(\pi )$ as the distribution
annihilating all Cartan forms
$\{\omega _{\sigma}^{j}=du_{\sigma}^{j}-\sum _{i}u_{\sigma +1_{i}}^{j}dx_{i}:
\;|\sigma |\geq 0,\;j=1,...,m\}$, or equivalently the distribution generated
by all \textit{total derivative}s
\begin{equation*}
D_{i}:=\partial _{x_{i}}+{\displaystyle \sum _{|\rho |\geq 0}\sum _{j=1}^{m}u_{
\rho +1_{i}}^{j}\partial _{u_{\rho}^{j}}},\qquad i=1,...,n.
\end{equation*}

It is easy to show that integral manifolds $\Sigma $ of Cartan distributions
with \textit{independence condition}
$\text{$\Omega$=}dx_{1}\wedge ...\wedge dx_{n}\neq 0$ (i.e.,
such that $\Omega |_{\Sigma}\neq 0$) are prolongations of sections of
$\pi $ (see for instance \cite{Stormark}).

Geometrically, the solutions of a $k$-th order differential equation (or
system) $\mathcal{E}=\{\mathbf{F}=\mathbf{0}\}\subset J^{k}(\pi )$ are
just sections $s$ of $\pi $ whose $k$-order prolongations $j_{k}(s)$ lay
on $\mathcal{E}$. Under regularity conditions for $\mathbf{F}$, if
$\mathcal{E}$ is a submanifold of $J^{k}(\pi )$, it is naturally equipped
with the \textit{induced Cartan distribution}
$\mathcal{C}^{k}(\mathcal{E}):=\mathcal{C}^{k}(\pi )\cap T\mathcal{E}$
and the solutions of $\mathcal{E}$ are sections of $\pi $ whose $k$-order
prolongation are integral manifolds of
$\mathcal{C}^{k}(\mathcal{E})$.

On the other hand, under further regularity conditions for
$\mathcal{E}$, for any $r\in \mathbb{N}$ one may also consider the
$r$-\textit{th order prolongation}
$\mathcal{E}^{(r)}=\{D_{\mu}\mathbf{F}=\mathbf{0}:\;0\leq |\mu |\leq r
\}$, where
$D_{\mu}:=\left (D_{1}\right )^{\mu _{1}}\circ ...\circ \left (D_{n}
\right )^{\mu _{n}}$. One says that $\mathcal{E}$ is
\textit{formally integrable} if and only if for any $r\in \mathbb{N}$ the
prolongations $\mathcal{E}^{(r)}$ are submanifolds of
$J^{k+r}(\pi )$ and the maps
$\pi _{k+r+1,k+r}:\mathcal{E}^{(r+1)}\rightarrow \mathcal{E}^{(r)}$ are
smooth fiber bundles.

Then the \textit{infinite prolongation} $\mathcal{E}^{(\infty )}$, of a
formally integrable equation $\mathcal{E}$, is defined as the inverse limit
of the sequence of fiber bundles
$\pi _{k+r+1,k+r}:\mathcal{E}^{(r+1)}\rightarrow \mathcal{E}^{(r)}$. Since
each $\mathcal{E}^{(r)}$ is naturally equipped with the induced Cartan
distribution $\mathcal{C}^{k+r}(\mathcal{E}^{(r)})$, also
$\mathcal{E}^{(\infty )}$ is equipped with an induced
\textit{Cartan distribution} $\mathcal{C}(\mathcal{E})$ defined by the inverse
limit of the sequence of surjections
$\pi _{k+r+1,k+r\,*}:\mathcal{C}^{k+r+1}\left (\mathcal{E}^{(r+1)}
\right )\rightarrow \mathcal{C}^{k+r}\left (\mathcal{E}^{(r)}\right )$.
One has that
$\mathcal{E}^{(\infty )}=\{D_{\mu}F=0:\;|\mu |\geq 0\}\subset J^{
\infty}(\pi )$ and
$\mathcal{C}(\mathcal{E})=\left \langle \bar{D}_{1},...,\bar{D}_{n}
\right \rangle =\text{Ann}\left \{ \bar{\omega}_{\sigma}^{j}:\,j=1,...,m,
\quad |\sigma |\geq 0\right \} $, where $\bar{D}_{i}$ and
$\bar{\omega}_{\sigma}^{j}$ are the restrictions to
$\mathcal{E}^{(\infty )}$ of the total derivatives and Cartan forms, respectively.
Moreover, by restricting $\Lambda ^{*}(\pi )$ to
$\mathcal{E}^{(\infty )}$ one gets the exterior algebra
$\Lambda ^{*}(\mathcal{E})$ of \textit{differential forms on}
$\mathcal{E}^{(\infty )}$ and in particular the algebra
$\mathcal{F}(\mathcal{E})$ of \textit{smooth functions on}
$\mathcal{E}^{(\infty )}$.

Now, since $\mathcal{C}(\pi )$ is totally horizontal with respect to the
mapping $\pi _{\infty}:J^{\infty}(\pi )\rightarrow M$, the tangent bundle
$\mathcal{T}\left (\pi \right )$ on $J^{\infty}(\pi )$ decomposes as
$\mathcal{T}\left (\pi \right )=\mathcal{V}(\pi )\oplus \mathcal{C}(
\pi )$, where
$\mathcal{V}(\pi ):=Ker\left (\pi _{\infty}\right )_{*}$. Dually one has
$\Lambda ^{1}(\pi )=\Lambda ^{(1,0)}(\pi )\oplus \Lambda ^{(0,1)}(
\pi )$, where
$\Lambda ^{(1,0)}(\pi ):=Ann\left (\mathcal{V}(\pi )\right )$ and
$\Lambda ^{(0,1)}(\mathcal{\pi}):=Ann\left (\mathcal{C}(\pi )\right )$
are the $\mathcal{F}(\pi )$-modules of horizontal and vertical $1$-forms
on $J^{\infty}(\pi )$ locally generated by $\{dx_{i}\}$ and Cartan forms
$\left \{ \omega _{\sigma}^{j}\right \} $, respectively. More in general,
by considering
$\Lambda ^{(p,q)}(\pi )=\left (\bigwedge ^{p}\Lambda ^{(1,0)}(\pi )
\right )\bigwedge \left (\bigwedge ^{q}\Lambda ^{(0,1)}(\pi )\right )$,
the $\mathcal{F}(\pi )$-module of $r$-forms on $J^{\infty}(\pi )$ decomposes
as
$\Lambda ^{r}\left (\pi \right )=\bigoplus _{p+q=r}\Lambda ^{(p,q)}
\left (\pi \right )$. By definition we set
$\mathcal{F}(\pi )=\Lambda ^{(0,0)}(\pi )$. Accordingly, the exterior differential
splits into the sum $d=d_{H}+d_{V}$ of the horizontal and vertical differentials
$d_{H}:\Lambda ^{(p,q)}(\mathcal{\pi})\rightarrow \Lambda ^{(p+1,q)}(
\pi )$ and
$d_{V}:\Lambda ^{(p,q)}(\pi )\rightarrow \Lambda ^{(p,q+1)}(\pi )$, satisfying
$d_{H}^{2}=d_{V}^{2}=0$ and $d_{H}\circ d_{V}=-d_{V}\circ d_{H}$. In coordinates,
these differentials can be easily computed since they act as graded derivations
on $\Lambda ^{*}(\pi )$ and for any function
$f\in \mathcal{F}(\pi )$ one has $d_{H}f:=\sum _{i}D_{i}f\,dx_{i}$ and
$d_{V}f:=\sum _{\sigma}\sum _{j}
\frac{\partial f}{\partial u_{\sigma}^{j}}\omega _{\sigma}^{j}$.

Analogously, given a formally integrable equation $\mathcal{E}$, since
$\mathcal{C}(\mathcal{E})$ is totally horizontal with respect to the mapping
$\bar{\pi}_{\infty}:\mathcal{E}^{(\infty )}\rightarrow M$, the tangent
bundle $\mathcal{T}\left (\mathcal{E}\right )$ on
$\mathcal{E}^{(\infty )}$ decomposes as
$\mathcal{T}\left (\mathcal{E}\right )=\mathcal{V}(\mathcal{E})
\oplus \mathcal{C}(\mathcal{E})$, where
$\mathcal{V}(\mathcal{E}):=Ker\left (\bar{\pi}_{\infty *}\right )$ is the
vertical bundle on $\mathcal{E}^{(\infty )}$. Hence the
$\mathcal{F}(\mathcal{E})$-modules $\Lambda ^{(1,0)}(\mathcal{E})$ and
$\Lambda ^{(0,1)}(\mathcal{E})$ of horizontal and vertical $1$-forms on
$\mathcal{E}^{(\infty )}$, locally generated by $\{dx^{i}\}$ and restricted
Cartan forms
$\left \{ \bar{\omega}_{\sigma}^{j}:=\left .\omega _{\sigma}^{j}
\right |_{\mathcal{E}^{(\infty )}}\right \} $, can be used to decompose
the $\mathcal{F}(\mathcal{E})$-module of $r$-forms on
$\mathcal{E}^{(\infty )}$ as
$\Lambda ^{r}\left (\mathcal{E}^{(\infty )}\right )=\bigoplus _{p+q=r}
\Lambda ^{(p,q)}\left (\mathcal{E}\right )$; in particular one has
$\mathcal{F}(\mathcal{E})=\Lambda ^{(0,0)}(\mathcal{E})$. Accordingly,
on $\mathcal{E}^{(\infty )}$ the exterior differential $d$ (still denoted
by $d$, for ease of notation) splits into the sum
$d=\bar{d}_{H}+\bar{d}_{V}$ of the horizontal and vertical differentials
$\bar{d}_{H}:\Lambda ^{(p,q)}(\mathcal{E})\rightarrow \Lambda ^{(p+1,q)}(
\mathcal{E})$ and
$\bar{d}_{V}:\Lambda ^{(p,q)}(\mathcal{E})\rightarrow \Lambda ^{(p,q+1)}(
\mathcal{E})$, which satisfy $\bar{d}_{H}^{2}=\bar{d}_{V}^{2}=0$ and
$\bar{d}_{H}\circ \bar{d}_{V}=-\bar{d}_{V}\circ \bar{d}_{H}$. Also in this
case, these differentials can be easily computed in coordinates, since
they act as graded derivations on
$\Lambda ^{*}(\mathcal{E}^{(\infty )})$ and for any function
$f\in \mathcal{F}(\mathcal{E})$ one has that
$\bar{d}_{H}f:=\sum _{i}\bar{D}_{i}f\,dx_{i}$ and
$\bar{d}_{V}f:=\sum _{\sigma}\sum _{j}
\frac{\partial f}{\partial u_{\sigma}^{j}}\bar{\omega}_{\sigma}^{j}$, where
$\bar{D}_{i}$ denote the total derivatives restricted to
$\mathcal{E}^{(\infty )}$. For notational convenience, from now on we will
denote $\Lambda ^{(p,0)}(\pi )$ and $\Lambda ^{(p,0)}(\mathcal{E})$ by
$\bar{\Lambda}^{p}(\pi )$ and $\bar{\Lambda}^{p}(\mathcal{E})$, respectively.

Starting from these algebraic structures in the algebra of differential
forms on $\mathcal{E}$, new constructions and new notions can be introduced
further. For instance, using the horizontal differentials
$\bar{d}_{H}:\overline{\Lambda}^{p}(\mathcal{E})\rightarrow
\overline{\Lambda}^{p+1}(\mathcal{E})$, one can consider the notion of
\textit{conservation law} for $\mathcal{E}$, which is a closed horizontal
form $\mu \in \overline{\Lambda}^{n-1}(\mathcal{E})$, i.e., a horizontal
$(n-1)$-form on $\mathcal{E}^{(\infty )}$ satisfying
$\bar{d}_{H}\mu =0$. On the other hand, since an exact horizontal form
is trivially closed, one can naturally limit itself to consider closed
horizontal $(n-1)$-forms up to exact horizontal forms. Thus, a conservation
law can also be understood as a cohomology class
$[\mu ]\in \overline{H}^{n-1}(\mathcal{E})$, i.e.,
$[\mu ]=\left \{ \mu +\bar{d}_{H}\rho :\,\rho \in \overline{\Lambda}^{n-2}(
\mathcal{E})\right \} $. For instance, when $n=2$, a conservation law is
locally described by an horizontal $1$-form $Adx_{1}+Bdx_{2}$ on
$\mathcal{E}^{(\infty )}$ such that $\bar{D}_{2}A-\bar{D}_{1}B=0$, i.e.,
$D_{2}A-D_{1}B=0$ on $\mathcal{E}^{(\infty )}$.

Moreover, given a matrix Lie algebra $\mathfrak{g}$, one may consider the
exterior algebras $\mathfrak{g}\otimes \Lambda ^{*}(\pi )$ and
$\mathfrak{g}\otimes \Lambda ^{*}(\mathcal{E})$ of $\mathfrak{g}$-valued
forms on $J^{\infty}(\pi )$ and $\mathcal{E}^{(\infty )}$, respectively.
Also, one can consider the graded algebra of $\mathfrak{g}$-valued horizontal
forms on $J^{\infty}(\pi )$ and $\mathcal{E}^{(\infty )}$, that will be
denoted here by
$\mathfrak{g}\otimes \bar{\Lambda}^{*}\left (\pi \right )=\bigoplus _{p}
\mathfrak{g}\otimes \bar{\Lambda}^{p}(\pi )$ and
$\mathfrak{g}\otimes \bar{\Lambda}^{*}\left ({\mathcal{E}}\right )=
\bigoplus _{p}\mathfrak{g}\otimes \bar{\Lambda}^{p}\left ({\mathcal{E}}
\right )$, respectively. By definition, $\mathfrak{g}$-valued horizontal
$p$-forms on $J^{\infty}(\pi )$ (resp., $\mathcal{E}^{(\infty )}$) are
generated by $\mathfrak{g}$-valued $p$-forms $A\omega $, with $A$ a
$\mathfrak{g}$-valued functions on $J^{\infty}(\pi )$ (resp.,
$\mathcal{E}^{(\infty )}$). Then, one may define a bilinear product
$[\;,\;]$ by linearly extending the product
$\left [A_{1}\omega _{1},A_{2}\omega _{2}\right ]:=\left [A_{1},A_{2}
\right ]\omega _{1}\wedge \omega _{2}$, between generators. One can check
that $[\;,\;]$ satisfies the following properties: (i)
$[\rho ,\sigma ]=-(-1)^{rs}[\sigma ,\rho ]$; (ii)
$(-1)^{rt}[\rho ,[\sigma ,\tau ]]+(-1)^{sr}[\sigma ,[\tau ,\rho ]]+(-1)^{ts}[
\tau ,[\rho ,\sigma ]]=0$; (iii)
$d_{H}[\rho ,\sigma ]=[d_{H}\rho ,\sigma ]+(-1)^{r}[\rho ,d_{H}
\sigma ]$, analogously for $\bar{d}_{H}$. Where $r,s$ and $t$ are the degrees
of the $\mathfrak{g}$-valued horizontal forms $\rho ,\sigma $ and
$\tau $, respectively. Also, one may define an exterior product
$\wedge $ on
$\mathfrak{g}\otimes \bar{\Lambda}^{*}\left (\pi \right )$, or
$\mathfrak{g}\otimes \bar{\Lambda}^{*}\left (\mathcal{E}\right )$, by linearly
extending the product
$A_{1}\,\omega _{1}\wedge A_{2}\,\omega _{2}=A_{1}A_{2}\,\omega _{1}
\wedge \omega _{2}$.

This allows one to introduce the notion of
\textit{zero-curvature representation} (ZCR) of ${\mathcal{E}}$, that is a
$\mathfrak{g}$-valued non-vanishing $1$-form
$\alpha \in \mathfrak{g}\otimes \bar{\Lambda}^{1}\left ({\mathcal{E}}
\right )$ such that
%
%e1 #&#
\begin{equation}
\bar{d}_{H}\alpha -\frac{1}{2}\left [\alpha ,\alpha \right ]=0.
\label{eq:ZCR}
%%LEAP%%%\label{eq1}
\end{equation}
This is a very important notion in the theory of integrable equations with
$2$ independent variables
\cite{Ablowitz-Segur,DiegoLuiz,Dun,Taktajan,Zakharov}. In such a case,
by taking $\alpha =A\,dx_{1}+Bdx_{2}$, condition (\ref{eq:ZCR}) reads
$D_{2}A-D_{1}B+[A,B]=0$ on $\mathcal{E}^{(\infty )}$.

Here we notice that (\ref{eq:ZCR}) can also be equivalently written as
$\bar{d}_{H}\alpha -\alpha \wedge \alpha =0$. Moreover, since
$\mathcal{E}^{(\infty )}\subset J^{\infty}(\pi )$, any element of
$\mathfrak{g}\otimes \bar{\Lambda}^{1}(\mathcal{E})$ can be identified
with an element of $\mathfrak{g}\otimes \bar{\Lambda}^{1}(\pi )$. Hence,
in the outer geometry, (\ref{eq:ZCR}) can also be rewritten as
$d_{H}\alpha -\frac{1}{2}\left [\alpha ,\alpha \right ]=0\;\mbox{mod}
\,\mathcal{E}^{(\infty )}$. In particular, when
$\mathcal{E}=\{F^{j}=0,\,j=1,...,h\}$, under regularity assumptions (i.e.,
if any prolongation $\mathcal{E}^{(h )}$, $h\geq 0$, is totally non-degenerating
\cite{olver}) equation (\ref{eq:ZCR}) can also be rewritten in ``characteristic''
form
$d_{H}\alpha -\frac{1}{2}\left [\alpha ,\alpha \right ]=\sum D_{
\sigma}(F^{j})\gamma _{j}^{\sigma}$, where
$\gamma _{j}^{\sigma}\in \mathfrak{g}\otimes \bar{\Lambda}^{2}(\pi )$.
Hence in general (\ref{eq:ZCR}) holds modulo differential consequences
of $\mathcal{E}$, thus (\ref{eq:ZCR}) is not equivalent to
$\{F^{j}=0,\,j=1,...,h\}$.

%s2.2 #&#
\subsection{$\mathcal{C}$-morphisms (or Lie-B\"acklund maps)}
%%LEAP%%%\label{sec2.2}
\label{subsec:morphisms}

In this subsection $J^{\infty}(\pi )$ and $J^{\infty}(\pi ')$ will denote
the infinite jet spaces of sections of two fiber bundles
$\pi :E\rightarrow M$ and $\pi ':E'\rightarrow M'$ with
\textit{canonical coordinates} $\{x_{i},u_{\sigma}^{j}\}$ and
$\{x'_{i},u'{}_{\sigma}^{j}\}$, respectively. In particular we assume that
$n=\dim M=\dim M'$, since in the paper we are mainly concerned with this
case.

A $\mathcal{C}$-morphism, also referred to as Lie-B\"{a}cklund transformations
\cite{And-Ibr,KLV}, is a smooth map
$\mathcal{B}:J^{\infty}(\pi )\rightarrow J^{\infty}(\pi ')$, such that
$\mathcal{B}_{*}\mathcal{C}_{\theta}(\pi )\subseteq \mathcal{C}_{
\mathcal{B}(\theta )}(\pi ')$, for any $\theta \in J^{\infty}(\pi )$, i.e.,
%
%e2 #&#
\begin{equation}
\mathcal{B}_{*}\mathcal{C}(\pi )\subseteq \mathcal{C}(\pi ').
\label{C-invar}
%%LEAP%%%\label{eq2}
\end{equation}
Notice that
$\mathcal{B}_{*}\mathcal{C}(\pi )\subseteq \mathcal{C}(\pi ')$ is equivalent
to
$\mathcal{B}^{*}\left (\mathcal{I}_{\mathcal{C}(\pi ')}\right )
\subseteq \mathcal{I}_{\mathcal{C}(\pi )}$, where
$\mathcal{I}_{\mathcal{C}(\pi )}$ is the EDS generated by Cartan (or multi-contact)
forms
$\omega _{\sigma}^{j}=du_{\sigma}^{j}-\sum _{i}u_{\sigma +1_{i}}^{j}dx_{i}$
on $J^{\infty}(\pi )$, analogously
$\mathcal{I}_{\mathcal{C}(\pi ')}$.%

Notice that in general these transformations are generalizations of Lie
transformations (i.e., point and contact transformations) that need not
to be diffeomorphisms. Indeed, as shown by B\"{a}cklund
\cite{Backlund2} (see also \cite{And-Ibr, Backlund1,Lie}), only Lie transformations are
invertible. Moreover, by
$\mathcal{B}^{*}\left (\mathcal{I}_{\mathcal{C}(\pi ')}\right )
\subseteq \mathcal{I}_{\mathcal{C}(\pi )}$ it follows that under such a
transformation the image of an integral manifolds of Cartan distribution
is still an integral manifold.

In canonical coordinates, a $\mathcal{C}$-morphism (Lie-B\"{a}cklund transformation)
has the form
%
%e3 #&#
\begin{equation}
\left \{
\begin{array}{l}
x'_{i}=\xi ^{i}\left (x,u^{(k)}\right ),
\vspace{10pt}
\\
u'{}^{j}=\nu ^{j}\left (x,u^{(k)}\right ),
\vspace{10pt}
\\
\vdots
\\
u'{}_{\sigma}^{j}=\nu _{\sigma}^{j}\left (x,u^{(k+|\sigma |)}\right ),
\qquad |\sigma |\geq 0
\end{array}
\right .
\label{eq:BL_transf}
%%LEAP%%%\label{eq3}
\end{equation}
where in view of
$\mathcal{B}^{*}\left (\mathcal{I}_{\mathcal{C}(\pi ')}\right )
\subseteq \mathcal{I}_{\mathcal{C}(\pi )}$ the functions $\xi ^{i}$ and
$\nu _{\sigma}^{j}$ are such that
$d\nu _{\sigma}^{j}-\sum _{i}\nu _{\sigma +1_{i}}^{j}d\xi ^{i}=0
\mod{\mathcal{I}_{\mathcal{C}}}$.

Like in the finite order case, integral manifolds $\Sigma $ of
$\mathcal{C}(\pi )$ with independence condition
$\text{$\Omega$=}dx_{1}\wedge ...\wedge dx_{n}\neq 0$ (i.e.,
such that $\Omega |_{\Sigma}\neq 0$) are $\infty $-th order prolongations
of sections of $\pi $. One has the following
%
%p1 #&#
\begin{prop}
%%LEAP%%%\label{prop1}
\label{morp_sends_prol}%
A $\mathcal{C}$-morphism maps an integral manifold $\Sigma $ of
$\mathcal{C}(\pi )$ with independence condition
$dx_{1}\wedge ...\wedge dx_{n}\neq 0$ to an integral manifold
$\Sigma '$ of $\mathcal{C}(\pi ')$ with independence condition
$dx'_{1}\wedge ...\wedge dx'_{n}\neq 0$ whenever the regularity assumption
%
%e4 #&#
\begin{equation}
\det \left (D_{s}\xi ^{i}\right )\neq 0
\label{eq:xi_nondeg}
%%LEAP%%%\label{eq4}
\end{equation}
is satisfied on $\Sigma $. \end{prop}

\begin{proof}
Indeed, being $\Sigma $ the infinite prolongation of a section
$\mathfrak{s}$ locally described as $\mathfrak{s}(x)=(x,u(x))$, one has
$\left (\mathcal{B}\circ \mathfrak{s}^{(\infty )}\right )^{*}(\omega '{}_{
\sigma}^{j})=
\allowbreak
\left (\mathfrak{s}^{(\infty )}\right )^{*}\circ \mathcal{B}^{*}
\left (\omega '{}_{\sigma}^{j}\right )
\allowbreak
=0\mod{\mathcal{I}_{\mathcal{C}(\pi )}}$, since the graph of
$\mathfrak{s}^{(\infty )}$ is an integral manifold of
$\mathcal{I}_{\mathcal{C}(\pi )}$. Thus the graph $\Sigma '$ of
$\mathcal{B}\circ \mathfrak{s}^{(\infty )}$ describes another integral manifold
of $\mathcal{I}_{\mathcal{C}(\pi ')}$. On the other hand whenever
$\det \left (D_{s}\xi ^{i}\right )\neq 0$ along the graph of
$\mathfrak{s}^{(k)}$ (i.e.,
$\mathfrak{s}^{(k)}{}^{*}\left (\det \left (D_{s}\xi ^{i}\right )
\right )\neq 0$), the first $n$-equations
$\{x'_{i}=\xi ^{i}(x,u^{(k)}(x)),\,i=1,...,n\}$ of t{(\ref{eq:BL_transf})}
restricted to $\mathfrak{s}^{(k)}$ can be locally solved with respect to
$(x_{1},...,x_{n})$. Thus, there exists a local diffeomorphism
$x'=\mathcal{B}_{\mathfrak{s}}(x)=\left (\xi \circ \mathfrak{s}^{(k)}
\right )(x)$ that allows one to pass from the parametrization
$\left (\mathcal{B}\circ \mathfrak{s}^{(\infty )}\right )(x)$ of
$\Sigma ' $ to the new parametrization $\left (\mathcal{B}
\circ \mathfrak{s}^{(\infty )}\circ \mathcal{B}_{\mathfrak{s}}^{-1}
\right )(x')$.

Then, since
\begin{equation*}
\begin{array}{l@{\,}l}
\left (\mathcal{B}\circ \mathfrak{s}^{(\infty )}\circ \mathcal{B}_{
\mathfrak{s}}^{-1}\right )^{*}\left (dx'_{1}\wedge ...\wedge dx'_{n}
\right ) & =\left(\left (\mathcal{B}_{\mathfrak{s}}^{-1}\right )^{*}\circ s^{(
\infty )\,*}\circ \mathcal{B}^{*}\right)\left (dx'_{1}\wedge ...\wedge dx'_{n}
\right )
\vspace{5pt}
\\
& =\left(\left (\mathcal{B}_{\mathfrak{s}}^{-1}\right )^{*}\circ s^{(
\infty )\,*}\right) \left (d\left (\xi ^{1}(x,u^{(k)})\right )\wedge ...
\wedge d\left (\xi ^{n}\left (x,u^{(k)}\right )\right )\right )
\vspace{5pt}
\\
& =\left (\mathcal{B}_{\mathfrak{s}}^{-1}\right )^{*} \left (d
\left (\xi ^{1}\circ s^{(k)}(x)\right )\wedge ...\wedge d\left (\xi ^{n}
\circ s^{(k)}(x)\right )\right )
\vspace{5pt}
\\
& =dx'_{1}\wedge ...\wedge dx'_{n},
\end{array}
\end{equation*}
it turns out that
$\left (\mathcal{B}\circ \mathfrak{s}^{(\infty )}\circ \mathcal{B}_{
\mathfrak{s}}^{-1}\right )(x')$ is the infinite prolongation of the section
of $\pi '$
\begin{equation*}
\mathfrak{s}'(x')=\left (\mathcal{B}^{(0)}\circ \mathfrak{s}^{(k)}
\circ \mathcal{B}_{\mathfrak{s}}^{-1}\right )(x')=\left (x',\nu
\left (x,u^{(k)}\left (\mathcal{B}_{\mathfrak{s}}^{-1}(x')\right )
\right )\right ).\qedhere
\end{equation*}
\end{proof}
%
%r2 #&#
\begin{rem}
%%LEAP%%%\label{rem2}
\label{rem:morph_prol}
t{Proposition~\ref{morp_sends_prol}} entails that a $\mathcal{C}$-morphism
sends infinite prolongations of (local) sections to infinite prolongations
of (local) sections, whenever t{(\ref{eq:xi_nondeg})} is satisfied. Indeed,
it is noteworthy to stress that the action of such a $\mathcal{C}$-morphism
$\mathcal{B}$ on the infinite prolongation
$\mathfrak{s}^{(\infty )}(x)$ of a (local) section $\mathfrak{s}(x)$ is
not necessarily defined for any $x$, because
$\det \left (D_{s}\xi ^{i}\right )(\mathfrak{s}^{(\infty )}(x))$ could
be zero at some points. Also, we stress that t{(\ref{eq:xi_nondeg})} entails
that the push-forward
$\mathcal{B}_{*\theta}:\mathcal{C}_{\theta}(\pi )\rightarrow
\mathcal{C}_{\mathcal{B}(\theta )}(\pi ')$ establishes an isomorphism for
any $\theta \in J^{\infty}(\pi )$.
\end{rem}

We will refer to a $\mathcal{C}$-morphism t{(\ref{eq:BL_transf})} satisfying
regularity assumption t{(\ref{eq:xi_nondeg})} as a
\textit{regular $\mathcal{C}$-morphism}. One also have the following
%
%p3 #&#
\begin{prop}
%%LEAP%%%\label{prop3}
\label{prop:3}%
A regular $\mathcal{C}$-morphism t{(\ref{eq:BL_transf})} is completely determined
by its lower components
\begin{equation*}
\left \{
\begin{array}{l}
x'_{i}=\xi ^{i}\left (x,u^{(k)}\right ),
\vspace{10pt}
\\
u'{}^{j}=\nu ^{j}\left (x,u^{(k)}\right ),
\end{array}
\right .
\end{equation*}
through the prolongation formulas
%
%e5 #&#
\begin{equation}
\left (
\begin{array}{l@{\quad}l@{\quad}l}
\nu _{\sigma +1_{1}}^{1} & \cdots & \nu _{\sigma +1_{1}}^{m}
\\
\vdots & & \vdots
\\
\nu _{\sigma +1_{n}}^{1} & \cdots & \nu _{\sigma +1_{n}}^{m}
\end{array}
\right )=\left (
\begin{array}{l@{\quad}l@{\quad}l}
D_{1}\xi ^{1} & \cdots & D_{1}\xi ^{n}
\\
\vdots & & \vdots
\\
D_{n}\xi ^{1} & \cdots & D_{n}\xi ^{n}
\end{array}
\right )^{-1}\cdot \left (
\begin{array}{l@{\quad}l@{\quad}l}
D_{1}\nu _{\sigma}^{1} & \cdots & D_{1}\nu _{\sigma}^{m}
\\
\vdots & & \vdots
\\
D_{n}\nu _{\sigma}^{1} & \cdots & D_{n}\nu _{\sigma}^{m}
\end{array}
\right ).
\label{Prol_mor}
%%LEAP%%%\label{eq5}
\end{equation}
\end{prop}

\begin{proof}
Indeed, in view of the decomposition $d=d_{H}+d_{V}$, one readily gets
that
\begin{equation*}
\begin{array}{l@{\,}l}
\mathcal{B}^{*}\left (\omega '{}_{\sigma}^{j}\right ) & =\mathcal{B}^{*}
\left (du'{}_{\sigma}^{j}-\sum _{i}u'{}_{\sigma +1_{i}}^{j}dx'_{i}
\right )=d\nu _{\sigma}^{j}-\sum _{i}\nu _{\sigma +1_{i}}^{j}d\xi ^{i}
\vspace{5pt}
\\
& =d_{H}\nu _{\sigma}^{j}-\sum _{i}\nu _{\sigma +1_{i}}^{j}d_{H}\xi ^{i}+d_{V}
\nu _{\sigma}^{j}-\sum _{i}\nu _{\sigma +1_{i}}^{j}d_{V}\xi ^{i}
\vspace{5pt}
\\
& =\sum _{s}\left (D_{s}\nu _{\sigma}^{j}-\sum _{i}\nu _{\sigma +1_{i}}^{j}D_{s}
\xi ^{i}\right )dx_{s}\mod{\mathcal{I}_{\mathcal{C}(\pi )},}
\end{array}
\end{equation*}
in view of the identity
$d_{V}\nu _{\sigma}^{j}-\sum _{i}\nu _{\sigma +1_{i}}^{j}d_{V}\xi ^{i}=0
\mod{\mathcal{I}_{\mathcal{C}(\pi )}}$. Thus
$\mathcal{B}^{*}\left (\mathcal{I}_{\mathcal{C}(\pi ')}\right )
\subseteq \mathcal{I}_{\mathcal{C}(\pi )}$ if and only if
$D_{s}\nu _{\sigma}^{j}-\sum _{i}\nu _{\sigma +1_{i}}^{j}D_{s}\xi ^{i}=0$,
that is
\begin{equation*}
\left (
\begin{array}{l@{\quad}l@{\quad}l}
D_{1}\xi ^{1} & \cdots & D_{1}\xi ^{n}
\\
\vdots & & \vdots
\\
D_{n}\xi ^{1} & \cdots & D_{n}\xi ^{n}
\end{array}
\right )\cdot \left (
\begin{array}{l@{\quad}l@{\quad}l}
\nu _{\sigma +1_{1}}^{1} & \cdots & \nu _{\sigma +1_{1}}^{m}
\\
\vdots & & \vdots
\\
\nu _{\sigma +1_{n}}^{1} & \cdots & \nu _{\sigma +1_{n}}^{m}
\end{array}
\right )=\left (
\begin{array}{l@{\quad}l@{\quad}l}
D_{1}\nu _{\sigma}^{1} & \cdots & D_{1}\nu _{\sigma}^{m}
\\
\vdots & & \vdots
\\
D_{n}\nu _{\sigma}^{1} & \cdots & D_{n}\nu _{\sigma}^{m}
\end{array}
\right ).
\end{equation*}
Hence, in view of t{(\ref{eq:xi_nondeg})} one gets t{(\ref{Prol_mor})}.
\end{proof}
  Moreover, in view of t{(\ref{C-invar})}, one also has the following
%
%l4 #&#
\begin{lem}
\label{lem4}
For any $\mathcal{C}$-morphism
$\mathcal{B}:J^{\infty}(\pi )\rightarrow J^{\infty}(\pi ')$ one has
%
%e6 #&#
\begin{equation}
D_{i}\circ \mathcal{B}^{*}=\sum _{j}\alpha _{ij}\mathcal{B}^{*}\circ D_{j}',
\label{eq:morf_com1}
%%LEAP%%%\label{eq6}
\end{equation}
for some smooth functions $\alpha _{ij}$ on $J^{\infty}(\pi )$. In particular,
if $\mathcal{B}$ is regular, one has
$\det \left (\alpha _{ij}\right )\neq 0$ and
%
%e7 #&#
\begin{equation}
\mathcal{B}^{*}\circ D_{j}'=\sum _{i}\alpha ^{ij}D_{i}\circ
\mathcal{B}^{*},
\label{eq:morf_com2}
%%LEAP%%%\label{eq7}
\end{equation}
with $(\alpha ^{ij})=(\alpha _{ij})^{-1}$.
\end{lem}

\begin{proof}
The invariance condition t{(\ref{C-invar})} is equivalent to say that for
any $\theta \in J^{\infty}(\pi )$ one has
$\mathcal{B}_{*}(\left .D_{i}\right |_{\theta})=\sum _{j}a_{ij}\left .D_{j}
\right |_{\mathcal{B}(\theta )}$, where $a_{ij}$ are some constants. Hence,
for any function $f$ on $J^{\infty}(\pi ')$ one has that
$D_{i}\left (\mathcal{B}^{*}(f)\right )(\theta )=\sum _{j}a_{ij}D_{j}f
\,(\mathcal{B}(\theta ))=\sum _{j}a_{ij}\mathcal{B}^{*}\left (D_{j}f
\right )(\theta )$. Then, since this holds for any $\theta $ one readily
gets t{(\ref{eq:morf_com1})}.

Now, applying t{(\ref{eq:morf_com1})} to $f=x'_{s}$ and using the fact that
$\mathcal{B}^{*}\left (x'_{s}\right )=\xi ^{s}\left (x,u^{(k)}\right )$
and
$\mathcal{B}^{*}\left (D_{j}'\,x'_{s}\right )=\mathcal{B}^{*}\delta _{js}=
\delta _{js}$, one readily gets $D_{i}\xi ^{s}=\alpha _{is}$ and hence
t{(\ref{eq:morf_com2})} readily follows by t{(\ref{eq:xi_nondeg})} and t{(\ref{eq:morf_com1})}.
\end{proof}
We will use the following
%
%d5 #&#
\begin{defn}
\label{defn5}
Let $\mathcal{E}^{(\infty )}$ and $\mathcal{E}'{}^{(\infty )}$ be the infinite
prolongations of two formally integrable equations
$\mathcal{E}\subset J^{k}(\pi )$ and
$\mathcal{E}'\subset J^{l}(\pi ')$, respectively. By a (regular)
$\mathcal{C}$-morphism from $\mathcal{E}$ to $\mathcal{E}'$ we mean a regular
$\mathcal{C}$-morphism
$\mathcal{B}:J^{\infty}(\pi )\rightarrow J^{\infty}(\pi ')$ such that
$\mathcal{B}(\mathcal{E}^{(\infty )})\subseteq \mathcal{E}'{}^{(
\infty )}$.
\end{defn}

In view of t{Proposition~\ref{morp_sends_prol}} and t{Remark~\ref{rem:morph_prol}}, a regular $\mathcal{C}$-morphism from
$\mathcal{E}$ to $\mathcal{E}'$ transforms solutions of
$\mathcal{E}$ to solutions of $\mathcal{E}'$. Moreover, one has the following
%
%p6 #&#
\begin{prop}
\label{prop6}
A regular $\mathcal{C}$-morphism
$\mathcal{B}:J^{\infty}(\pi )\rightarrow J^{\infty}(\pi ')$ is a
$\mathcal{C}$-morphism from $\mathcal{E}=\{\mathbf{F}=\mathbf{0}\}$ to
$\mathcal{E}'=\{\mathbf{F'}=\mathbf{0}\}$ if and only if
%
%e8 #&#
\begin{equation}
\mathcal{B}^{*}(\mathbf{F'})=0
\mod{\{D{}_{\sigma}\mathbf{F}=\mathbf{0}:|\sigma |\geq 0\}}.
\label{eq:Cmorf_E}
%%LEAP%%%\label{eq8}
\end{equation}
\end{prop}

\begin{proof}
Since
$\mathcal{E}^{(\infty )}=\{D{}_{\sigma}\mathbf{F}=\mathbf{0}:|\sigma |
\geq 0\}$ and
$\mathcal{E}'{}^{(\infty )}=\{D{}_{\rho}'\mathbf{F}'=\mathbf{0}:|
\rho |\geq 0\}$, the condition
$\mathcal{B}(\mathcal{E}^{(\infty )})\subseteq \mathcal{E}'{}^{(
\infty )}$ is equivalent to
$\mathcal{B}^{*}\left (D'_{\rho}\left (\mathbf{F}'\right )\right )=0
\mod{\mathcal{E}^{(\infty )}}$, for any multi-index $\rho $. Hence, in
view of t{(\ref{eq:morf_com2})},
$\mathcal{B}(\mathcal{E}^{(\infty )})\subseteq \mathcal{E}'{}^{(
\infty )}$ is equivalent to
%
%e9 #&#
\begin{equation}
D_{\mu}\mathcal{B}^{*}\left (\mathbf{F}'\right )=0
\mod{\{D{}_{\sigma}\mathbf{F}=\mathbf{0}:|\sigma |\geq 0\}},
\label{eq:aux1}
%%LEAP%%%\label{eq9}
\end{equation}
for any multi-index $\mu $. Thus t{(\ref{eq:aux1})} implies
$\mathcal{B}^{*}(\mathbf{F'})=0
\mod{\{D{}_{\sigma}\mathbf{F}=\mathbf{0}:|\sigma |\geq 0\}}$, for
$\mu =0$. Conversely, by totally deriving
$\mathcal{B}^{*}(\mathbf{F'})=0
\mod{\{D{}_{\sigma}\mathbf{F}=\mathbf{0}:|\sigma |\geq 0\}}$ one gets t{(\ref{eq:aux1})}.
\end{proof}
  It is noteworthy to remark that, being the order of
$\mathcal{B}^{*}(\mathbf{F'})$ always finite, the analysis of condition
t{(\ref{eq:Cmorf_E})} in practice only requires the consideration of a finite
number of differential consequences of $\mathbf{F=0}$.

%e7 #&#
\begin{example}
%%LEAP%%%\label{exmp7}
\label{exa:Cole-Hopf-I}%
An example of regular $\mathcal{C}$-morphism between two equations is the
Cole-Hopf transformation from the heat equation
$\mathcal{E}=\{u_{t}-u_{xx}=0\}$ to the Burgers equation
$\mathcal{E}'=\{u'_{t}-u'_{xx}-2u'u'_{x}=0\}$, where $u=u(x,t)$ and
$u'=u'(x,t)$. Indeed the Cole-Hopf transformation is defined by
$u'=u_{x}/u$, at the points where $u\neq 0$. In this case, by repeatedly
applying t{(\ref{Prol_mor})}, one can readily check that
\begin{equation*}
u'_{x}=\frac{u_{xx}}{u}-\left (\frac{u_{x}}{u}\right )^{2},\qquad u'_{xx}=
\frac{u_{xxx}}{u}-3\left (\frac{u_{x}}{u}\right )\,\left (
\frac{u_{xx}}{u}\right )+2\left (\frac{u_{x}}{u}\right )^{3},\quad
\ldots \;.
\end{equation*}
On the other hand
\begin{equation*}
u'_{t}=\frac{u_{xt}-u_{t}\,u'}{u}=\frac{u_{xxx}}{u}-\left (
\frac{u_{xx}}{u}\right )\,\left (\frac{u_{x}}{u}\right )
\mod{\mathcal{E}^{(1)}},
\end{equation*}
hence $u'_{t}-u'_{xx}-2u'u'_{x}=0\mod{\mathcal{E}^{(1)}}$.
\end{example}

%e8 #&#
\begin{example}
%%LEAP%%%\label{exmp8}
\label{exa:Miura-I}%
Another example of regular $\mathcal{C}$-morphism between two equations
is the Miura transformation from the mKdV equation
$\mathcal{E}=\{u_{t}-u_{xxx}+6u^{2}u_{x}=0\}$ to the KdV equation
$\mathcal{E}'=\{u'_{t}-u'_{xxx}-6u'u'_{x}=0\}$, where $u=u(x,t)$ and
$u'=u'(x,t)$. For ease of comparison, here and in t{Example~\ref{exa:Miura-II}}, it will be used the form of KdV adopted in
\cite{Sokolov} where the Miura transformation is defined by
$u'=u_{x}-u^{2}$. In this case, by repeatedly applying t{(\ref{Prol_mor})},
one can readily check that
\begin{eqnarray*}[ll]
u'_{x}=-2\,uu_{x}+u_{xx},\qquad u'_{xx}=-2\,u_{x}^{2}-2\,uu_{xx}+u_{xxx},
\\ u'_{xxx}=-6\,u_{x}u_{xx}-2\,uu_{xxx}+u_{xxxx},\quad \ldots \;.
\end{eqnarray*}
On the other hand
\begin{equation*}
u'_{t}=u_{xt}-2\,uu_{t}=u_{xxxx}-2\,uu_{xxx}-6\,u^{2}u_{xx}-12\,u^{3}u_{x}-12
\,uu_{x}^{2}\mod{\mathcal{E}^{(1)}},
\end{equation*}
hence $u'_{t}-u'_{xxx}-6u'u'_{x}=0\mod{\mathcal{E}^{(1)}}$.
\end{example}

%s2.3 #&#
\subsection{Symmetries and pseudosymmetries}
%%LEAP%%%\label{sec2.3}
\label{subsec:Symmetries-and-pseudosymmetries}

In this subsection, after reviewing the notion of symmetry of a differential
equation \cite{B_KrV,KLV,olver}, we will give an introduction to pseudosymmetries,
one of its possible generalizations proposed by Sokolov in the paper
\cite{Sokolov}, together with some of its key properties, which will be
particularly important in the forthcoming parts of the paper.

%s2.3.1 #&#
\subsubsection{Symmetries of differential equations}
%%LEAP%%%\label{sec2.3.1}
\label{subsec:Symmetries-of-differential}

Finite symmetries of a smooth distribution
$\mathcal{D}$ on a manifold $N$ are diffeomorphisms
$\psi :N\rightarrow N$ such that
$\psi _{*}\mathcal{D}\subseteq \mathcal{D}$. Analogously, from the infinitesimal
point of view, infinitesimal symmetries of a smooth distribution
$\mathcal{D}$ on a manifold $N$ are smooth vector fields
$Y$ on $N$ such that $L_{Y}\mathcal{D}\subseteq \mathcal{D}$. Hence, the
flow of an infinitesimal symmetry of $\mathcal{D}$ is a $1$-parameter local
group of finite symmetries of $\mathcal{D}$. If $\mathcal{D}$ is generated
by a system of vector fields, i.e., $\mathcal{D}=<X_{1},...,X_{n}>$, the
symmetry condition is equivalent to
$[Y,X_{i}]=\sum _{s}\alpha _{i}^{s}X_{s}$, for any
$i\in \{1,...,n\}$ and some smooth functions $\alpha _{i}^{s}$.

Now, given a $k$-th order equation (or system)
$\mathcal{E}=\{\mathbf{F}=\mathbf{0}\}\subset J^{k}(\pi )$, the
\textit{classical finite symmetries} of $\mathcal{E}$ are finite symmetries
of the distribution $\mathcal{C}^{k}(\pi )$ which leave invariant the submanifold
$\mathcal{E}$. Analogously,
\textit{classical infinitesimal symmetries} of $\mathcal{E}$ are vector
fields on $J^{k}(\pi )$ which are infinitesimal symmetries of
$\mathcal{C}^{k}(\pi )$ and are tangent to $\mathcal{E}$. A finite symmetry
$\psi $ is called \textit{projectable} if
$\psi ^{*}(C^{\infty}(M))\subseteq C^{\infty}(M)$. Analogously, an infinitesimal
symmetry $X$ is called projectable if
$X(C^{\infty}(M))\subseteq C^{\infty}(M)$.

It can be seen (see \cite{B_KrV} for details) that infinitesimal symmetries
$X$ of $\mathcal{C}^{k}(\pi )$ are completely described by a
\textit{generating function} $\varphi =(\varphi ^{j})$, defined as
$\varphi ^{j}:=X\,\lrcorner \,\omega _{0}^{j}$ (where
$\omega _{0}^{j}=du^{j}-\sum _{i}u_{i}^{j}dx_{i}$). This function is always
of first order (i.e., only depend on $x_{i},u^{j},u_{i}^{j}$) and, whenever
$m>1$, it is always of the form
$\varphi ^{j}=\sum _{i}a_{i}(x,u)u_{i}^{j}+b^{j}(x,u)$. In particular,
in canonical coordinates $\{x_{i},u_{\sigma}^{j}\}$, one has
%
%e10 #&#
\begin{equation}
X=-\sum _{i}\frac{\partial \varphi ^{s}}{\partial u_{i}^{s}}D_{i}^{(k)}+
\sum _{|\sigma |=0}^{k-1}\,\sum _{j=1}^{m}D_{\sigma}^{(k)}\varphi ^{j}
\,\partial _{u_{\sigma}^{j}},
\label{eq:simm_m_arb}
%%LEAP%%%\label{eq10}
\end{equation}
where
$D_{\sigma}^{(k)}:=\left (D_{1}^{(k)}\right )^{\sigma _{1}}\circ ...
\circ \left (D_{n}^{(k)}\right )^{\sigma _{n}}$ and $s$ is any fixed integer
in $\{1,...,m\}$.

Thus, in view of t{(\ref{eq:simm_m_arb})}, the higher order components of
an infinitesimal symmetry can be obtained from lower order ones by means
of a recurrence formula. In this sense one says that the infinitesimal
symmetries of $\mathcal{C}^{k}(\pi )$ are \textit{prolongations} of lower
order vector fields. For instance, when $m>1$ any infinitesimal symmetry
$X$ of $\mathcal{C}^{k}(\pi )$ is the prolongation of a vector field on
$J^{0}(\pi )$; in such a case, $X$ is usually referred to as a
\textit{point infinitesimal symmetry}. On the other hand, when $m=1$ an
infinitesimal symmetry $X$ of $\mathcal{C}^{k}(\pi )$ is in general the
prolongation of a symmetry of $\mathcal{C}^{1}(\pi )$, that is not necessarily
the prolongation of a vector field on $J^{0}(\pi )$. In particular, when
$m=1$, $X$ is usually referred to as a
\textit{contact infinitesimal symmetry} whenever it is not a point symmetry.

Thus, the computation of infinitesimal classical symmetries $X$ of an equation
$\mathcal{E}\subset J^{k}(\pi )$ reduces to the determination of functions
$\varphi =(\varphi ^{j})=(X\,\lrcorner \,\omega _{0}^{j})$ such that
$X$ is tangent to $\mathcal{E}$. This condition provides typically an overdetermined
linear system of PDEs for $\varphi $ that can be analyzed and explicitly
solved by considering its integrability conditions. The analysis of such
a system is generally more feasible if one uses some symbolic manipulation
package of the type developed in a computer algebra system like Maple.

On the other hand, when $\mathcal{E}$ is formally integrable, a noteworthy
extension of the notion of symmetry is possible. Indeed, one can say that
a vector field $X\in \mathcal{D}(\pi )$ is a
\textit{generalised symmetry} of $\mathcal{E}$ if, and only if, $X$ is a
symmetry of $\mathcal{C}(\pi )$ which is tangent to
$\mathcal{E}^{(\infty )}$. However, since $\mathcal{C}(\pi )$ is Frobenius,
the Lie algebra $\mathcal{DC}(\pi )$ of vector fields inscribed in
$\mathcal{C}(\pi )$ is an ideal of trivial (or characteristic) symmetries,
that one wants preliminarily gauge out from symmetries of
$\mathcal{C}(\pi )$. Thus, denoting by $\text{Sym}(\pi )$ the full Lie algebra
of symmetries of $\mathcal{C}(\pi )$, one can define the Lie algebra of
\textit{higher symmetries of} $\mathcal{C}(\pi )$ as the quotient algebra
$\text{sym}(\pi )=\text{Sym}(\pi )/\mathcal{DC}(\pi )$. Hence, since any
field of $\mathcal{DC}(\pi )$ is already tangent to
$\mathcal{E}^{(\infty )}$, all elements of a coset
$\left [X\right ]\in \text{sym}(\pi )$ (i.e.,
$\left [X\right ]=X\,\mod{\mathcal{DC}(\pi )}$, with
$X\in \text{Sym}(\pi )$) are tangent to $\mathcal{E}^{(\infty )}$, or neither
are they. Thus, since it make sense to say that an element of
$\text{sym}(\pi )$ is tangent or not to $\mathcal{E}^{(\infty )}$, one defines
the Lie algebra of \textit{higher symmetries of} $\mathcal{E}$ as the sub-algebra
$\text{sym}(\mathcal{E})\subset \text{sym}(\pi )$ of higher symmetries of
$\mathcal{C}(\pi )$ that are tangent to $\mathcal{E}^{(\infty )}$.

In coordinates, one can see that any $X\in \text{Sym}(\pi )$ has the form
%
%e11 #&#
\begin{equation}
X=\sum _{i}a_{i}D_{i}+\sum _{|\sigma |\geq 0}\sum _{j=1}^{m}D_{\sigma}
\varphi ^{j}\,\partial _{u_{\sigma}^{j}},
\label{eq:Sym}
%%LEAP%%%\label{eq11}
\end{equation}
where $\varphi =(\varphi ^{j})=(X\,\lrcorner \,\omega _{0}^{j})$ and
$a_{i}$ are some differentiable functions. Thus, any higher symmetry
$\left [X\right ]$ can be identified with its vertical (or evolutionary)
representative
%
%e12 #&#
\begin{equation}
\re _{\varphi}=\sum _{|\sigma |\geq 0}\sum _{j=1}^{m}D_{\sigma}
\varphi ^{j}\,\partial _{u_{\sigma}^{j}},
\label{eq:Evol_sym}
%%LEAP%%%\label{eq12}
\end{equation}
with generating function $\varphi =(\varphi ^{a})$. Moreover,
$\re _{\varphi}\in \text{sym}(\mathcal{E})$ whenever $\re _{\varphi}$ is
tangent to $\mathcal{E}^{(\infty )}$, i.e., $\varphi $ satisfies the system
of linear differential equations
%
%e13 #&#
\begin{equation}
\sum _{|\sigma |\geq 0}\sum _{j=1}^{m}D_{\sigma}(\varphi ^{j})\,
\partial _{u_{\sigma}^{j}}\mathbf{F}=0\mod{\mathcal{E}^{(\infty )}}.
\label{eq:eq_sym}
%%LEAP%%%\label{eq13}
\end{equation}
For any $\varphi $ of fixed order, this is typically an overdetermined
linear system for $\varphi $ that involves only a finite number of differential
consequences of $\mathcal{E}$ and can be explicitly solved by considering
its integrability conditions. As with the symmetry condition for classical
symmetries, the analysis of the linear system t{(\ref{eq:eq_sym})} is generally
more feasible if one uses some symbolic manipulation package of the type
developed in a computer algebra system like Maple.

For further details on the theory of classical and generalised symmetries,
as well as for examples and explicit computations, see
\cite{B_KrV,olver}.

%s2.3.2 #&#
\subsubsection{Pseudosymmetries of differential equations and factorisation}
%%LEAP%%%\label{sec2.3.2}
\label{subsec:Pseudosymmetries}

Pseudosymmetries were introduced by Sokolov in \cite{Sokolov} as a generalization
of infinitesimal symmetries.

In general, given a smooth distribution $\mathcal{D}$, a pseudosymmetry
of $\mathcal{D}$ can be understood as a smooth vector field $Y$ such that
$L_{Y}\mathcal{D}\subseteq \,<Y>+\mathcal{D}$. For instance, for a distribution
$\mathcal{D}=<X_{1},...,X_{n}>$, the pseudosymmetry condition is equivalent
to say that for any $i\in \{1,...,n\}$ one has
$[Y,X_{i}]=\alpha _{i}Y+\sum _{s}\beta _{i}^{s}X_{s}$, for some smooth
functions $\alpha _{i}$ and $\beta _{i}^{s}$.

Sokolov introduced also the more general notion of an $r$-pseudosymmetry
of $\mathcal{D}$, that can be seen as a system $\{Y_{1},...,Y_{r}\}$ of
vector fields such that
$L_{Y_{h}}\mathcal{D}\subseteq \,<Y_{1},...,Y_{r}>+\mathcal{D}$, for any
$h=1,...,r$. Hence, for a distribution
$\mathcal{D}=<X_{1},...,X_{n}>$, one has an $r$-pseudosymmetry
$\{Y_{1},...,Y_{r}\}$ if for any $i\in \{1,...,n\}$ the $r\times 1$ matrix
$\mathbb{Y}=\left [Y_{1},...,Y_{r}\right ]^{T}$ is such that
%
%e14 #&#
\begin{equation}
\left [\mathbb{Y},X_{i}\right ]=\mathbb{U}_{i}\mathbb{Y}+\mathbb{V}_{i}
\mathbb{X},
\label{eq:multi_pseudosim}
%%LEAP%%%\label{eq14}
\end{equation}
where
$\left [\mathbb{Y},X_{i}\right ]:=\left [\left [Y_{1},X_{i}\right ],...,
\left [Y_{r},X_{i}\right ]\right ]^{T}$ and
$\mathbb{X}=\left [X_{1},...,X_{n}\right ]^{T}$, whereas
$\mathbb{U}_{i}$ and $\mathbb{V}_{i}$ are some $r\times r$ and
$r\times n$ matrix-valued smooth functions, respectively.

It is well known that the flow of an infinitesimal symmetry always transforms
an integral manifolds of $\mathcal{D}$ to another (possibly different)
integral manifolds. Also, in view of
$[Y,X_{i}]=\sum _{s}\alpha _{i}^{s}X_{s}$, it is readily seen that an infinitesimal
symmetry of $\mathcal{D}$ sends an invariant (or first integral) $I$ of
$\mathcal{D}$ to another invariant $Y(I)$. These properties are in general
not true for pseudosymmetries. However, pseudosymmetries share with symmetries
another important property described by the following
%
%p9 #&#
\begin{prop}
%%LEAP%%%\label{prop9}
\label{prop:9}%
Let $\left \{ \xi ^{1},...,\xi ^{n},\nu ^{1},...,\nu ^{p}\right \} $ be
invariants of a pseudosymmetry $Y$ (or, an $r$-pseudosymmetry
$\{Y_{1},...,Y_{r}\}$) of $\mathcal{D}=<X_{1},...,X_{n}>$, such that
\begin{equation*}
\det \left (
\begin{array}{l@{\quad}l@{\quad}l}
X_{1}\left (\xi ^{1}\right ) & \cdots & X_{1}\left (\xi ^{n}\right )
\\
\vdots & & \vdots
\\
X_{n}\left (\xi ^{1}\right ) & \cdots & X_{n}\left (\xi ^{n}\right )
\end{array}
\right )\neq 0.
\end{equation*}
Then the functions $\nu _{i}^{j}$, defined by
%
%e15 #&#
\begin{equation}
\left (
\begin{array}{l@{\quad}l@{\quad}l}
\nu _{1}^{1} & \cdots & \nu _{1}^{p}
\\
\vdots & & \vdots
\\
\nu _{n}^{1} & \cdots & \nu _{n}^{p}
\end{array}
\right )=\left (
\begin{array}{l@{\quad}l@{\quad}l}
X_{1}\left (\xi ^{1}\right ) & \cdots & X_{1}\left (\xi ^{n}\right )
\\
\vdots & & \vdots
\\
X_{n}\left (\xi ^{1}\right ) & \cdots & X_{n}\left (\xi ^{n}\right )
\end{array}
\right )^{-1}\left (
\begin{array}{l@{\quad}l@{\quad}l}
X_{1}(\nu ^{1}) & \cdots & X_{1}(\nu ^{p})
\\
\vdots & & \vdots
\\
X_{n}(\nu ^{1}) & \cdots & X_{n}(\nu ^{p})
\end{array}
\right ),
\label{eq:Prol_invar_pseudo}
%%LEAP%%%\label{eq15}
\end{equation}
are invariants of $Y$ (resp., $\{Y_{1},...,Y_{r}\}$).
\end{prop}

\begin{proof}
Consider the matrix identity
\begin{equation*}
\left (
\begin{array}{l@{\quad}l@{\quad}l}
X_{1}\left (\xi ^{1}\right ) & \cdots & X_{1}\left (\xi ^{n}\right )
\\
\vdots & & \vdots
\\
X_{n}\left (\xi ^{1}\right ) & \cdots & X_{n}\left (\xi ^{n}\right )
\end{array}
\right )\left (
\begin{array}{l@{\quad}l@{\quad}l}
\nu _{1}^{1} & \cdots & \nu _{1}^{p}
\\
\vdots & & \vdots
\\
\nu _{n}^{1} & \cdots & \nu _{n}^{p}
\end{array}
\right )=\left (
\begin{array}{l@{\quad}l@{\quad}l}
X_{1}(\nu ^{1}) & \cdots & X_{1}(\nu ^{p})
\\
\vdots & & \vdots
\\
X_{n}(\nu ^{1}) & \cdots & X_{n}(\nu ^{p})
\end{array}
\right ).
\end{equation*}
In order to save space and avoid ambiguity, since we have the parentheses
due to the derivatives along the field X, we will rewrite last identity
as
$\left [X_{i}(\xi ^{h})\right ]\left [\nu _{h}^{j}\right ]=\left [X_{i}(
\nu ^{j})\right ]$, instead of
$\left (X_{i}(\xi ^{h})\right )\left (\nu _{h}^{j}\right )=\left (X_{i}(
\nu ^{j})\right )$. The same will be done for the other matrix identities
that follow. Thus, by taking the Lie derivative of this identity with respect
to $Y$ on gets
\begin{equation*}
\left [Y\left (X_{i}(\xi ^{h})\right )\right ]\left [\nu _{h}^{j}
\right ]+\left [X_{i}(\xi ^{h})\right ]\left [Y(\nu _{h}^{j})\right ]=
\left [Y\left (X_{i}(\nu ^{j})\right )\right ].
\end{equation*}
On the other hand, in view of
$[Y,X_{i}]=\sum _{s}\alpha _{i}^{s}X_{s}+\beta _{i}Y$ and the fact that
$Y(\xi ^{h})=Y(\nu ^{j})=0$, last identity reduces to
\begin{equation*}
\left [\sum _{s}\alpha _{i}^{s}X_{s}(\xi ^{h})\right ]\left [\nu _{h}^{j}
\right ]+\left [X_{i}(\xi ^{h})\right ]\left [Y(\nu _{h}^{j})\right ]=
\left [\sum _{s}\alpha _{i}^{s}X_{s}(\nu ^{j})\right ],
\end{equation*}
or equivalently
\begin{equation*}
\left [\alpha _{i}^{s}\right ]\left [X_{s}(\xi ^{h})\right ]\left [
\nu _{h}^{j}\right ]+\left [X_{i}(\xi ^{h})\right ]\left [Y(\nu _{h}^{j})
\right ]=\left [\alpha _{i}^{s}\right ]\left [X_{s}(\nu ^{j})\right ].
\end{equation*}
Then by
$\left [X_{s}(\xi ^{h})\right ]\left [\nu _{h}^{j}\right ]=\left [X_{s}(
\nu ^{j})\right ]$ one readily gets
$\left [X_{i}(\xi ^{h})\right ]\left [Y(\nu _{h}^{j})\right ]=0$. The result
follows from non-degeneracy condition
$\det \left [X_{i}(\xi ^{h})\right ]\neq 0$. An analogous proof holds in
the case of an $r$-pseudosymmetry $\{Y_{1},...,Y_{r}\}$, since for any
$X_{i}$ and $Y_{h}$ one has
$[Y_{h},X_{i}]=\sum _{s}\alpha _{hi}^{s}X_{s}+\sum _{k}\beta _{hi}^{k}Y_{k}$,
for some functions $a_{hi}^{s}$ and $\beta _{hi}^{k}$.
\end{proof}

Concerning the pseudosymmetries or $r$-pseudosymmetries of
$\mathcal{\mathcal{C}}(\mathcal{E})$ on $\mathcal{E}^{(\infty )}$, it is
convenient to adopt the following
%
%d10 #&#
\begin{defn}
\label{defn10}
Let $\mathcal{E}=\{\mathbf{F=0}\}\subset J^{k}(\pi )$ be a formally integrable
differential equation. A \textit{pseudosymmetry} of $\mathcal{E}$ is a vector
field $\bar{Y}$ on $\mathcal{E}^{(\infty )}$ that is a pseudosymmetry of
$\mathcal{C}(\mathcal{E})$. Analogously, an \textit{$r$-pseudosymmetry}
of $\mathcal{E}$ is a system of vector fields
$\{\bar{Y}_{1},...,\bar{Y}_{r}\}$ on $\mathcal{E}^{(\infty )}$ that is
an $r$-pseudosymmetry of $\mathcal{C}(\mathcal{E})$.
\end{defn}

In his paper \cite{Sokolov}, Sokolov discussed a method for calculating
pseudosymmetries of $\mathcal{E}$. Here we will describe that method by
means of the following proposition, along with a short proof for the convenience
of our readers.
%
%p11 #&#
\begin{prop}
%%LEAP%%%\label{prop11}
\label{prop:pseudoE}%
Let $\mathcal{E}=\{\mathbf{F=0}\}\subset J^{k}(\pi )$, with
$\text{$\mathbf{F}$}=(F^{1},...,F^{p})$, be a formally integrable
differential equation and $\mu =U_{1}dx^{1}+...+U_{n}dx^{n}$ an horizontally
closed $1$-form on $\mathcal{E}^{(\infty )}$, i.e., such that
%
%e16 #&#
\begin{equation}
D_{j}\,U_{i}-D_{i}\,U_{j}=0\mod{\mathcal{E}^{(\infty )}}.
\label{eq:U_psym}
%%LEAP%%%\label{eq16}
\end{equation}
Then, the vector field on $J^{\infty}(\pi )$
%
%e17 #&#
\begin{equation}
Y=\sum _{i}a_{i}\partial _{x_{i}}+\sum _{|\sigma |\geq 0}\sum _{j}b_{
\sigma}^{j}\partial _{u_{\sigma}^{j}}=\sum _{i}a_{i}D_{i}+\sum _{|
\sigma |\geq 0}\sum _{j}\left (D+U\right )_{\sigma}(\varphi ^{j})
\partial _{u_{\sigma}^{j}},
\label{eq:Y_pseudo_int}
%%LEAP%%%\label{eq17}
\end{equation}
where
\begin{equation*}
\begin{array}{l}
\varphi ^{j}:=Y\,\lrcorner \,\omega _{0}^{j},\qquad \qquad \omega _{0}^{j}=du^{j}-{
\displaystyle \sum _{i}}u_{i}^{j}dx_{i},
\vspace{10pt}
\\
b_{\sigma}^{j}:={\displaystyle \sum _{i}a_{i}\,u_{\sigma +1_{i}}^{j}+
\left (D+U\right )_{\sigma}(\varphi ^{j})},
\vspace{10pt}
\\
\left (D+U\right )_{\sigma}:=\left (D_{1}+U_{1}\right )^{\sigma _{1}}
\circ ...\circ \left (D_{n}+U_{n}\right )^{\sigma _{n}},
\vspace{10pt}
\end{array}
\end{equation*}
restricts on $\mathcal{E}^{(\infty )}$ to a pseudosymmetry $\bar{Y}$ of
$\mathcal{E}$ whenever $Y(F^{h})=0\mod{\mathcal{E}^{(\infty )}}$, i.e.,
%
%e18 #&#
\begin{equation}
\sum _{|\sigma |\geq 0}\sum _{j}\partial _{u_{\sigma}^{j}}\left (F^{h}
\right )\left (D+U\right )_{\sigma}(\varphi ^{j})=0
\mod{\mathcal{E}^{(\infty )}},
\label{eq:Y_tang}
%%LEAP%%%\label{eq18}
\end{equation}
for any $h=1,...,p$, which is equivalent to the tangency condition of
$Y$ to $\mathcal{E}^{(\infty )}$.
\end{prop}

\begin{proof}
In view of t{(\ref{eq:Y_pseudo_int})} one gets (for ease of notation, we use
here the Einstein's summation convention)
\begin{equation*}
\begin{array}{l@{\,}l}
\left [D_{s},Y\right ] & =\left [D_{s},a_{i}D_{i}\right ]+\left [D_{s},
\left (D+U\right )_{\sigma}(\varphi ^{j})\partial _{u_{\sigma}^{j}}
\right ]
\vspace{5pt}
\\
& =D_{s}\left (a_{i}\right )\,D_{i}+\left [D_{s}+U_{s},\left (D+U
\right )_{\sigma}(\varphi ^{j})\partial _{u_{\sigma}^{j}}\right ]-
\left [U_{s},\left (D+U\right )_{\sigma}(\varphi ^{j})\partial _{u_{
\sigma}^{j}}\right ].
\end{array}
\end{equation*}
On the other hand, by t{(\ref{eq:U_psym})} one has
\begin{equation*}
\left (D_{s}+U_{s}\right )\circ \left (D+U\right )_{\sigma}=\left (D+U
\right )_{\sigma +1_{s}}\mod{\mathcal{E}^{(\infty )}}.
\end{equation*}
Hence, above identity can be rewritten as
\begin{equation*}
\begin{array}{l@{\,}l}
\left [D_{s},Y\right ] & =D_{s}(a_{i})D_{i}+\left (D+U\right )_{
\sigma +1_{s}}(\varphi ^{j})\partial _{u_{\sigma}^{j}}-\left (D+U
\right )_{\sigma}(\varphi ^{j})\partial _{u_{\sigma -1_{s}}^{j}}
\vspace{5pt}
\\
& \qquad -\left (D+U\right )_{\sigma}(\varphi ^{j})\partial _{u_{
\sigma}^{j}}\circ U_{s}+\left (D+U\right )_{\sigma}(\varphi ^{j})
\partial _{u_{\sigma}^{j}}\circ U_{s}
\vspace{5pt}
\\
& \qquad -U_{s}\circ (Y-a_{i}D_{i})\mod{\mathcal{E}^{(\infty )}},
\end{array}
\end{equation*}
which reduces to
%
%e19 #&#
\begin{equation}
[D_{s},Y]=\left (D_{s}a_{i}+a_{i}U_{s}\right )D_{i}-U_{s}Y
\mod{\mathcal{E}^{(\infty )}},
\label{eq:aux2}
%%LEAP%%%\label{eq19}
\end{equation}
in view of arbitrariness of $\sigma $. On the other hand t{(\ref{eq:aux2})}
entails that
\begin{equation*}
\begin{array}{l@{\,}l}
D_{s}\left (Y(D_{\sigma}F^{h})\right )-Y\left (D_{s}(D_{\sigma}F^{h})
\right ) & =\left (D_{s}a_{i}+a_{i}U_{s}\right )\,D_{i}\left (D_{
\sigma}F^{h}\right )-U_{s}\,Y(D_{\sigma}F^{h})
\vspace{5pt}
\\
& =-U_{s}\,Y(D_{\sigma}F^{h})\mod{\mathcal{E}^{(\infty )}},
\end{array}
\end{equation*}
i.e.,
%
%e20 #&#
\begin{equation}
Y\left (D_{\sigma +1_{s}}F^{h}\right )=U_{s}\,Y(D_{\sigma}F^{h})+D_{s}
\left (Y(D_{\sigma}F^{h})\right )\mod{\mathcal{E}^{(\infty )}},
\label{aux-1}
%%LEAP%%%\label{eq20}
\end{equation}
for any $\sigma $ and for any $h$. Thus, by repeatedly using t{(\ref{aux-1})}
with increasing $|\sigma |$, conditions
$Y(F^{h})=0\mod{\mathcal{E}^{(\infty )}}$, $h=1,...,p$, guarantee that
$Y\left (D_{\sigma}F^{h}\right )=0\mod{\mathcal{E}^{(\infty )}}$, i.e.,
that $Y$ is tangent to $\mathcal{E}^{(\infty )}$. Therefore, in view of
t{(\ref{eq:aux2})}, one gets that $Y$ restricts on
$\mathcal{E}^{(\infty )}$ to a pseudosymmetry $\bar{Y}$ of
$\mathcal{C}(\mathcal{E})$.
\end{proof}
  Some remarks are in order here. First, in view of t{(\ref{eq:Y_pseudo_int})},
one has recurrence formulas for the components of a pseudosymmetry that
allow one to obtain higher order components from lower order ones. To distinguish
the case of pseudosymmetries from that of symmetries, we will say that
t{(\ref{eq:Y_pseudo_int})} is the \textit{pseudoprolongation}, relatively to
the $1$-form \textit{$\mu $,} of the (possibly relative) vector field
$\sum _{i}a_{i}\partial _{x_{i}}+\sum _{j}b^{j}\partial _{u^{j}}$ on
$J^{0}(\pi )$. In particular, when $\mu =0$ the pseudoprolongation reduces
to standard prolongation t{(\ref{eq:Sym})} and $Y$ is a symmetry of
$\mathcal{C}(\pi )$; thus $\varphi =(\varphi ^{j})$ is the generating function
of a (classical or generalised) symmetry of $\mathcal{E}$ whenever
$\mu =0$ and $Y$ is tangent to $\mathcal{E}^{(\infty )}$. Also, we remark
that for any given $1$-form $\mu $, the system (\ref{eq:Y_tang}) is a linear
system of PDEs for $\varphi =(\varphi ^{j})$ that involves only a finite
number of differential consequences of $\mathcal{E}$, whenever the order
of $\varphi $ is finite. As with the standard symmetry condition, the analysis
of such a system of linear PDEs is generally more feasible if one uses
some symbolic manipulation package of the type developed in a computer
algebra system like Maple. Finally, we note that, as with generalised symmetries,
also pseudosymmetries could be put in an evolutionary form (see
\cite{Chet,Sokolov}). In this article, however, we will not limit ourselves
to considering pseudosymmetries in evolutionary form, because in general
the space of invariants of a pseudosymmetry of $\mathcal{E}$ is not the
same as that of its evolutionary form, and this is a non-negligible fact
when, as in this paper, one is interested in studying factorisation by
pseudosymmetries.

For $r$-pseudosymmetries there is an analogous result illustrated by the
following proposition, which can be found in either the original paper
\cite{Sokolov} or the more recent paper \cite{Chet}.
%
%p12 #&#
\begin{prop}
%%LEAP%%%\label{prop12}
\label{prop:pseudoE-1}%
Let $\mathcal{E}=\{\mathbf{F=0}\}\subset J^{k}(\pi )$, with
$\text{$\mathbf{F}$}=(F^{1},...,F^{p})$, be a formally integrable
differential equation and
$\gamma =\mathbb{U}_{1}dx^{1}+...+\mathbb{U}_{n}dx^{n}$ a horizontal
$1$-form with $r\times r$ matrix-valued components satisfying
$d_{H}\gamma +\gamma \wedge \gamma =0\mod{\mathcal{E}^{(\infty )}}$, i.e.,
%
%e21 #&#
\begin{equation}
D_{j}\,\mathbb{U}_{i}-D_{i}\,\mathbb{U}_{j}-[\mathbb{U}_{i},\mathbb{U}_{j}]=0
\mod{\mathcal{E}^{(\infty )}}.
\label{eq:U_psym-1}
%%LEAP%%%\label{eq21}
\end{equation}
Then, the $r\times 1$ matrix
$\mathbb{Y}=\left [Y_{1},...,Y_{r}\right ]^{T}$
%
%e22 #&#
\begin{equation}
\mathbb{Y}=\mathbb{A}\mathbb{D}+\sum _{|\sigma |\geq 0}
\left (D+\mathbb{U}\right )_{\sigma}(\Phi )\partial _{\mathbf{u}_{
\sigma}},
\label{eq:Y_pseudo_int-1}
%%LEAP%%%\label{eq22}
\end{equation}
where $\mathbb{D}:=(D_{1},...,D_{n})^{T}$,
$\partial _{\mathbf{u}_{\sigma}}=(\partial _{u_{\sigma}^{1}},...,
\partial _{u_{\sigma}^{m}})^{T}$, $\mathbb{A}$ is an $r\times n$ matrix
function, $\Phi =(\varphi _{i}^{j})$ is an $r\times m$ matrix function
with $\varphi _{i}^{j}:=Y_{i}\,\lrcorner \,\omega _{0}^{j}$ and
\begin{equation*}
\left (D+\mathbb{U}\right )_{\sigma}:=\left (D_{1}+\mathbb{U}_{1}\right )^{\sigma _{1}}
\circ ...\circ \left (D_{n}+\mathbb{U}_{n}\right )^{\sigma _{n}},
\end{equation*}
defines on $\mathcal{E}^{(\infty )}$ an $r$-pseudosymmetry
$\{\bar{Y}_{1}=\left .Y_{1}\right |_{\mathcal{E}^{(\infty )}},...,
\bar{Y}_{r}=\left .Y_{r}\right |_{\mathcal{E}^{(\infty )}}\}$ of
$\mathcal{E}$ whenever
$\mathbb{Y}(F^{h})=0\mod{\mathcal{E}^{(\infty )}}$, i.e.,
%
%e23 #&#
\begin{equation}
\sum _{|\sigma |\geq 0}\partial _{\mathbf{u}_{\sigma}}\left (F^{h}
\right )\left (D+\mathbb{U}\right )_{\sigma}(\Phi )=0
\mod{\mathcal{E}^{(\infty )}},
\label{eq:Y_tang-1}
%%LEAP%%%\label{eq23}
\end{equation}
for any $h=1,...,p$, which is equivalent to the tangency condition of each
$Y_{1},...,Y_{r}$ to $\mathcal{E}^{(\infty )}$.
\end{prop}

  For $r$-pseudosymmetries apply same considerations made above for pseudosymmetries.
In particular, one can say that the system t{(\ref{eq:Y_pseudo_int-1})} is
the \textit{pseudoprolongation}, relatively to \textit{$\gamma $,} of the
(possibly relative) vector fields
$\sum _{i}a_{hi}\partial _{x_{i}}+\sum _{j}b_{h}^{j}\partial _{u^{j}}$,
$h=1,...,r$ (with
$\varphi _{h}^{j}=b_{h}^{j}-\sum _{i}a_{hi}u_{i}^{j})$ on
$J^{0}(\pi )$.

In general one has no guarantee that the system of ``derived'' (or prolonged)
invariants $\{\xi ^{i},\nu ^{j},\nu _{i}^{j}\}$ (provided by t{Proposition~\ref{prop:9}}) is functionally independent, even when
$\{\xi ^{i},\nu ^{j}\}$ is a functionally independent system. The scenario,
however, changes in the case of the Cartan distribution
$\mathcal{C}(\pi )$ on the infinite jet space, for an $m$-dimensional bundle
$\pi :E\rightarrow M$ with $n=\dim M$. Indeed, in such a case the generators
$X_{i}$ of $\mathcal{C}(\pi )$ are the total derivatives, hence the jet-order
of
$\left [\nu _{h}^{j}\right ]=\left [D_{i}\xi ^{h}\right ]^{-1}\left [D_{i}
\nu ^{j}\right ]$ (that is equation t{(\ref{eq:Prol_invar_pseudo})}) is higher
than that of $\{\xi ^{i},\nu ^{j}\}$. Thus, if
$\left \{ \xi ^{i},\nu ^{j}\right \} $ are functionally independent invariants
of a pseudosymmetry $Y$ (or, an $r$-pseudosymmetry
$\{Y_{1},...,Y_{r}\}$) of
$\mathcal{\mathcal{C}}(\pi )=<D_{1},...,D_{n}>$ satisfying
$\det \left (D_{s}\xi ^{i}\right )\neq 0$, the prolongation formulas t{(\ref{eq:Prol_invar_pseudo})}
provide an infinite sequence
$\{x_{i}'=\xi ^{i},\,u'{}_{\mu}^{j}=\nu _{\mu}^{j}\}$ of invariants of
$Y$ (resp., $\{Y_{1},...,Y_{r}\}$); when $p=m$ this sequence defines a
regular $\mathcal{C}$-morphism in $J^{\infty}(\pi )$, in view of t{Proposition~\ref{prop:3}}.

On the other hand, given a formally integrable equation
$\mathcal{E}=\{\mathbf{F}=0\}\subset J^{k}(\pi )$, if
$\big \{ \xi ^{1},...,\xi ^{n},\allowbreak \nu ^{1},...,\nu ^{p}\big \} $,
$p\geq m$, are functionally independent invariants of a pseudosymmetry
$\bar{Y}$ (resp., $r$-pseudosymmetry
$\{\bar{Y}_{1},...,\bar{Y}_{r}\}$) of
$\mathcal{\mathcal{C}}(\mathcal{E})=<\bar{D}_{1},...,\bar{D}_{n}>$ on
$\mathcal{E}^{(\infty )}$ satisfying
$\det \left (\bar{D}_{s}\xi ^{i}\right )\neq 0$, then prolongation formulas
t{(\ref{eq:Prol_invar_pseudo})} allow one to define a sequence
$\{\xi ^{i},\,\nu _{\mu}^{j}\}$ of invariants of $\bar{Y}$ (resp.,
$\{\bar{Y}_{1},...,\bar{Y}_{r}\}$) that are functionally dependent, due
to the constraints imposed by $\mathbf{F}=0$ and its differential consequences.
Indeed, assuming for instance that $\{\xi ^{i},\,\nu _{\mu}^{j}\}$ are
$\ell $-th order invariants of $\bar{Y}$, the cardinality of
$\{\xi ^{i},\,\nu _{\mu}^{j}\}$ grows with $|\mu |$ faster than the dimension
of the prolongation $\mathcal{E}^{(\ell +|\mu |-k)}$. Therefore, for sufficiently
large $|\mu |$ the system $\{\xi ^{i},\,\nu _{\mu}^{j}\}$ is necessarily
functionally dependent. In particular, a nontrivial $\bar{Y}$ has at most
$\dim \mathcal{E}^{(\ell +|\mu |-k)}-1$ functionally independent invariants
of order $\ell +|\mu |\geq k$, with the maximum reached only when
$\bar{Y}$ projects (or restricts) to a nontrivial vector field (i.e., a
derivation) on $\mathcal{E}^{(\ell +|\mu |-k)}$. Indeed, when
$\bar{Y}$ projects to a relative vector field on
$\mathcal{E}^{(\ell +|\mu |-k)}$, the invariants of the projection are
not necessarily also invariants of $\bar{Y}$. Analogous considerations
hold when $\{\xi ^{i},\,\nu _{\mu}^{j}\}$ are invariants of an $r$-pseudosymmetry.

Therefore, since $\mathbf{F}=0$ and its differential consequences entail
that for sufficiently large $|\mu |$ the system
$\{\xi ^{i},\,\nu _{\mu}^{j}\}$ is functionally dependent on
$\mathcal{E}^{(\infty )}$, it follows that the invariants
$\{\nu{}_{\sigma}^{j}\}$ satisfy an associated system of partial differential
equations $\mathbf{Q}=0$ (together with their prolongations), modulo
$\mathbf{F}=0$ and its differential consequences. This way one can look
at $\{x'_{i}=\xi ^{i},\,u'{}_{\mu}^{j}=\nu _{\mu}^{j}\}$ as the components
of a $\mathcal{C}$-morphism from
$\mathcal{E}=\{\mathbf{F}=\mathbf{0}\}$ to
$\mathcal{Q}=\{\mathbf{Q}=\mathbf{0}\}$. Of course, if one chooses another
system of invariants
$\{\hat{x}'_{i}=\hat{\xi}^{i},\,\hat{u}'{}^{j}=\hat{\nu}^{j}\}$ equivalent
to $\{\xi ^{i},\,\nu ^{j}\}$, i.e., such that
$\{\hat{\xi}^{i}=A^{i}(\xi ^{i},\nu ^{j}),\,\hat{\nu}^{j}=B^{j}(\xi ^{i},
\nu ^{j})\}$ and
$\det \left (\bar{D}_{s}\hat{\xi}^{i}\right )\neq 0$ (where
$\bar{D}_{s}$ are the total derivatives on $\mathcal{E}$), one would obtain
a system $\mathcal{\hat{Q}}$ which is point equivalent to
$\mathcal{Q}$, through the point transformation
$\{\hat{x}'_{i}=A^{i}(x',u'),\,\hat{u}'{}^{j}=B^{j}(x',u')\}$. In particular,
one can choose $\left \{ \xi ^{i},\nu ^{j}\right \} $ in such a way that
none of its element is the prolongation of another one and, in addition
to that, $\{\xi ^{i},\,\nu _{\mu}^{j}\}$ is a complete system of invariants,
i.e., a system that generates invariants of all possible orders. In such
a case, according to the common terminology used in the particular case
of symmetries (see for instance \cite{KLV,Sokolov,Svino-Sok}), one can
look at $\mathcal{Q}$ as the factorisation, or the quotient, of
$\mathcal{E}$ by $\bar{Y}$ (resp., $\{\bar{Y}_{1},...,\bar{Y}_{r}\}$).

Next two examples will provide a simple illustration of above factorisation
procedure. Other examples, describing B\"{a}cklund transformations, will
be discussed in Section~\ref{sec:last}.

%e13 #&#
\begin{example}
%%LEAP%%%\label{exmp13}
\label{exa:Cole-Hopf-II}%
The heat equation $\mathcal{E}=\{u_{t}-u_{xx}=0\}$, where $u=u(x,t)$, admits
the classical infinitesimal symmetry
$Y=u\,\partial _{u}+u_{x}\partial _{u_{x}}+u_{t}\partial _{u_{t}}+\,...$
(which defines a pseudosymmetry $\bar{Y}$ of $\mathcal{E}$ with
$\mu =0$) provided by the prolongation of $u\,\partial _{u}$; indeed it
is readily seen that
$Y(u_{t}-u_{xx})=u_{t}-u_{xx}=0\mod{\mathcal{E}}$. In this case, since
$Y$ is a classical symmetry of the ambient Cartan distribution, it projects
to a vector field on any finite order jet space $J^{k}(\pi )$ (where
$\pi :\mathbb{R}^{2}\times \mathbb{R}\rightarrow \mathbb{R}^{2},\;(x,t,u)
\mapsto (x,t)$). On the other hand, being also tangent to
$\mathcal{E}$ (hence a symmetry of $\mathcal{E}$), its restriction
$\bar{Y}:=\left .Y\right |_{\mathcal{E}^{(\infty )}}$ projects to a vector
field on $\mathcal{E}$ as well as any of its prolongations
$\mathcal{E}^{(\ell )}$. From now one we will limit ourself to domains
where $u$ and its derivatives do not vanish, hence above projections of
$Y$ are always nontrivial. Thus, being $\mathcal{E}$ a $7$-dimensional
submanifold of $J^{2}(\pi )$, $\bar{Y}$ admits
$6=\dim \mathcal{E}-1$ functionally independent invariants of order
$2$, $8=\dim \mathcal{E}^{(1)}-1$ functionally independent invariants of
order $3$, and so on. These invariants can be obtained by means of prolongations
from the following basic system of invariants
$\{\xi ^{1}=x,\,\xi ^{2}=t,\,\nu ^{1}=u_{x}/u,\,\nu ^{2}=u_{t}/u\}$ of
$\bar{Y}$, that define a $\mathcal{C}$-morphism
$\mathcal{B}:J^{\infty}(\pi )\rightarrow J^{\infty}(\pi ')$, where
$\pi ':\mathbb{R}^{2}\times \mathbb{R}^{2}\rightarrow \mathbb{R}^{2},
\;(x,t,u^{1}{}',u^{2}{}')\mapsto (x,t)$. Indeed, non-degeneracy condition
t{(\ref{eq:xi_nondeg})} is satisfied and a first prolongation provides
\begin{equation*}
u^{1}{}'_{x}=\frac{u_{xx}}{u}-\left (\frac{u_{x}}{u}\right )^{2},
\qquad u^{1}{}'_{t}=u^{2}{}'_{x}=\frac{u_{xt}}{u}-\left (
\frac{u_{x}}{u}\right )\left (\frac{u_{t}}{u}\right ),\qquad u^{2}{}'_{t}=
\frac{u_{tt}}{u}-\left (\frac{u_{t}}{u}\right )^{2}.
\end{equation*}
Thus, since the other two invariants $u^{1}{}'_{x},u^{1}{}'_{t}$ satisfy
the relations
\begin{equation*}
u^{1}{}'_{x}=u^{2}{}'-\left (u^{1}{}'\right )^{2},\qquad u^{1}{}'_{t}=u^{2}{}'_{x}
\qquad \mod{\{u_{xx}=u_{t}\}},
\end{equation*}
a basic system of second order functionally independent invariants of
$\bar{Y}$ is given by
$\{x,\,t,\,u^{1}{}',\allowbreak  \,u^{2}{}',\allowbreak \,u^{2}{}'_{x},\allowbreak  \,u^{2}{}'_{t}\}$. This means
that $\mathcal{B}$ is a $\mathcal{C}$-morphism from the heat equation
$\mathcal{E}=\{u_{t}-u_{xx}=0\}$ to the system
$\mathcal{Q}=\{u^{1}{}'_{x}-u^{2}{}'+\left (u^{1}{}'\right )^{2}=0,\,u^{1}{}'_{t}-u^{2}{}'_{x}=0
\}$, that one can interpret as the factorisation of $\mathcal{E}$ by
$\bar{Y}$. Hence,
$\left \{ u^{1}{}'=u_{x}/u,\,u^{2}{}'=u_{t}/u\right \} $ transforms solutions
of $\mathcal{E}$ to solutions of $\mathcal{Q}$, as well as of any of its
differential consequences. In particular, since the Burgers equation
$\mathcal{E}'=\{u'_{t}-u'_{xx}-2u'u'_{x}=0\}$ with $u':=u^{1}{}'$ is a
differential consequence of $\mathcal{Q}$, one obtains the Cole-Hopf transformation
$u'=u_{x}/u$ as a by-product of above transformation.
\end{example}

%e14 #&#
\begin{example}
%%LEAP%%%\label{exmp14}
\label{exa:Miura-II}%
The mKdV equation $\mathcal{E}=\{u_{t}-u_{xxx}+6u^{2}u_{x}=0\}$, can be
written in conservation law form as
$D_{t}u+D_{x}\left (-u_{xx}+2u^{3}\right )=0$. Thus any
$\mu =c\left (udx+(u_{xx}-2u^{3})dt\right )$, with $c\in \mathbb{R}$, is
a conservation law of $\mathcal{E}$ and one can search for pseudosymmetries
of the form t{(\ref{eq:Y_pseudo_int})}, where $x_{1}=x$ and $x_{2}=t$. Here
we consider the pseudosymmetry $\bar{Y}$ found by Sokolov in
\cite{Sokolov}, obtained from the vector field $\partial _{u}$ by a pseudoprolongation
relative to $2\mu $:
\begin{equation*}
\begin{array}{l@{\quad}l}
Y & =\partial _{u}+2u\partial _{u_{x}}+2(u_{xx}-2u^{3})\partial _{u_{t}}+2(u_{x}+2u^{2})
\partial _{u_{xx}}+2(u_{t}+2u(u_{xx}-2u^{3}))\partial _{u_{xt}}
\vspace{5pt}
\\
& \qquad +2(u_{xxt}-6u^{2}u_{t}+2(u_{xx}-2u^{3})^{2})\partial _{u_{tt}}+...
\;.
\end{array}
\end{equation*}
In this case, the projections of $Y$ to $J^{k}(\pi )$ (where
$\pi :\mathbb{R}^{2}\times \mathbb{R}\rightarrow \mathbb{R}^{2},\;(x,t,u)
\mapsto (x,t)$) are relative vector fields. The same occurs for projections
of $Y$ to $\mathcal{E}$ and prolongations $\mathcal{E}^{(\ell )}$. For
instance, the projection of $Y$ on $J^{1}(\pi )$ is the relative vector
field
$\partial _{u}+2u\partial _{u_{x}}+2(u_{xx}-2u^{3})\partial _{u_{t}}$.
As a consequence, $\bar{Y}$ has only the $3$ functionally independent invariants
of first order $\{\xi ^{1}=x,\,\xi ^{2}=t,\,\nu ^{1}=u_{x}-u^{2}\}$,
$3$ further functionally independent invariants of second order
$\{\nu _{x}^{1}=u_{xx}-2uu_{x},\,\nu _{t}^{1}=u_{xt}-2uu_{t},\,\nu ^{2}=u_{t}-2u
\,\nu _{x}^{1}-2u^{2}\nu ^{1}\}$ and only $3$ of the derived invariants
of third order
$\{\nu _{x}^{2},\nu _{y}^{2},\nu _{xx}^{1},\nu _{xy}^{1},\nu _{yy}^{1}
\}$ are functionally independent since
$2\nu ^{1}\nu _{x}^{1}+\nu _{x}^{2}-\nu _{t}^{1}=0$,
$2\left (\nu ^{1}\right )^{2}+\nu _{xx}^{1}-\nu ^{2}=0$. Then all higher
order invariants are generated by $\{\xi ^{i},\,\nu _{\mu}^{j}\}$ and the
factorisation of $\mathcal{E}$ by $\bar{Y}$ is described by the system
of differential equations
$\mathcal{Q}=\{2\nu ^{1}\nu _{x}^{1}+\nu _{x}^{2}-\nu _{t}^{1}=0,\,2
\left (\nu ^{1}\right )^{2}+\nu _{xx}^{1}-\nu ^{2}=0\}$, that describes
the lower order functional dependencies between the invariants
$\{\xi ^{i},\,\nu _{\mu}^{j}\}$. Hence, the map
$\left \{ \nu ^{1}=u_{x}-u^{2},\,\nu ^{2}=u_{t}-2u\,\left (u_{xx}-2uu_{x}
\right )-2u^{2}\left (u_{x}-u^{2}\right )\right \} $ transforms solutions
of $\mathcal{E}$ to solutions of $\mathcal{Q}$, as well as of any of its
differential consequences. In particular, since the KdV equation
$\mathcal{E}'=\{u'_{t}-u'_{xxx}-6u'u'_{x}=0\}$ with $u':=\nu ^{1}$ is a
differential consequence of $\mathcal{Q}$, one obtains the Miura transformation
$u'=u_{x}-u^{2}$ as a by-product of above transformation.
\end{example}

%s2.4 #&#
\subsection{Differentiable coverings and B\"acklund transformations as nonlocal
$\mathcal{C}$-morphisms}
%%LEAP%%%\label{sec2.4}
\label{subsec:Cov-BT}

Let $\mathcal{E}\subset J^{k}(\pi )$ be a formally integrable equation,
with $\pi :E\rightarrow M$, $\dim M=n$. Throughout the paper it will be
adopted the following
%
%d15 #&#
\begin{defn}
\label{defn15}
A \textit{differentiable covering} of a formally integrable equation
$\mathcal{E}$ is a fiber bundle
$\tau :\tilde{\mathcal{E}}\longrightarrow \mathcal{E}^{(\infty )}$ with
the following properties:
\end{defn}
\begin{enumerate}
\item $\tilde{\mathcal{E}}$ is equipped with a $n$-dimensional involutive
distribution $\tilde{\mathcal{C}}$;
\item for any $\theta \in \tilde{\mathcal{E}}$ the push-forward
$\tau _{*\,\theta}:\tilde{\mathcal{C}}_{\theta}\longrightarrow
\mathcal{C}_{\tau (\theta )}(\mathcal{E})$ is an isomorphism.
\end{enumerate}

According to above definition (see \cite{B_KrV}), when $\tau $ has finite
fiber dimension $q$, $\tilde{\mathcal{E}}$ is locally diffeomorphic to
the product $\mathcal{E}^{(\infty )}\times W$, where
$W\subseteq \mathbb{R}^{q}$ is an open set. Thus, locally identifying
$\tilde{\mathcal{E}}$ with $\mathcal{E}^{(\infty )}\times W$,
$\tau $ can be locally realized as the natural projection
$\mathcal{E}^{(\infty )}\times W\rightarrow \mathcal{E}^{(\infty )}$. Then,
using standard coordinates $v^{1},...,v^{q}$ in $\mathbb{R}^{q}$, the generators
$\tilde{D}_{i}$ of the involutive distribution $\tilde{\mathcal{C}}$ on
$\mathcal{E}^{(\infty )}\times W$ can be locally written as
$\tilde{D}_{i}=\bar{D}_{i}+\sum _{s}X_{i}^{s}\partial _{v^{s}}$ where
$\bar{D}_{i}$ are the total derivatives on $\mathcal{E}^{(\infty )}$ and
$X_{i}^{s}$ are smooth functions such that
$\left [\tilde{D}_{i},\tilde{D}_{j}\right ]=0$, i.e.,
$\tilde{D}_{i}X_{j}^{s}=\tilde{D}_{j}X_{i}^{s}$, in view of
$[\bar{D}_{i},\bar{D}_{j}]=0$.

In this paper we will consider only the case when
$\tilde{\mathcal{E}}$ is the infinite prolongation of a formally integrable
system $\mathcal{V}\subset J^{l}(\pi )$. Thus, in view of above considerations,
the fiber coordinates $\{u^{1},...,u^{m}\}$ of $\pi $ will be separated
in two subsets $\{z^{1},...,z^{m-q}\}$ and $\{v^{1},...,v^{q}\}$, such
that $\mathcal{E}=\left \{ \mathbf{F}(x,z_{\sigma})=0\right \} $ and
$\mathcal{V}^{(\infty )}$ is the infinite prolongation of a system
%
%e24 #&#
\begin{equation}
\left \{
\begin{array}{l}
F^{h}(x,z_{\sigma})=0,
\\
v_{i}^{s}=X_{i}^{s}(x,z_{\sigma},v).
\end{array}
\right .
\label{eq:Y_covering}
%%LEAP%%%\label{eq24}
\end{equation}
It is noteworthy to observe that, in view of the definition of a differentiable
covering
$\tau :\mathcal{V}^{(\infty )}\longrightarrow \mathcal{E}^{(\infty )}$,
any solution of $\mathcal{V}$ can be projected by $\tau $ to a solution
of $\mathcal{E}$.

A system $\mathcal{V}$ like t{(\ref{eq:Y_covering})} will be called a
$q$-dimensional \textsl{differentiable extension} of
$\mathcal{E}=\left \{ \mathbf{F}=0\right \} $. Moreover, given a differentiable
extension $\mathcal{V}$ of $\mathcal{E}$, it is customary to call
$v^{1},...,v^{q}$ the nonlocal variables of the given extension (or the
corresponding differentiable covering). Also, non-local symmetries and
nonlocal conservation laws of $\mathcal{E}$ are the symmetries and conservation
laws of a differentiable extension $\mathcal{Y}$ of $\mathcal{E}$.

Examples of differential extensions, hence of coverings, can be easily
obtained by using in t{(\ref{eq:Y_covering})} the first order equations for
the potentials of conservation laws, when $n=2$. For instance, the equations
$v_{x}=u$, $v_{t}=u_{xx}-2u^{3}$ for the potential $v$ of the conservation
law $u\,dx+(u_{xx}-2u^{3})\,dt$ of mKdV (see t{Example~\ref{exa:Miura-II}}), provide a differentiable extension of mKdV.

Other examples, that are particularly relevant in the study of B\"{a}cklund
transformations, are related to ZCRs. Indeed condition t{(\ref{eq:ZCR})},
for an $\ell \times \ell $ matrix-valued ZCR $\alpha =X\,dx+T\,dt$, is
equivalent to the integrability condition of a linear system
\begin{equation*}
\left (
\begin{array}{l}
v_{x}^{1}
\vspace{5pt}
\\
\vdots
\\
v_{x}^{\ell}
\end{array}
\right )=X\left (
\begin{array}{l}
v^{1}
\vspace{5pt}
\\
\vdots
\\
v^{\ell}
\end{array}
\right ),\qquad \left (
\begin{array}{l}
v_{t}^{1}
\vspace{5pt}
\\
\vdots
\\
v_{t}^{\ell}
\end{array}
\right )=T\left (
\begin{array}{l}
v^{1}
\vspace{5pt}
\\
\vdots
\\
v^{\ell}
\end{array}
\right ).
\end{equation*}

%e16 #&#
\begin{example}
%%LEAP%%%\label{exmp16}
\label{example_KdV_lax1}
The KdV equation $\mathcal{E}=\{z_{t}=-6zz_{x}-z_{xxx}\}$, with
$z=z(x,t)$, admits the ZCR $\alpha =X\,dx+T\,dt$ given by
\begin{equation*}
X=\left (
\begin{array}{l@{\quad}l}
0 & 1
\vspace{5pt}
\\
-z-\lambda \qquad & 0
\end{array}
\right ),\qquad T=\left (
\begin{array}{l@{\quad}l}
z_{x} & 4\lambda -2z
\vspace{5pt}
\\
z_{xx}+2z^{2}-2\lambda z-4\lambda ^{2}\qquad & -z_{x}
\end{array}
\right ),
\end{equation*}
with $\lambda $ a real parameter usually referred to as the spectral parameter.
Thus, by introducing the fiber bundle
$\mathbb{R}^{2}\times \mathbb{R}^{3}\rightarrow \mathbb{R}^{2}$,
$(x,t,z,v^{1},v^{2})\mapsto (x,t)$, the infinite prolongation of the linear
system
%
%e25 #&#
\begin{equation}
\mathcal{V}:\quad \left \{
\begin{array}{l}
v_{x}^{1}=v^{2},
\vspace{5pt}
\\
v_{x}^{2}=-(z+\lambda )v^{1},
\vspace{5pt}
\\
v_{t}^{1}=z_{x}v^{1}+\left (4\lambda -2z\right )v^{2},
\vspace{5pt}
\\
v_{t}^{2}=\left (z_{xx}+2z^{2}-2\lambda z-4\lambda ^{2}\right )v^{1}-z_{x}v^{2},
\vspace{5pt}
\\
z_{t}=-6zz_{x}-z_{xxx},
\end{array}
\right .
\label{eq:cov2_KdV}
%%LEAP%%%\label{eq25}
\end{equation}
defines a $2$-dimensional differentiable covering
$\mathcal{V}^{(\infty )}\longrightarrow \mathcal{E}^{(\infty )}$ of
$\mathcal{E}$. It is noteworthy to remark that, by setting
$v^{1}:=\phi $ and $v^{2}:=\phi _{x}$, the above first order linear system
can also be rewritten in the following equivalent form (Lax pair)
%
%e26 #&#
\begin{equation}
\left \{
\begin{array}{l}
\phi _{xx}=-z\phi -\lambda \phi ,
\vspace{5pt}
\\
\phi _{t}=2(2\lambda -z)\phi _{x}+z_{x}\phi .
\vspace{5pt}
\end{array}
\right .
\label{eq:cov1_kdv}
%%LEAP%%%\label{eq26}
\end{equation}
Indeed, it is well known \cite{Miura-Gardner-Kruskal} (see also
\cite{CHZ,Lax,MatvSal,Taktajan} and references therein) that
$\mathcal{E}$ is the integrability condition of t{(\ref{eq:cov1_kdv})}.

Also, it is interesting to notice here that the considered $2$-dimensional
differentiable covering, defined by t{(\ref{eq:cov2_KdV})}, is an extension
of a $1$-dimensional covering, by means of the equations for the potential
of a conservation law. Indeed, by introducing $\rho :=v^{2}/v^{1}$ one
can readily see that t{(\ref{eq:cov2_KdV})} reduces to
\begin{equation*}
\left \{
\begin{array}{l}
D_{x}\left (\ln v^{1}\right )=\rho ,
\vspace{5pt}
\\
D_{t}\left (\ln v^{1}\right )=z_{x}+\left (4\lambda -2z\right )\rho ,
\vspace{5pt}
\\
\rho _{x}=-\rho ^{2}-z-\lambda ,
\vspace{5pt}
\\
\rho _{t}=-2(2\lambda -z)\rho ^{2}-2z_{x}\rho +z_{xx}+2z^{2}-2
\lambda z-4\lambda ^{2},
\vspace{5pt}
\\
z_{t}=-6zz_{x}-z_{xxx}.
\end{array}
\right .
\end{equation*}
Thus t{(\ref{eq:cov2_KdV})} describes an extension of the $1$-dimensional
differentiable covering
$\tau :\mathcal{Y}^{(\infty )}\longrightarrow \mathcal{E}^{(\infty )}$,
defined by the system
%
%e27 #&#
\begin{equation}
\mathcal{Y}:\quad \left \{
\begin{array}{l}
\rho _{x}=-\rho ^{2}-z-\lambda ,
\vspace{5pt}
\\
\rho _{t}=-2(2\lambda -z)\rho ^{2}-2z_{x}\rho +z_{xx}+2z^{2}-2
\lambda z-4\lambda ^{2},
\vspace{5pt}
\\
z_{t}=-6zz_{x}-z_{xxx},
\end{array}
\right .
\label{eq:cov3_KdV}
%%LEAP%%%\label{eq27}
\end{equation}
by means of the two differential equations
$\left (\ln v^{1}\right )_{x}=\rho $ and
$\left (\ln v^{1}\right )_{t}=z_{x}+\left (4\lambda -2z\right )\rho $,
for the potential $\ln v^{1}$ of the conservation law
$\mu :=\rho \,dx+\left (z_{x}+\left (4\lambda -2z\right )\rho \right )dt$,
defined in $\tau $.
\end{example}

%e17 #&#
\begin{example}
%%LEAP%%%\label{exmp17}
\label{Covering_TZ}
The Tzitzeica equation
$\mathcal{E}=\{z_{xt}=-{{\mathrm{e}}^{-2\,z}}+{{\mathrm{e}}^{z}}\}$, with
$z=z(x,t)$, admits the ZCR $\alpha =X\,dx+T\,dt$ given by \cite{Brez}
\begin{equation*}
X=\left (
\begin{array}{c@{\quad}c@{\quad}c}
-z_{x} & 0 & \lambda
\\
\noalign{\vspace{4pt}}\lambda & z_{x} & 0
\\
\noalign{\vspace{4pt}}
0 & \lambda & 0
\end{array}
\right ),\qquad T=\left ({\displaystyle
\begin{array}{c@{\quad}c@{\quad}c}
0 & \frac{1}{\lambda}{{\mathrm{e}}^{-2\,z}} & 0
\\
\noalign{\vspace{4pt}}
0 & 0 & \frac{1}{\lambda}{{\mathrm{e}}^{z}}
\\
\noalign{\vspace{4pt}}
\frac{1}{\lambda}{{\mathrm{e}}^{z}} & 0 & 0
\end{array}
}\right ),
\end{equation*}
with $\lambda $ a real parameter. Thus, by introducing the fiber bundle
$\mathbb{R}^{2}\times \mathbb{R}^{4}\rightarrow \mathbb{R}^{2}$,
$(x,t,z,v^{1},v^{2},v^{3})\allowbreak \mapsto (x,t)$, the infinite prolongation of
the linear system
%
%e28 #&#
\begin{equation}
\mathcal{V}:\quad \left \{
\begin{array}{l@{\quad}l@{\quad}l}
v_{x}^{1}=-z_{x}v^{1}+\lambda v^{3}, &
\vspace{5pt}
& {\displaystyle v_{t}^{1}=\frac{1}{\lambda}e^{-2z}v^{2}},
\\
v_{x}^{2}=\lambda v^{1}+z_{x}v^{2}, &
\vspace{5pt}
& {\displaystyle v_{t}^{2}=\frac{1}{\lambda}e^{z}v^{3}},
\\
v_{x}^{3}=\lambda v^{2}, & & {\displaystyle v_{t}^{3}=
\frac{1}{\lambda}e^{z}v^{1}},
\end{array}
\right .
\label{Lin_cov_TZ}
%%LEAP%%%\label{eq28}
\end{equation}
defines a $3$-dimensional differentiable covering
$\mathcal{V}^{(\infty )}\longrightarrow \mathcal{E}^{(\infty )}$ of
$\mathcal{E}$. Again, similarly to what was observed for KdV in the previous
example, by means of $v^{2}=v_{x}^{3}/\lambda $ and
$v^{1}=\lambda e^{-z}v_{t}^{3}$ the first-order linear system can be reduced
to a higher-order linear one
%
%e29 #&#
\begin{equation}
\left \{
\begin{array}{l}
\phi _{xx}=z_{x}\phi _{x}+\lambda ^{3}e^{-z}\phi _{t},
\vspace{5pt}
\\
\phi _{xt}=e^{z}\phi ,
\vspace{5pt}
\\
{\displaystyle \phi _{tt}=z_{t}\phi _{t}+\frac{1}{\lambda ^{3}}e^{-z}
\phi _{x}},
\end{array}
\right .
\label{lin_phi_TZ}
%%LEAP%%%\label{eq29}
\end{equation}
where $\phi :=v^{3}$. Indeed, one can readily see that $\mathcal{E}$ is
also the integrability condition of t{(\ref{lin_phi_TZ})}. Moreover, similarly
to the previous example, the $3$-dimensional covering defined by t{(\ref{Lin_cov_TZ})}
is the extension of a $2$-dimensional covering, by means of the equations
for the potential of a conservation law. Indeed, by introducing
$\rho ^{1}=v^{1}/v^{3}$ and $\rho ^{2}=v^{2}/v^{3}$ one can readily see
that t{(\ref{Lin_cov_TZ})} reduces to
\begin{equation*}
\left \{
\begin{array}{l}
D_{x}\left (\ln v^{3}\right )=\lambda{\rho ^{2}},
\vspace{5pt}
\\
{\displaystyle D_{t}\left (\ln v^{3}\right )=\frac{1}{\lambda}e^{z}{
\rho ^{1}}},
\vspace{5pt}
\\
\rho _{x}^{1}=-\lambda \,\rho ^{1}\rho ^{2}-z_{x}\,\rho ^{1}+\lambda ,
\vspace{5pt}
\\
{\displaystyle \rho _{t}^{1}=-\frac{1}{\lambda}{{\mathrm{e}}^{z}}{\left (
\rho ^{1}\right )}^{2}+\frac{1}{\lambda}e^{-2z}\rho ^{2}},
\vspace{5pt}
\\
\rho _{x}^{2}=-\lambda{\left (\rho ^{2}\right )}^{2}+z_{x}{\rho ^{2}}+
\lambda \rho ^{1},
\vspace{5pt}
\\
{\displaystyle \rho _{t}^{2}=-\frac{1}{\lambda}{{\mathrm{e}}^{z}}{\rho ^{1}}{
\rho ^{2}}+\frac{1}{\lambda}e^{z}},
\vspace{5pt}
\\
z_{xt}=-{{\mathrm{e}}^{-2\,z}}+{{\mathrm{e}}^{z}}.
\end{array}
\right .
\end{equation*}
Thus t{(\ref{Lin_cov_TZ})} describes an extension of the $2$-dimensional differentiable
covering
$\tau :\mathcal{Y}^{(\infty )}\longrightarrow \mathcal{E}^{(\infty )}$,
defined by the system
%
%e30 #&#
\begin{equation}
\mathcal{Y}:\quad \left \{
\begin{array}{l}
\rho _{x}^{1}=-\lambda \,\rho ^{1}\rho ^{2}-z_{x}\,\rho ^{1}+\lambda ,
\vspace{5pt}
\\
{\displaystyle \rho _{t}^{1}=-\frac{1}{\lambda}{{\mathrm{e}}^{z}}{\left (
\rho ^{1}\right )}^{2}+\frac{1}{\lambda}e^{-2z}\rho ^{2}},
\vspace{5pt}
\\
\rho _{x}^{2}=-\lambda{\left (\rho ^{2}\right )}^{2}+z_{x}{\rho ^{2}}+
\lambda \rho ^{1},
\vspace{5pt}
\\
{\displaystyle \rho _{t}^{2}=-\frac{1}{\lambda}{{\mathrm{e}}^{z}}{\rho ^{1}}{
\rho ^{2}}+\frac{1}{\lambda}e^{z}},
\vspace{5pt}
\\
z_{xt}=-{{\mathrm{e}}^{-2\,z}}+{{\mathrm{e}}^{z}},
\end{array}
\right .
\label{Riccati_TZ}
%%LEAP%%%\label{eq30}
\end{equation}
by means of the two differential equations
$\left (\ln v^{3}\right )_{x}=\lambda{\rho ^{2}}$ and
$\left (\ln v^{3}\right )_{t}=e^{z}{\rho ^{1}/\lambda}$, for the
potential $\ln v^{3}$ of the conservation law
$\mu :=\lambda{\rho ^{2}}\,dx+\frac{1}{\lambda}e^{z}{\rho ^{1}}\,dt$, defined in $\tau $.
\end{example}

  For further details on the theory of differential coverings, as well
as for examples and explicit computations, see \cite{B_KrV} and references
therein.

For B\"{a}cklund transformations one can adopt the following
%
%d18 #&#
\begin{defn}
%%LEAP%%%\label{defn18}
\label{def:BT-ABT}%
A B\"{a}cklund transformation from a formally integrable equation
$\mathcal{E}$ of order $k$ to another formally integrable equation
$\mathcal{E}'\subset J^{k'}(\pi ')$, both with the same number $n$ of independent
variables, is a regular $\mathcal{C}$-morphism
$\mathcal{B}:J^{\infty}(\pi )\rightarrow J^{\infty}(\pi ')$ from
$\mathcal{V}$ to $\mathcal{E}'$, for some differentiable covering
$\tau :\mathcal{V}^{(\infty )}\longrightarrow \mathcal{E}^{(\infty )}$
defined by a differentiable extension
$\mathcal{V}\subset J^{l}(\pi )$ of $\mathcal{E}$. In particular
$\mathcal{B}$ is an auto-B\"{a}cklund transformation whenever
$\mathcal{E}'$ is a copy of $\mathcal{E}$.
\end{defn}

  According to this definition, a B\"{a}cklund transformation can be seen
as a pair of $\mathcal{C}$-morphisms
\begin{equation*}
\begin{array}{c@{\quad}c@{\quad}c}
& \mathcal{V}^{(\infty )}
\\
\qquad \tau \swarrow
\vspace{3pt}
& & \searrow \mathcal{B}
\vspace{3pt}
\qquad
\\
\mathcal{E}^{(\infty )} & & \qquad \mathcal{E}'{}^{(\infty )}
\end{array}
\end{equation*}
which is analogous to the point of view followed by Krasil'shchik and Vinogradov
in \cite{Kras-Vin}. However, in this paper, we will not limit ourselves
to considering B\"{a}cklund transformations, which, like in
\cite{Kras-Vin}, keep the independent variables unchanged.

%e19 #&#
\begin{example}
%%LEAP%%%\label{exmp19}
\label{exa:KdV_B11}%
Continuing t{Example~\ref{example_KdV_lax1}} we can provide an example of
auto-B\"{a}cklund transformation for KdV. Indeed, by considering the trivial
bundle
$\pi :\mathbb{R}^{2}\times \mathbb{R}^{2}\rightarrow \mathbb{R}^{2}$,
$(x,t,z,\rho )\mapsto (x,t)$ and the $1$-dimensional covering
$\tau :\mathcal{Y}^{(\infty )}\longrightarrow \mathcal{E}^{(\infty )}$,
it can be checked that the $\mathcal{C}$-morphism
$\mathcal{B}:J^{\infty}(\pi )\rightarrow J^{\infty}(\pi ')$ defined by
%
%e31 #&#
\begin{equation}
z'=-z-2\rho ^{2}-2\lambda ,
\label{eq:B11_KdV}
%%LEAP%%%\label{eq31}
\end{equation}
through the prolongation formulas t{(\ref{Prol_mor})}, defines a
$\mathcal{C}$-morphism from $\mathcal{Y}$ to
$\mathcal{E}'=\{z'_{t}+6z'z'_{x}+z'_{xxx}=0\}$, which is a copy of
$\mathcal{E}$. Hence t{(\ref{eq:B11_KdV})} provides an auto--B\"{a}cklund
transformation of KdV. In practice, for any known solution $f_{0}$ of the
KdV, such a transformation associates a new solution
\begin{equation*}
f_{1}=-f_{0}-2g_{0}^{2}-2\lambda _{0},
\end{equation*}
where $g_{0}$ is the solution of the Riccati-type system corresponding
to $f_{0}$ for some $\lambda _{0}$, i.e.,
\begin{equation*}
\left \{
\begin{array}{l}
g_{0x}=-g_{0}^{2}-f_{0}-\lambda _{0},
\vspace{5pt}
\\
g_{0t}=-2(2\lambda _{0}-f_{0})g_{0}^{2}-2f_{0x}g_{0}+f_{0xx}+2f_{0}^{2}-2
\lambda _{0}f_{0}-4\lambda _{0}^{2},
\vspace{5pt}
\end{array}
\right .
\end{equation*}
Further iterations with pairwise different parameters
$\lambda _{1},\lambda _{2},...$ leads to a sequence of new solutions
\begin{equation*}
f_{i}=-f_{i-1}-2g_{i-1}^{2}-2\lambda _{i-1},
\end{equation*}
where each $g_{i-1}$ is the solution of the Riccati-type system corresponding
to $f_{i-1}$ and $\lambda _{i-1}$.
\end{example}

%r20 #&#
\begin{rem}
\label{rem20}
Since $\rho =\phi _{x}/\phi $, one can readily check that in terms of the
Lax potential above auto-B\"{a}cklund transformation can be written as
$z'=-z-2(\phi _{x}/\phi )^{2}-2\lambda $, or equivalently as
$z'=z+2\left (\ln \phi \right )_{xx}$, in view of t{(\ref{eq:cov1_kdv})}.
Thus, for any known solution $f_{0}$ of the KdV, the transformation associates
the new solution $f_{1}=f_{0}+2\left (\ln g_{0}\right )_{xx}$ where
$g_{0}$ is the corresponding solution of Lax system
\begin{equation*}
\left \{
\begin{array}{l}
g_{0xx}=-f_{0}g_{0}-\lambda _{0}g_{0},
\vspace{5pt}
\\
g_{0t}=2(2\lambda _{0}-f_{0})g_{0x}+f_{0x}g_{0}.
\vspace{5pt}
\end{array}
\right .
\end{equation*}
\end{rem}

%e21 #&#
\begin{example}
%%LEAP%%%\label{exmp21}
\label{exa:KdV_B21}%
If in the previous example we augment the space by passing to the bundle
$\hat{\pi}:\mathbb{R}^{2}\times \mathbb{R}^{3}\rightarrow \mathbb{R}^{2},(x,t,z,
\rho ,\hat{\rho})\mapsto (x,t)$, and instead of t{(\ref{eq:cov3_KdV})} consider
the system
%
%e32 #&#
\begin{equation}
\hat{\mathcal{Y}}:\quad \left \{
\begin{array}{l}
\hat{\rho}_{x}=-\hat{\rho}^{2}-z-\hat{\lambda},
\vspace{5pt}
\\
\hat{\rho}_{t}=-2(2\hat{\lambda}-z)\hat{\rho}^{2}-2z_{x}\hat{\rho}+z_{xx}+2z^{2}-2
\hat{\lambda}z-4\hat{\lambda}^{2},
\vspace{5pt}
\\
\rho _{x}=-\rho ^{2}-z-\lambda ,
\vspace{5pt}
\\
\rho _{t}=-2(2\lambda -z)\rho ^{2}-2z_{x}\rho +z_{xx}+2z^{2}-2
\lambda z-4\lambda ^{2},
\vspace{5pt}
\\
z_{t}=-6zz_{x}-z_{xxx},
\end{array}
\right .
\label{eq:cov2_KdV-1}
%%LEAP%%%\label{eq32}
\end{equation}
obtained by doubling the Riccati-type system, i.e., by augmenting another
copy of the Riccati-type system with spectral parameter
$\hat{\lambda}\neq \lambda $, then the natural projection
$\hat{\tau}:\hat{\mathcal{Y}}^{(\infty )}\longrightarrow \mathcal{E}^{(
\infty )}$ is another covering of KdV and the $\mathcal{C}$-morphism
$J^{\infty}(\hat{\pi})\rightarrow J^{\infty}(\pi ')$ defined by
%
%e33 #&#
\begin{equation}
z'=-z-2\hat{\rho}^{2}-2\hat{\lambda},\qquad \rho '=-
\frac{\hat{\lambda}-\lambda}{\hat{\rho}-\rho}-\hat{\rho},
\label{eq:DT_KdV}
%%LEAP%%%\label{eq33}
\end{equation}
is a $\mathcal{C}$-morphism from $\hat{\mathcal{Y}}$ to the system
\begin{equation*}
\mathcal{E}':\quad \left \{
\begin{array}{l}
\rho '_{x}=-\rho '{}^{2}-z'-\lambda ,
\vspace{5pt}
\\
\rho '_{t}=-2(2\lambda -z')\rho '{}^{2}-2z'_{x}\rho '+z'_{xx}+2z'{}^{2}-2
\lambda z'-4\lambda ^{2},
\vspace{5pt}
\\
z'_{t}=-6z'z'_{x}-z'_{xxx}.
\end{array}
\right .
\end{equation*}
In this case, to obtain a new solution $z'$ of KdV out of a given one
$z=f_{0}$, the first step is to compute the corresponding solution
$\rho =g_{0}$ of the Riccati-type system, with arbitrary $\lambda $. Then,
since $g_{0}$ depends on $\lambda $, one gets another instance
$\hat{g}_{0}$ of that function, by substituting $\lambda $ with
$\hat{\lambda}\neq \lambda $. This way t{(\ref{eq:DT_KdV})} provides the new
pair of functions
\begin{equation*}
z'=-f_{0}-2\hat{g}_{0}^{2}-2\hat{\lambda},\qquad \rho '=-
\frac{\hat{\lambda}-\lambda}{\hat{g}_{0}-g}-\hat{g}_{0},
\end{equation*}
that satisfy
\begin{equation*}
\left \{
\begin{array}{l}
\rho _{x}=-\rho ^{2}-z-\lambda ,
\vspace{5pt}
\\
\rho _{t}=-2(2\lambda -z)\rho ^{2}-2z_{x}\rho +z_{xx}+2z^{2}-2
\lambda z-4\lambda ^{2},
\vspace{5pt}
\\
z_{t}=-6zz_{x}-z_{xxx}.
\end{array}
\right .
\end{equation*}
The peculiarity of this transformation, in comparison with that of previous
example, is that it provides the new solution $z'$ together with the corresponding
solution $\rho '$ of the Riccati-type system, with arbitrary
$\lambda $. Thus, further applications of the transformations only require
derivations and algebraic operations.

It is noteworthy to remark here that, passing to the Lax potential
$\phi $, the transformation t{(\ref{eq:DT_KdV})} takes the form (with
$\phi $ and $\hat{\phi}$ solutions of Lax equations with parameters
$\lambda $ and $\hat{\lambda}$, respectively)
\begin{equation*}
z'=z+2\left (\ln \hat{\phi}\right )_{xx},\qquad \phi '=\phi _{x}-
\phi \left (\ln \hat{\phi}\right )_{x},
\end{equation*}
which is known in literature as the Darboux transformation of KdV (see
for instance \cite{CHZ,MatvSal}). Thus, the notion of B\"{a}cklund
transformation described by t{Definition~\ref{def:BT-ABT}} allows one to regard
also Darboux transformations as examples of B\"{a}cklund transformations.
\end{example}

%e22 #&#
\begin{example}
%%LEAP%%%\label{exmp22}
\label{exa:SP_B11}%
Consider the covering
$\tau :\mathcal{Y}^{(\infty )}\longrightarrow \mathcal{E}^{(\infty )}$
of the short-pulse equation
$\mathcal{E}=\{z_{xt}=\frac{1}{2}z^{2}z_{xx}+zz_{x}^{2}+z\}\subset J^{3}(
\pi )$, defined by the system
%
%e34 #&#
\begin{equation}
\mathcal{Y}:\quad \left \{
\begin{array}{l}
{\displaystyle \rho _{x}=-{\displaystyle \frac{z_{x}}{4\eta}}\rho ^{2}+{
\displaystyle \frac{1}{2\eta}}\,\rho +{\displaystyle
\frac{z_{x}}{4\eta}}},
\vspace{10pt}
\\
{\displaystyle \rho _{t}=-{\displaystyle
\frac{z\left (zz_{x}+4\,\eta \right )}{8\eta}}\rho ^{2}+{
\displaystyle \frac{\left (8\,\eta ^{2}+z^{2}\right )}{4\eta}}\rho -{
\displaystyle \frac{z\left (-zz_{x}+4\,\eta \right )}{8\eta}}},
\vspace{10pt}
\\
{\displaystyle z_{xt}={\displaystyle \frac{1}{2}}z^{2}z_{xx}+zz_{x}^{2}+z},
\end{array}
\right .
\label{eq:cov1_SP}
%%LEAP%%%\label{eq34}
\end{equation}
where
$\pi :\mathbb{R}^{2}\times \mathbb{R}^{2}\rightarrow \mathbb{R}^{2}$,
$(x,t,z,\rho )\mapsto (x,t)$. One can check that the transformation
%
%e35 #&#
\begin{equation}
x'=x+8\frac{\eta}{\rho ^{2}+1},\qquad t'=t,\qquad z'=-z+8\eta
\frac{\rho}{\rho ^{2}+1},
\label{eq:B11_SP}
%%LEAP%%%\label{eq35}
\end{equation}
determines an auto-B\"{a}cklund transformation, i.e., a B\"{a}cklund transformation
from $\mathcal{E}=\{z_{xt}=\frac{1}{2}z^{2}z_{xx}+zz_{x}^{2}+z\}$ to
$\mathcal{E}'=\{z'_{xt}=\frac{1}{2}z'{}^{2}z'_{xx}+z'z'{}_{x}^{2}+z'
\}$.%
\end{example}

%s3 #&#
\section{Riccati-type differentiable coverings determined by ZCRs}
%%LEAP%%%\label{sec3}
\label{subsec:Riccati-type-differentiable-cove}

In this section, starting from a ZCR of an equation $\mathcal{E}$, we will
show how to determine a first-order Riccati-type system defining a differentiable
covering together with a corresponding nonlocal conservation law. We will
see in Section~\ref{sec:last} that this type of covering and the corresponding
nonlocal conservation law are fundamental for the computation of pseudosymmetries
that allow one to determine B\"{a}cklund transformations.

Let
$\alpha \in \mathfrak{g}\otimes \bar{\Lambda}^{1}\left ({\mathcal{E}}
\right )$ be a ZCR of the form $\alpha =X\,dx+T\,dt$, for a formally integrable
equation $\mathcal{E}:=\left \{ \mathbf{F=0}\right \} $ with two independent
variables $\left (x,t\right )$. We assume here that $X$ and $T$ are
$\ell \times \ell $ matrices, with $\ell \geq 2$.

One can readily check that t{(\ref{eq:ZCR})} is equivalent to the integrability
condition
%
%e36 #&#
\begin{equation}
\bar{D}_{t}X-\bar{D}_{x}T+\left [X,T\right ]=0
\label{eq:Matrix_ZCR}
%%LEAP%%%\label{eq36}
\end{equation}
of a linear system of the form
%
%e37 #&#
\begin{equation}
V_{x}=XV,\qquad V_{t}=TV,
\label{eq:Linear_problem}
%%LEAP%%%\label{eq37}
\end{equation}
where $V=(v^{1},...,v^{\ell})^{T}$ is an auxiliary vector valued function
of $(x,t)$.

This means that, by suitably enlarging the space of dependent variables,
$\alpha $ defines a differentiable covering
$\mathcal{V}^{(\infty )}\rightarrow \mathcal{E}^{(\infty )}$, with
$\mathcal{V}:=\left \{ \mathbf{F=0},\,\bar{D}_{x}V=XV,\,\bar{D}_{t}V=TV
\right \} $. In such an enlarged space, t{(\ref{eq:Matrix_ZCR})} can be also
understood as the integrability condition of an equation of the form
%
%e38 #&#
\begin{equation}
\bar{d}_{H}V=XV\,dx+TV\,dt,
\label{eq:dV}
%%LEAP%%%\label{eq38}
\end{equation}
where $\bar{d}_{H}V=\bar{D}_{x}V\,dx+\bar{D}_{t}V\,dt$. Notice that
$V_{x}=\tilde{D}_{x}V$ and $V_{t}=\tilde{D}_{t}V$, in terms of the total
derivatives $\tilde{D}_{x}$ and $\tilde{D}_{t}$ in
$\mathcal{V}^{(\infty )}$.

Now, it is noteworthy to observe that linear transformations
$V\mapsto V^{S}:=SV$, by a nonsingular matrix valued function $S$ on
$\mathcal{E}^{(\infty )}$, naturally determine a $\mathcal{C}$-morphism
of $\mathcal{V}$. Under such transformations, t{(\ref{eq:Linear_problem})}
transforms as
\begin{equation*}
V_{x}^{S}=X^{S}V^{S},\qquad V_{t}^{S}=T^{S}V^{S},
\end{equation*}
with
\begin{equation*}
X^{S}:=\bar{D}_{x}\left (S\right )S^{-1}+SXS^{-1},\qquad T^{S}:=
\bar{D}_{t}\left (S\right )S^{-1}+STS^{-1}.
\end{equation*}
Accordingly, t{(\ref{eq:dV})} transforms as
$\bar{d}_{H}V^{S}=X^{S}V^{S}\,dx+T^{S}V^{S}\,dt$. These transformations
are usually referred to as \textit{gauge transformations}.

We have the following
%
%l23 #&#
\begin{lem}
%%LEAP%%%\label{lem23}
\label{lem22}%
For any fixed $h\in \{1,...,\ell \}$, under the gauge transformation
$V\mapsto V'=\frac{1}{v^{h}}V$ equation t{(\ref{eq:dV})} rewrites as
%
%e39 #&#
\begin{equation}
\left \{
\begin{array}{l}
\bar{d}_{H}\left (\ln v^{h}\right )=\left [\alpha \right ]^{h}V_{}',
\vspace{10pt}
\\
\bar{d}_{H}V'=\alpha V'-\left (\left [\alpha \right ]^{h}V'\right )V',
\end{array}
\right .
\label{eq:Lema22-1}
%%LEAP%%%\label{eq39}
\end{equation}
where $\left [\alpha \right ]^{h}$ denotes the $h$-th row of
$\alpha $. Thus, the first order system t{(\ref{eq:Linear_problem})} is equivalent
to the following one
%
%e40 #&#
\begin{equation}
\left \{
\begin{array}{l}
{\displaystyle \frac{v_{x}^{h}}{v^{h}}}=\left [X\right ]^{h}V',
\vspace{10pt}
\\
{\displaystyle \frac{v_{t}^{h}}{v^{h}}}=\left [T\right ]^{h}V',
\vspace{10pt}
\\
V'_{x}=XV_{}'-\left (\left [X\right ]^{h}V'\right )V',
\\
V'_{t}=TV'-\left (\left [T\right ]^{h}V'\right )V',
\end{array}
\right .
\label{eq:Lema22-2}
%%LEAP%%%\label{eq40}
\end{equation}
where $\left [X\right ]^{h}$ and $\left [T\right ]^{h}$ denote the
$h$-th rows of $X$ and $T$, respectively.
\end{lem}

\begin{proof}
The result follows by substituting $V=v^{h}V'$ (with $h$ fixed) in t{(\ref{eq:dV})}.
\end{proof}
  Since $V'=\frac{1}{v^{h}}V$, it is natural to write $V'$ as
\begin{equation*}
V'=(\rho ^{1},...,\rho ^{h-1},1,\rho ^{h},...,\rho ^{\ell -1})^{T},
\end{equation*}
where
\begin{equation*}
\rho ^{j}:=\left \{
\begin{array}{l@{\quad}l}
v^{j}/v^{h},\qquad & j<h,
\vspace{10pt}
\\
v^{j+1}/v^{h},\qquad & j>h.
\end{array}
\right .
\end{equation*}
It follows that the differentiable covering
$\mathcal{V}^{(\infty )}\rightarrow \mathcal{E}^{(\infty )}$ is the extension,
by means of the potential of a nonlocal conservation law, of a differentiable
covering
$\mathcal{Y}^{(\infty )}\rightarrow \mathcal{E}^{(\infty )}$ defined by
a system of $\ell -1$ Riccati-type equations. This generalises a fact already
discovered by Chern and Tenenblat in \cite{CT}, in the case of
$\mathfrak{sl}_{2}(\mathbb{R})$-valued ZCRs for equations describing pseudospherical
surfaces.

For instance, by taking $h=\ell $ for ease of notations, one has the following
%
%p24 #&#
\begin{prop}
%%LEAP%%%\label{prop24}
\label{ZCR-Riccati_NLCL}%
To any zero-curvature representation $\alpha =X\,dx+T\,dt$ of
$\mathcal{E}:=\left \{ \mathbf{F=0}\right \} $, with
$\ell \times \ell $ matrices $X$ and $T$, it is associated the
$(\ell -1)$-dimensional differentiable covering defined by the system
\begin{equation*}
\left \{
\begin{array}{l}
\rho _{x}^{j}={\displaystyle \sum _{s=1}^{\ell -1}\left (-X_{\ell s}
\rho ^{s}\rho ^{j}+X_{js}\rho ^{s}\right )-X_{\ell \ell}\rho ^{j}+X_{j
\ell},\qquad j=1,...,\ell -1}
\vspace{10pt}
\\
\rho _{t}^{j}={\displaystyle \sum _{s=1}^{\ell -1}\left (-T_{\ell s}
\rho ^{s}\rho ^{j}+T_{js}\rho ^{s}\right )-T_{\ell \ell}\rho ^{j}+T_{j
\ell},\qquad j=1,...,\ell -1}
\vspace{10pt}
\\
\mathbf{F=0}.
\end{array}
\right .
\end{equation*}
In particular, with respect to this covering, $\mathcal{E}$ admits the
following nonlocal conservation law
\begin{equation*}
\mu =\left (\sum _{s=1}^{\ell -1}X_{\ell s}\rho ^{s}+X_{\ell \ell}
\right )dx+\left (\sum _{s=1}^{\ell -1}T_{\ell s}\rho ^{s}+T_{\ell
\ell}\right )dt.
\end{equation*}
\end{prop}

  Analogous results hold for any other $h=1,2,...,\ell -1$, with the corresponding
nonlocal variables
$\bar{\rho}_{1},\bar{\rho}_{2},...,\bar{\rho}_{\ell -1}$ reciprocally related,
for any choice of $h$, to the
$\rho _{1},\rho _{2},...,\rho _{\ell -1}$ above.

%e25 #&#
\begin{example}
%%LEAP%%%\label{exmp25}
\label{example_kdv_lax1_cont}%
We consider here KdV equation
$\mathcal{E}=\{z_{xxx}=-z_{t}-6zz_{x}\}$ with the ZCR
\begin{equation*}
\alpha =\left (
\begin{array}{c@{\quad}c}
0 & 1
\vspace{5pt}
\\
-\lambda -z\qquad & 0
\end{array}
\right )dx+\left (
\begin{array}{c@{\quad}c}
z_{x} & 4\,\lambda -2\,z
\vspace{5pt}
\\
-4\lambda ^{2}-2\,\lambda \,z+2\,z^{2}+z_{xx}\qquad & -z_{x}
\end{array}
\right )dt,\qquad \lambda \in \mathbb{R}.
\end{equation*}
In this case, the procedure described above, with
$\rho =v^{2}/v^{1}$, provides the $1$-dimensional differentiable covering
$\mathcal{Y}^{(\infty )}\rightarrow \mathcal{E}^{(\infty )}$ defined by
the system t{(\ref{eq:cov3_KdV})} of t{Example~\ref{example_KdV_lax1}}, together
with the nonlocal conservation law
$\mu =\rho dx+\left (z_{x}+\left (4\lambda -2z\right )\rho \right )dt$.
According to first equation of t{(\ref{eq:Lema22-1})}
$\mu =\bar{d}_{H}(\ln v^{1})$, thus $\rho =v_{x}^{1}/v^{1}$.
\end{example}

%e26 #&#
\begin{example}
%%LEAP%%%\label{exmp26}
\label{exa:Tzitzeica1}%
We consider here Tzitzeica equation
$\mathcal{E}=\{z_{xt}=-{{\mathrm{e}}^{-2\,z}}+{{\mathrm{e}}^{z}}\}$ with the ZCR
\begin{equation*}
\alpha =\left (
\begin{array}{c@{\quad}c@{\quad}c}
-z_{x} & 0 & \lambda
\\
\noalign{\vspace{4pt}}
\lambda & z_{x} & 0
\\
\noalign{\vspace{4pt}}
0 & \lambda & 0
\end{array}
\right )dx+\left (
\begin{array}{c@{\quad}c@{\quad}c}
0 & {\displaystyle \frac{1}{\lambda}{{\mathrm{e}}^{-2\,z}}} & 0
\\
\noalign{\vspace{4pt}}
0 & 0 & {\displaystyle \frac{1}{\lambda}{{\mathrm{e}}^{z}}}
\\
\noalign{\vspace{4pt}}
{\displaystyle \frac{1}{\lambda}{{\mathrm{e}}^{z}}} & 0 & 0
\end{array}
\right )dt,\qquad \lambda \in \mathbb{R}.
\end{equation*}
In this case, the procedure described above, with
$\rho ^{1}=v^{1}/v^{3}$ and $\rho ^{2}=v^{2}/v^{3}$, provides the
$2$-dimensional differentiable covering
$\mathcal{Y}^{(\infty )}\rightarrow \mathcal{E}^{(\infty )}$ defined by
the system t{(\ref{Riccati_TZ})} of t{Example~\ref{Covering_TZ}}, together with
the nonlocal conservation law
$\mu =\lambda{\rho ^{2}}\,dx+e^{z}{\rho ^{1}/\lambda}\,dt$.
In this case, according to first equation of t{(\ref{eq:Lema22-1})}
$\mu =\bar{d}_{H}(\ln v^{3})$, thus
$\rho ^{1}=\lambda e^{-z}v_{t}^{3}/v^{3}$ and
$\rho ^{2}=\lambda ^{-1}v_{x}^{3}/v^{3}$.
\end{example}

%s4 #&#
\section{Nonlocal pseudosymmetries and B\"acklund transformations}
%%LEAP%%%\label{sec4}
\label{sec:last}

As discussed in Subsection \ref{subsec:Cov-BT}, a B\"{a}cklund transformation
from $\mathcal{E}$ to $\mathcal{E}'$ can be seen as a nonlocal
$\mathcal{C}$-morphism, from a differentiable extension
$\mathcal{V}$ of the equation $\mathcal{E}$ to the equation
$\mathcal{E}'$. Our aim here is to show how such transformations can result
from factorisations of differentiable extensions $\mathcal{V}$ with respect
to nonlocal pseudosymmetries. In these cases one specialises to
$\mathcal{V}$ the factorisation procedure described in the end of Subsection \ref{subsec:Pseudosymmetries}.

The basic assumption is that $\mathcal{E}$ admits a differentiable extension
$\mathcal{V}$ which has a pseudosymmetry $\bar{Y}$ (or an $r$-pseudosymmetry
$\{\bar{Y}_{1},...,\bar{Y}_{r}\}$) which leads to a quotient system
$\mathcal{Q}$ that, among its equations, has a copy of equation
$\mathcal{E }'$. In such a case, the B\"{a}cklund transformation
is determined by a suitable choice of some basic system of functionally
independent invariants $S=\{x'_{i}=\xi ^{i},\,z'{}^{j}=\nu ^{j}\}$ of the
exploited pseudosymmetry (resp., $r$-pseudosymmetry). Indeed, since the
concrete form of $\mathcal{Q}$ depends on the particular choice of
$S$, if one chooses another basic system of functionally independent invariants
$\hat{S}=\{\hat{x}'_{i}=\hat{\xi}^{i},\,\hat{z}'{}^{j}=\hat{\nu}^{j}
\}$ equivalent to $\{\xi ^{i},\,\nu ^{j}\}$, i.e., such that
$\{\hat{\xi}^{i}=A^{i}(\xi ^{i},\nu ^{j}),\,\hat{\nu}^{j}=B^{j}(\xi ^{i},
\nu ^{j})\}$ and
$\det \left (\tilde{D}_{s}\hat{\xi}^{i}\right )\neq 0$ (where
$\tilde{D}_{s}$ are the total derivatives on
$\mathcal{V}^{(\infty )}$), one would obtain a system
$\mathcal{\hat{Q}}$ which is point equivalent to $\mathcal{Q}$, through
the point transformation
$\{\hat{x}'_{i}=A^{i}(x',z'),\,\hat{z}'{}^{j}=B^{j}(x',z')\}$.

On the other hand, in view of t{Proposition~\ref{prop:pseudoE}} and t{Proposition~\ref{prop:pseudoE-1}}, the infinitesimal conditions t{(\ref{eq:Y_tang})} and
t{(\ref{eq:Y_tang-1})} for pseudosymmetries and $r$-pseudosymmetries are linear
only with respect to $\varphi $ and $\Phi $. Thus, in general the determination
of pseudosymmetries and $r$-pseudosymmetries of
$\mathcal{V }$ is feasible if one already knows some
$\mu $ and $\gamma $. It follows that, for the practical application of
the procedure described above, it may be necessary to already have available
the possible differential coverings $\mathcal{V}$ together with some corresponding
$1$-forms $\mu $ and $\gamma $, defined in $\mathcal{V}$.

Based on the results of the previous section, it is therefore clear that
in the case of an equation with two independent variables a particularly
advantageous situation is that of equations admitting ZCRs: indeed, starting
from these representations, the Riccati-type differentiable coverings together
with associated horizontally closed $1$-forms $\mu $ remain automatically
determined and one can directly pass to search for corresponding pseudosymmetries.

Unfortunately, one does not have yet an analogous result for $r$-pseudosymmetries.
Hence, calculating the possible $r$-pseudosymmetries admitted by the Riccati-type
differentiable coverings, associated with ZCRs, remains rather complicated;
as is already the case without considering any differentiable covering
so much so that in the literature it is not easy to find significant examples
of $r$-pseudosymmetries of differential equations. Yet we believe that
r-pseudosymmetries should play an important role when the Riccati-type
differentiable covering has dimension greater than $1$, because the larger
number of covering equations in this case may lead to greater restrictions
on the existence of pseudosymmetries; therefore, in the case of
$\ell \times \ell $ matrix-valued ZCR, with $\ell >2$, the more general
$r$-pseudosymmetries should play a key role. For instance, in t{Example~\ref{exa:Tzitzeica2}}, we show that the auto-B\"{a}cklund transformation
of the Tzitzeica equation found in \cite{Brez} can be obtained by using
a $2$-pseudosymmetry of the covering obtained from an
$\mathfrak{sl}_{3}(\mathbb{R})$-valued ZCR. Thus, the computation of
$r$-pseudosymmetries deserves further detailed study. The same is true
for the more general question on the role that pseudosymmetries and
$r$-pseudosymmetries can have in the integrability property of an equation;
in particular, we believe that they can play a fundamental role in studying
B\"{a}cklund transformations in higher dimensions. These and other questions
will be addressed in future works.

%s4.1 #&#
\subsection{Examples}
\label{sec4.1}

In this subsection, we will analyze several examples of equations with
two independent variables for which, starting from a ZCR, it is possible
to determine pseudosymmetries and subsequently B\"{a}cklund transformations,
following the procedure described above. Moreover, in the last example,
we derive an auto-B\"{a}cklund transformation of the Tzitzeica equation
by using a $2$-pseudosymmetry. Taken together, these examples describe
not only the degree of generality of the method, but demonstrate also its
practical applicability. Throughout this section $x_{1}=x$ and
$x_{2}=t$.

%e27 #&#
\begin{example}
\label{exmp27}
We consider here KdV $\mathcal{E}=\{z_{xxx}=-z_{t}-6zz_{x}\}$ and show
how the invariants of a nonlocal pseudosymmetry $\bar{Y}$ can be used to
determine an auto-B\"{a}cklund transformation of this equation. As already
shown in t{Example~\ref{example_kdv_lax1_cont}}, $\mathcal{E}$ admits the
nonlocal conservation law
$\mu =\rho dx+\left (z_{x}+\left (4\lambda -2z\right )\rho \right )dt$
in the $1$-dimensional differentiable covering
$\mathcal{Y}^{(\infty )}\rightarrow \mathcal{E}^{(\infty )}$ defined by
\begin{equation*}
\mathcal{Y}:\;\left \{
\begin{array}{l}
\rho _{x}=-\rho ^{2}-z-\lambda ,
\vspace{5pt}
\\
\rho _{t}=-2(2\lambda -z)\rho ^{2}-2z_{x}\rho +z_{xx}+2z^{2}-2
\lambda z-4\lambda ^{2},
\vspace{5pt}
\\
z_{xxx}=-z_{t}-6zz_{x}.
\end{array}
\right .
\end{equation*}
We are searching for a regular $\mathcal{C}$-morphism from
$\mathcal{Y}$ to $\mathcal{E}'$, with $\mathcal{E}'$ being a copy of
$\mathcal{E}$, defined by a system
$\{x'=\xi ^{1},\;t'=\xi ^{2},\;z'=\nu \}$ such that
$\xi ^{1},\xi ^{2},\nu $ are invariants of a pseudosymmetry $Y$ of
$\mathcal{Y}$.

First, by searching for pseudosymmetries
$\bar{Y}:=\left .Y\right |_{\mathcal{Y}^{(\infty )}}$ of
$\mathcal{Y}$ determined by the pseudoprolongations $Y$ of vector fields
of the form
\begin{equation*}
a_{1}(x,t,z,\rho )\partial _{x}+a_{2}(x,t,z,\rho )\partial _{t}+b^{1}(x,t,z,
\rho )\partial _{z}+b^{2}(x,t,z,\rho )\partial _{\rho},
\end{equation*}
relatively to conservation laws $c\,\mu $, with
$c\in \mathbb{R}-\{0\}$, one finds that the most general pseudosymmetry
of this type with $b^{1}\neq 0$ is a constant multiple of
\begin{equation*}
\bar{Y}=\rho \partial _{z}-\frac{1}{4}\partial _{\rho}+\ldots \,,
\end{equation*}
with $c=2$. Such a pseudosymmetry of $\mathcal{Y}$ can be seen as a nonlocal
pseudosymmetry of $\mathcal{E}$, with respect to the differentiable covering
defined by $\mathcal{Y}$. Thus, being $\{x,\;t,\;z+2\rho ^{2}\}$ a basic
system of zeroth order invariants of $\bar{Y}$, one has that the auto-B\"{a}cklund
transformation of KdV discussed in t{Example~\ref{exa:KdV_B11}}
\begin{equation*}
x'=x,\qquad t'=t,\qquad z'=-z-2\rho ^{2}{-2\lambda},
\end{equation*}
is of the type described above, i.e., it is provided by the invariants
of the pseudosymmetry $\bar{Y}$.
\end{example}

%e28 #&#
\begin{example}
%%LEAP%%%\label{exmp28}
\label{exa:DT_KdV}%
We show here that also the Darboux transformation discussed in t{Example~\ref{exa:KdV_B21}} is provided by the invariants of a nonlocal pseudosymmetry
$\bar{Y}$. To this end one can consider the $2$-dimensional differentiable
covering of KdV $\mathcal{E}=\{z_{xxx}=-z_{t}-6zz_{x}\}$ defined by
\begin{equation*}
\mathcal{Y}:\;\left \{
\begin{array}{l}
\hat{\rho}_{x}=-\hat{\rho}^{2}-z-\hat{\lambda},
\vspace{5pt}
\\
\hat{\rho}_{t}=-2(2\hat{\lambda}-z)\hat{\rho}^{2}-2z_{x}\hat{\rho}+z_{xx}+2z^{2}-2
\hat{\lambda}z-4\hat{\lambda}^{2},
\vspace{5pt}
\\
\rho _{x}=-\rho ^{2}-z-\lambda ,
\vspace{5pt}
\\
\rho _{t}=-2(2\lambda -z)\rho ^{2}-2z_{x}\rho +z_{xx}+2z^{2}-2
\lambda z-4\lambda ^{2},
\vspace{5pt}
\\
z_{xxx}=-z_{t}-6zz_{x},
\end{array}
\right .
\end{equation*}
which coincides with t{(\ref{eq:cov2_KdV-1})} and is obtained by doubling
the Riccati-type system of previous example. Hence, being
$\hat{\mu} =\hat{\rho} dx+\left (z_{x}+\left (4\hat{\lambda} -2z\right )\hat{\rho} \right )dt$
still a conservation law for $\mathcal{Y}$, one can search for pseudosymmetries
$\bar{Y}:=\left .Y\right |_{\mathcal{Y}^{(\infty )}}$ of
$\mathcal{Y}$ determined by the pseudoprolongations $Y$ of vector fields
of the form
\begin{eqnarray*}[ll]
a_{1}(x,t,z,\rho ,\hat{\rho})\partial _{x}+a_{2}(x,t,z,\rho ,
\hat{\rho})\partial _{t}+b^{1}(x,t,z,\rho ,\hat{\rho})\partial _{z}+b^{2}(x,t,z,
\rho ,\hat{\rho})\partial _{\rho}\\
\quad { }+b^{3}(x,t,z,\rho ,\hat{\rho})
\partial _{\hat{\rho}}\;,
\end{eqnarray*}
relatively to $2\,\hat{\mu} $. In this case, when
$\hat{\lambda}\neq \lambda $ and $\hat{\rho}\neq \rho $, one finds that
the most general pseudosymmetry of this type with $b^{1}\neq 0$ is a constant
multiple of
\begin{equation*}
\bar{Y}=\hat{\rho}\partial _{z}-\frac{1}{4}\left (1+
\frac{(\rho -\hat{\rho})^{2}}{\lambda -\hat{\lambda}}\right )
\partial _{\rho}-\frac{1}{4}\partial _{\hat{\rho}}+\ldots \,.
\end{equation*}
Such a pseudosymmetry of $\mathcal{Y}$ can be seen as a nonlocal pseudosymmetry
of $\mathcal{E}$, with respect to the differentiable covering defined by
$\mathcal{Y}$. Thus, being
$\{x,\;t,\;z+2\hat{\rho}^{2},\hat{\rho}+(\hat{\lambda}-\lambda )/(
\hat{\rho}-\rho )\,\}$ a basic system of zeroth order invariants of
$\bar{Y}$, one has that the Darboux transformation discussed in t{Example~\ref{exa:KdV_B21}}
\begin{equation*}
x'=x,\qquad t'=t,\qquad z'=-z-2\hat{\rho}^{2}{-2
\hat{\lambda}},\qquad \rho '=-
\frac{\hat{\lambda}-\lambda}{\hat{\rho}-\rho}-\hat{\rho},
\end{equation*}
is provided by the invariants of the pseudosymmetry $\bar{Y}$. In this
case the transformation describes a regular $\mathcal{C}$-morphism from
$\mathcal{Y}$ to $\mathcal{E}'$, where by $\mathcal{E}'$ now we mean the
system
\begin{equation*}
\mathcal{E}':\;\left \{
\begin{array}{l}
\rho '_{x}=-\rho '{}^{2}-z'-\lambda ,
\vspace{5pt}
\\
\rho '_{t}=-2(2\lambda -z')\rho '{}^{2}-2z'_{x}\rho '+z'_{xx}+2z'{}^{2}-2
\lambda z'-4\lambda ^{2},
\vspace{5pt}
\\
z'_{xxx}=-6z'z'_{x}-z'_{t}.
\end{array}
\right .
\end{equation*}
\end{example}

%e29 #&#
\begin{example}
\label{exmp29}
We consider here sine-Gordon equation
$\mathcal{E}=\{z_{xt}=\sin z\}$ and show how the invariants of a nonlocal
pseudosymmetry $\bar{Y}$ can be used to determine an auto-B\"{a}cklund
transformation of this equation. To this end we consider the following
ZCR
\begin{equation*}
\alpha =\frac{1}{2}\left (
\begin{array}{l@{\quad}l}
\eta \  & -z_{x}
\vspace{5pt}
\\
z_{x}\qquad & -{\displaystyle \eta}
\end{array}
\right )dx+\frac{1}{2\eta}\left (
\begin{array}{l@{\quad}l}
\cos \,z
\vspace{5pt}
& \sin \,z
\\
\sin \,z\qquad & -\cos \,z
\end{array}
\right )dt,\qquad \eta \in \mathbb{R}-\{0\}.
\end{equation*}
In this case, the procedure described in Section~\ref{subsec:Riccati-type-differentiable-cove}, with
$\rho =v^{1}/v^{2}$, provides the $1$-dimensional differentiable covering
$\mathcal{Y}^{(\infty )}\rightarrow \mathcal{E}^{(\infty )}$ defined by
\begin{equation*}
\mathcal{Y}:\;\left \{
\begin{array}{l}
\rho _{x}=-\frac{1}{2}z_{x}\rho ^{2}+\eta \rho -\frac{1}{2}\,z_{x},
\vspace{5pt}
\\
\rho _{t}=\frac{1}{2\eta}\sin z\,\left (1-\rho ^{2}\right )+
\frac{1}{\eta}\cos z\,\rho ,
\vspace{5pt}
\\
z_{xt}=\sin z,
\end{array}
\right .
\end{equation*}
together with the nonlocal conservation law
\begin{equation*}
\mu =\frac{1}{2}\left (\rho z_{x}-\eta \right )\,dx+\frac{1}{2\eta}
\left (\rho \,\sin z-\cos z\right )\,dt.
\end{equation*}
We are searching for a regular $\mathcal{C}$-morphism from
$\mathcal{Y}$ to $\mathcal{E}'$, with $\mathcal{E}'$ being a copy of
$\mathcal{E}$, defined by a system
$\{x'=\xi ^{1},\;t'=\xi ^{2},\;z'=\nu \}$ such that
$\xi ^{1},\xi ^{2},\nu $ are invariants of a pseudosymmetry $Y$ of
$\mathcal{Y}$.

First, by searching for pseudosymmetries
$\bar{Y}:=\left .Y\right |_{\mathcal{Y}^{(\infty )}}$ of
$\mathcal{Y}$ determined by the pseudoprolongations $Y$ of vector fields
of the form
\begin{equation*}
a_{1}(x,t,z,\rho )\partial _{x}+a_{2}(x,t,z,\rho )\partial _{t}+b^{1}(x,t,z,
\rho )\partial _{z}+b^{2}(x,t,z,\rho )\partial _{\rho},
\end{equation*}
relatively to conservation laws $c\,\mu $, with
$c\in \mathbb{R}-\{0\}$, one finds that the most general pseudosymmetry
of this type with $b^{1}\neq 0$ is a constant multiple of
\begin{equation*}
\bar{Y}=\left (\rho ^{2}+1\right )\partial _{z}-\frac{\left(\rho ^{2}+1\right)^{2}}{4}
\,\partial _{\rho}+\ldots \,,
\end{equation*}
with $c=2$. Such a pseudosymmetry of $\mathcal{Y}$ can be seen as a nonlocal
pseudosymmetry of $\mathcal{E}$, with respect to the differentiable covering
defined by $\mathcal{Y}$. Thus, being
$\{x,\,t,\,z+4\,\arctan \rho \}$ a basic system of zeroth order invariants
of $\bar{Y}$, one is naturally lead to consider
$\xi ^{1}=x,\,\xi ^{2}=t$ and by taking $\nu =z+4\,\arctan \rho $ one obtains
$z'_{xt}-\sin z'=0$, whenever $\{z,\rho \}$ is a solution of
$\mathcal{Y}$. Hence, one has the auto-B\"{a}cklund transformation of
$\mathcal{E}$
\begin{equation*}
x'=x,\qquad t'=t,\qquad z'=z+4\,\arctan \rho ,
\end{equation*}
provided by the invariants of the pseudosymmetry $\bar{Y}$.
\end{example}

%e30 #&#
\begin{example}
\label{exmp30}
We consider here the short-pulse equation
$\mathcal{E}=\{z_{xt}=\frac{1}{2}z^{2}z_{xx}+zz_{x}^{2}+z\}$ and show how
the invariants of a nonlocal pseudosymmetry $\bar{Y}$ can be used to determine
an auto-B\"{a}cklund transformation of this equation. To this end we consider
the following ZCR
\begin{equation*}
\alpha =\frac{1}{4\eta}\left (
\begin{array}{l@{\quad}l}
1\quad & z_{x}
\vspace{5pt}
\\
z_{x} & -1
\end{array}
\right )dx+\frac{1}{8\eta}\left (
\begin{array}{l@{\quad}l}
8\eta ^{2}+z^{2} & z\left (zz_{x}-4\,\eta \right )
\vspace{5pt}
\\
z\left (zz_{x}+4\,\eta \right )\quad & -8\eta ^{2}-z^{2}
\end{array}
\right )dt,\qquad \eta \in \mathbb{R}-\{0\}.
\end{equation*}
In this case, the procedure described in Section~\ref{subsec:Riccati-type-differentiable-cove}, with
$\rho =v^{1}/v^{2}$, provides the $1$-dimensional differentiable covering
$\mathcal{Y}^{(\infty )}\rightarrow \mathcal{E}^{(\infty )}$ defined by
\begin{equation*}
\mathcal{Y}:\;\left \{
\begin{array}{l}
{\displaystyle \rho _{x}=-\frac{z_{x}}{4\eta}\rho ^{2}+
\frac{1}{2\eta}\,\rho +\frac{z_{x}}{4\eta}},
\vspace{10pt}
\\
{\displaystyle \rho _{t}=-
\frac{z\left (zz_{x}+4\,\eta \right )}{8\eta}\rho ^{2}+
\frac{\left (8\,\eta ^{2}+z^{2}\right )}{4\eta}\rho -
\frac{z\left (-zz_{x}+4\,\eta \right )}{8\eta}},
\vspace{10pt}
\\
{\displaystyle z_{xt}=\frac{1}{2}z^{2}z_{xx}+zz_{x}^{2}+z,}
\end{array}
\right .
\end{equation*}
together with the nonlocal conservation law
\begin{equation*}
\mu =\frac{1}{4\eta}\left (z_{x}\rho -1\right )dx+\frac{1}{8\eta}
\left (z^{2}z_{x}\,\rho +4\eta \,z\,\rho -z^{2}-8\,\eta ^{2}\right )dt.
\end{equation*}
We are searching for a regular $\mathcal{C}$-morphism from
$\mathcal{Y}$ to $\mathcal{E}'$, with $\mathcal{E}'$ being a copy of
$\mathcal{E}$, defined by a system
$\{x'=\xi ^{1},\;t'=\xi ^{2},\;z'=\nu \}$ such that
$\xi ^{1},\xi ^{2},\nu $ are invariants of a pseudosymmetry $Y$ of
$\mathcal{Y}$.

First, by searching for pseudosymmetries
$\bar{Y}:=\left .Y\right |_{\mathcal{Y}^{(\infty )}}$ of
$\mathcal{Y}$ determined by the pseudoprolongations $Y$ of vector fields
of the form
\begin{equation*}
a_{1}(x,t,z,\rho )\partial _{x}+a_{2}(x,t,z,\rho )\partial _{t}+b^{1}(x,t,z,
\rho )\partial _{z}+b^{2}(x,t,z,\rho )\partial _{\rho},
\end{equation*}
relatively to conservation laws $c\,\mu $, with
$c\in \mathbb{R}-\{0\}$, one finds that the most general pseudosymmetry
of this type with $b^{1}\neq 0$ is a constant multiple of
\begin{equation*}
\bar{Y}=-2\rho \partial _{x}+\left (\rho ^{2}-1\right )\partial _{z}-
\frac{1}{8\eta}\left (\rho ^{2}+1\right )^{2}\partial _{\rho}+\ldots
\,,
\end{equation*}
with $c=2$. Such a pseudosymmetry of $\mathcal{Y}$ can be seen as a nonlocal
pseudosymmetry of $\mathcal{E}$, with respect to the differentiable covering
defined by $\mathcal{Y}$. Thus, being
$\{x+8\eta /(\rho ^{2}+1),\,t,z-8\eta \rho /(\rho ^{2}+1)\}$ a basic system
of zeroth order invariants of $\bar{Y}$, one is naturally lead to consider
$\xi ^{1}=x+8\eta /(\rho ^{2}+1),\,\xi ^{2}=t$ and by taking
$\nu =-z+8\eta \rho /(\rho ^{2}+1)$ one obtains
$z'_{x't'}-\frac{1}{2}z'{}^{2}z'_{x'x'}-z'z'{}_{x'}^{2}-z'=0$, whenever
$\{z,\rho \}$ is a solution of $\mathcal{Y}$. Hence, one has the auto-B\"{a}cklund
transformation of $\mathcal{E}$
\begin{equation*}
x'=x+8\frac{\eta}{\rho ^{2}+1},\qquad t'=t,\qquad z'=-z+8\eta
\frac{\rho}{\rho ^{2}+1},
\end{equation*}
provided by the invariants of the pseudosymmetry $\bar{Y}$.
\end{example}

%e31 #&#
\begin{example}
\label{exmp31}
We consider here Camassa-Holm equation
$\mathcal{E}=\{z_{txx}=3\,z\,z_{x}-z\,z_{xxx}-2\,z_{x}\,z_{xx}+z_{t}
\}$ and show how the invariants of a nonlocal pseudosymmetry
$\bar{Y}$ can be used to determine an auto-B\"{a}cklund transformation
of this equation. To this end we consider the following ZCR
\begin{eqnarray*}
\alpha &=&{\displaystyle \frac{1}{2}}\left (
\begin{array}{l@{\quad}l}
0
\vspace{5pt}
& -{\displaystyle \frac{1}{\eta}}
\\
-z+z_{xx}-\eta & \quad 0
\end{array}
\right )\,dx\\
&&{}+{\displaystyle \frac{1}{2}}\left (
\begin{array}{l@{\quad}l}
z_{x}
\vspace{5pt}
& {\displaystyle \frac{1}{\eta}}(z-2\,\eta )
\\
z^{2}-z\,z_{xx}-\eta \,z-2\eta ^{2} & \quad -z_{x}
\end{array}
\right )dt,\quad  {\eta \in \mathbb{R}-\{0\}.}%
\end{eqnarray*}
In this case, the procedure described in Section~\ref{subsec:Riccati-type-differentiable-cove}, with
$\rho =v^{2}/v^{1}$, provides the $1$-dimensional differentiable covering
$\mathcal{Y}^{(\infty )}\rightarrow \mathcal{E}^{(\infty )}$ defined by
%
%e41 #&#
\begin{equation}
\mathcal{Y}:\;\left \{ {\displaystyle
\begin{array}{l}
\rho _{x}={\displaystyle \frac{1}{2\eta}\rho ^{2}+\frac{z_{xx}}{2}-
\frac{\eta+z}{2}},
\vspace{10pt}
\\
\rho _{t}={\displaystyle -\frac{\left (-2\,\eta +z\right )}{2\,\eta}
\rho ^{2}-z_{x}\,\rho -\frac{z\,z_{xx}}{2}-
\frac{2\,\eta ^{2}+\eta \, z-z^{2}}{2}},
\vspace{10pt}
\\
z_{txx}=3\,z\,z_{x}-z\,z_{xxx}-2\,z_{x}\,z_{xx}+z_{t},
\end{array}
}\right .
\label{eq:Riccati-CH}
%%LEAP%%%\label{eq41}
\end{equation}
together with the nonlocal conservation law
%
%e42 #&#
\begin{equation}
\mu =-\frac{\rho}{2\eta}\,dx+\left (\frac{z_{x}}{2}-\left (1-
\frac{z}{2\eta}\right )\rho \right )\,dt.
\label{eq:CL_CH}
%%LEAP%%%\label{eq42}
\end{equation}
We are searching for a regular $\mathcal{C}$-morphism from
$\mathcal{Y}$ to $\mathcal{E}'$, with $\mathcal{E}'$ being a copy of
$\mathcal{E}$, defined by a system
$\{x'=\xi ^{1},\;t'=\xi ^{2},\;z'=\nu \}$ such that
$\xi ^{1},\xi ^{2},\nu $ are invariants of a pseudosymmetry $Y$ of
$\mathcal{Y}$.

In this case, if one searches for pseudosymmetries
$\bar{Y}:=\left .Y\right |_{\mathcal{Y}^{(\infty )}}$ of
$\mathcal{Y}$ determined by the pseudoprolongations $Y$ of (relative) vector
fields of the form
\begin{equation*}
a_{1}(x,t,z,\rho )\partial _{x}+a_{2}(x,t,z,\rho )\partial _{t}+b^{1}(x,t,z,z_{x},
\rho )\partial _{z}+b^{2}(x,t,z,z_{x},\rho )\partial _{\rho},
\end{equation*}
relatively to conservation laws $c\,\mu $, with
$c\in \mathbb{R}-\{0\}$, one finds that the most general pseudosymmetry
of this type with $b^{1}\neq 0$ is a constant multiple of
\begin{equation*}
\bar{Y}=-\partial _{x}+(2\rho -z_{x})\,\partial _{z}+
\frac{\left (\eta ^{2}-\rho ^{2}\right )}{2\eta}\partial _{\rho}+
\ldots \,,
\end{equation*}
with $c=2$. Such a pseudosymmetry of $\mathcal{Y}$ can be seen as a nonlocal
pseudosymmetry of $\mathcal{E}$, with respect to the differentiable covering
defined by $\mathcal{Y}$. Thus, being
\begin{eqnarray*}[ll]
\xi ^{1}=x-\ln \left (\frac{\rho -\eta}{\rho +\eta}\right ),\quad
\xi ^{2}=t,\quad \nu ^{1}=\left (
\frac{\rho ^{2}+\eta ^{2}}{\rho ^{2}-\eta ^{2}}\right )z+
\frac{2\,\eta \,\rho}{\rho ^{2}-\eta ^{2}}\,z_{x},\\ \nu ^{2}=
\frac{2\,\eta \,\rho}{\rho ^{2}-\eta ^{2}}z+\left (
\frac{\rho ^{2}+\eta ^{2}}{\rho ^{2}-\eta ^{2}}\right )\,z_{x}-2\,\rho,
\end{eqnarray*}
a basic system of first-order order invariants of $\bar{Y}$, one finds
that in this case the system $\mathcal{Q}$ (the factorisation of
$\mathcal{E}$ by $\bar{Y}$) reads
\begin{equation*}
\mathcal{Q}:\quad \left \{
\begin{array}{l}
\nu _{\xi ^{1}}^{1}=\nu ^{2},
\vspace{10pt}
\\
\nu _{\xi ^{1}\xi ^{2}}^{2}=\left (-6\,\eta +3\,\nu ^{1}-2\,\nu _{
\xi ^{1}}^{2}\right )\nu ^{2}+\left (2\,\eta -\nu ^{1}\right )\nu _{
\xi ^{1}\xi ^{1}}^{2}+\nu _{\xi ^{2}}^{1}.
\vspace{10pt}
\end{array}
\right .
\end{equation*}
Therefore, by taking
%
%e43 #&#
\begin{equation}
x'=x-\ln \left (\frac{\rho -\eta}{\rho +\eta}\right ),\qquad t'=t,
\qquad z'=\nu ^{1}-2\eta =\left (
\frac{\rho ^{2}+\eta ^{2}}{\rho ^{2}-\eta ^{2}}\right )z+
\frac{2\,\eta \,\rho}{\rho ^{2}-\eta ^{2}}\,z_{x}{-2
\eta},
\label{eq:ABT-CH}
%%LEAP%%%\label{eq43}
\end{equation}
one gets that
$z'_{t'x'x'}-3\,z'\,z'_{x'}+z'\,z'_{x'x'x'}+2\,z'_{x'}\,z'_{x'x'}-z'_{t'}=0$,
whenever $\{z,\rho \}$ is a solution of $\mathcal{Y}$. Hence, by means
of t{(\ref{eq:ABT-CH})}, the invariants of $\bar{Y}$ determine an auto-B\"{a}cklund
transformation of $\mathcal{E}$. This auto-B\"{a}cklund transformation
coincides with that found in \cite{Rasin_Schiff}, by means of a different
approach.
\end{example}

%e32 #&#
\begin{example}
\label{exmp32}
We show here that, by following an approach similar to that used in t{Example~\ref{exa:DT_KdV}}, one can also find a Darboux transformation for the Camassa-Holm
$\mathcal{E}=\{z_{txx}=3\,z\,z_{x}-z\,z_{xxx}-2\,z_{x}\,z_{xx}+z_{t}
\}$. To this end one can consider the $2$-dimensional differentiable covering
of $\mathcal{E}$ defined now by
\begin{equation*}
\mathcal{Y}:\;\left \{ {\displaystyle
\begin{array}{l}
{\displaystyle \hat{\rho}{}_{x}=\frac{1}{2\hat{\eta}}\hat{\rho}^{2}+
\frac{z_{xx}}{2}-
\frac{{\hat{\eta}}+z}{2}},
\vspace{10pt}
\\
{\displaystyle \hat{\rho}{}_{t}=-
\frac{\left (-2\,\hat{\eta}+z\right )}{2\,\hat{\eta}}\hat{\rho}^{2}-z{}_{x}
\,\hat{\rho}-\frac{z\,z{}_{xx}}{2}-
\frac{2\,\hat{\eta}^{2}+\hat{\eta}\,z-z{}^{2}}{2}},
\vspace{10pt}
\\
{\displaystyle \rho _{x}=\frac{1}{2\eta}\rho ^{2}+\frac{z_{xx}}{2}-
\frac{{\eta}+z}{2}},
\vspace{10pt}
\\
{\displaystyle \rho _{t}=-\frac{\left (-2\,\eta +z\right )}{2\,\eta}
\rho ^{2}-z_{x}\,\rho -\frac{z\,z_{xx}}{2}-
\frac{2\,\eta ^{2}+\eta\,z-z^{2}}{2}},
\vspace{10pt}
\\
z_{txx}=3\,z\,z_{x}-z\,z_{xxx}-2\,z_{x}\,z_{xx}+z_{t},
\end{array}
}\right .
\end{equation*}
which is obtained by doubling the Riccati-type system t{(\ref{eq:Riccati-CH})}.
Hence, being t{(\ref{eq:CL_CH})} still a conservation law for the new system
$\mathcal{Y}$, one can search for pseudosymmetries
$\bar{Y}:=\left .Y\right |_{\mathcal{Y}^{(\infty )}}$ of
$\mathcal{Y}$ determined by the pseudoprolongations $Y$ of (relative) vector
fields of the form
\begin{eqnarray*}[ll]
a_{1}(x,t,z,\rho ,\hat{\rho})\partial _{x}+a_{2}(x,t,z,\rho ,
\hat{\rho})\partial _{t}+b^{1}(x,t,z,z_{x},\rho ,\hat{\rho})\partial _{z}+b^{2}(x,t,z,z_{x},
\rho ,\hat{\rho})\partial _{\rho},\\
\quad{}+b^{3}(x,t,z,z_{x},\rho ,\hat{\rho})
\partial _{\hat{\rho}}\;,
\end{eqnarray*}
relatively to $2\,\mu $. In this case, when $\hat{\eta}\neq \eta $ and
$\hat{\rho}\neq \rho $, one finds that the most general pseudosymmetry
of this type with $b^{1}\neq 0$ is a constant multiple of
\begin{equation*}
\bar{Y}={\displaystyle -\partial _{x}+(2\rho -z_{x})\,\partial _{z}+
\frac{\left (\eta ^{2}-\rho ^{2}\right )}{2\eta}\partial _{\rho}+
\frac{\eta ^{2}\hat{\eta}-\eta \,\hat{\eta}^{2}-2\eta \,\rho \,\hat{\rho}+\hat{\eta}\,\rho ^{2}+\eta \,\hat{\rho}^{2}}{2\,\eta \,\left (\eta -\hat{\eta}\right )}
\partial _{\hat{\rho}}+\ldots \,.}
\end{equation*}
Such a pseudosymmetry of $\mathcal{Y}$ can be seen as a nonlocal pseudosymmetry
of $\mathcal{E}$, with respect to the differentiable covering defined by
$\mathcal{Y}$. Thus, being
\begin{equation*}
\begin{array}{l}
{\displaystyle \xi ^{1}=x-\ln \left (\frac{\rho -\eta}{\rho +\eta}
\right ),\quad \xi ^{2}=t,\quad \nu ^{1}=\left (
\frac{\rho ^{2}+\eta ^{2}}{\rho ^{2}-\eta ^{2}}\right )z+
\frac{2\,\eta \,\rho}{\rho ^{2}-\eta ^{2}}\,z_{x},}
\vspace{10pt}
\\
{\displaystyle \nu ^{2}=\frac{2\,\eta \,\rho}{\rho ^{2}-\eta ^{2}}z+
\left (\frac{\rho ^{2}+\eta ^{2}}{\rho ^{2}-\eta ^{2}}\right )\,z_{x}-2\,\rho,
\quad \nu ^{3}=
\frac{\hat{\eta}\,\rho ^{2}-\hat{\eta}\,\rho \hat{\rho}-\eta \,\hat{\eta}\,\left (\eta -\hat{\eta}\right )}{\eta \,\hat{\rho}-\hat{\eta}\,\rho},}
\end{array}
\end{equation*}
a basic system of first-order invariants of $\bar{Y}$, one finds
that in this case
%
%e44 #&#
\begin{eqnarray}[ll]
\nonumber
x'=x-\ln \left (\frac{\rho -\eta}{\rho +\eta}\right ),\qquad t'=t,
\qquad z'=\left (\frac{\rho ^{2}+\eta ^{2}}{\rho ^{2}-\eta ^{2}}
\right )z+\frac{2\,\eta \,\rho}{\rho ^{2}-\eta ^{2}}\,z_{x}{
-2\eta},\\ \rho '=
\frac{\hat{\eta}\,\rho ^{2}-\hat{\eta}\,\rho \hat{\rho}-\eta \,\hat{\eta}\,\left (\eta -\hat{\eta}\right )}{\eta \,\hat{\rho}-\hat{\eta}\,\rho},
\label{eq:ABT-CH-1}
%%LEAP%%%\label{eq44}
\end{eqnarray}
defines a regular $\mathcal{C}$-morphism from $\mathcal{Y}$ to the system
\begin{equation*}
\mathcal{Y}':\;\left \{
\begin{array}{l}
{\displaystyle \rho '_{x'}=\frac{1}{2\hat{\eta}}\rho '{}^{2}+
\frac{z'_{x'x'}}{2}-\frac{{\hat{\eta}}+z'}{2},}
\vspace{10pt}
\\
{\displaystyle \rho '_{t'}=-
\frac{\left (-2\,\hat{\eta} +z'\right )}{2\,\hat{\eta}}\rho '{}^{2}-z'_{x'}\,
\rho '-\frac{z'\,z'_{x'x'}}{2}-
\frac{2\,\hat{\eta} ^{2}+\hat{\eta} z'-z'{}^{2}}{2}},
\vspace{10pt}
\\
z'_{t'x'x'}=3\,z'\,z'_{x'}-z'\,z'_{x'x'x'}-2\,z'_{x'}\,z'_{x'x'}+z'_{t'}.
\end{array}
\right .
\end{equation*}
 \end{example}

%e33 #&#
\begin{example}
\label{exmp33}
We consider here the Harry-Dym equation
$\mathcal{E}=\{z_{t}=z^{3}z_{xxx}\}$ and show how the invariants of a nonlocal
pseudosymmetry $\bar{Y}$ can be used to determine an auto-B\"{a}cklund
transformation of this equation. To this end we consider the following
ZCR
\begin{equation*}
\alpha =\left ({\displaystyle
\begin{array}{l@{\quad}l}
0\qquad & 4\eta
\\
-\frac{\eta}{z^{2}} & 0
\end{array}
}\right )dx+\left ({\displaystyle
\begin{array}{l@{\quad}l}
8\eta ^{2}z_{x}\qquad & -64\eta ^{3}z
\\
2\eta z_{xx}+\frac{16\eta ^{3}}{z}\qquad & -8\eta ^{2}z_{x}
\end{array}
}\right )dt,\qquad \eta \in \mathbb{R}-\{0\}.
\end{equation*}
In this case, the procedure described in Section~\ref{subsec:Riccati-type-differentiable-cove}, with
$\rho =v^{2}/v^{1}$, provides the $1$-dimensional differentiable covering
$\mathcal{Y}^{(\infty )}\rightarrow \mathcal{E}^{(\infty )}$ defined by
\begin{equation*}
\mathcal{Y}:\;\left \{
\begin{array}{l}
{\displaystyle \rho _{x}=-4\eta \rho ^{2}-\frac{\eta}{z^{2}}},
\vspace{5pt}
\\
{\displaystyle \rho _{t}=64\eta ^{3}\,z\,\rho ^{2}-16\eta ^{2}z_{x}
\rho +2\eta z_{xx}+\frac{16\,\eta ^{3}}{z}},
\vspace{5pt}
\\
z_{t}=z^{3}z_{xxx},
\end{array}
\right .
\end{equation*}
together with the nonlocal conservation law
\begin{equation*}
\mu =4\,\eta \,\rho \,dx+\left (-64\eta ^{3}\,\rho \,z+8\eta ^{2}\,z_{x}
\right )\,dt.
\end{equation*}
We are searching for a regular $\mathcal{C}$-morphism from
$\mathcal{Y}$ to $\mathcal{E}'$, with $\mathcal{E}'$ being a copy of
$\mathcal{E}$, defined by a system
$\{x'=\xi ^{1},\;t'=\xi ^{2},\;z'=\nu \}$ such that
$\xi ^{1},\xi ^{2},\nu $ are invariants of a pseudosymmetry $Y$ of
$\mathcal{Y}$.

First, by searching for pseudosymmetries
$\bar{Y}:=\left .Y\right |_{\mathcal{Y}^{(\infty )}}$ of
$\mathcal{Y}$ determined by the pseudoprolongations $Y$ of vector fields
of the form
\begin{equation*}
a_{1}(x,t,z,\rho )\partial _{x}+a_{2}(x,t,z,\rho )\partial _{t}+b^{1}(x,t,z,
\rho )\partial _{z}+b^{2}(x,t,z,\rho )\partial _{\rho},
\end{equation*}
relatively to conservation laws $c\,\mu $, with
$c\in \mathbb{R}-\{0\}$, one finds that the most general pseudosymmetry
of this type with $b^{1}\neq 0$ is a constant multiple of
\begin{equation*}
\bar{Y}=\frac{1}{8\eta}\partial _{x}+z\,\rho \,\partial _{z}-
\frac{\rho ^{2}}{2}\,\partial _{\rho}+\ldots \,,
\end{equation*}
with $c=2$. Such a pseudosymmetry of $\mathcal{Y}$ can be seen as a nonlocal
pseudosymmetry of $\mathcal{E}$, with respect to the differentiable covering
defined by $\mathcal{Y}$. Thus, being
$\{x-1/(4\eta \rho ),\,t,\,z\,\rho ^{2}\}$ a basic system of zeroth order
invariants of $\bar{Y}$, one is naturally lead to consider
$\xi ^{1}=x-1/(4\eta \rho ),\,\xi ^{2}=t$ and taking
$\nu =-1/(4z\rho ^{2})$ one obtains $z'_{t'}-z'{}^{3}z'_{x'x'x'}=0$, whenever
$\{z,\rho \}$ is a solution of $\mathcal{Y}$. Hence, one has the auto-B\"{a}cklund
transformation of $\mathcal{E}$
\begin{equation*}
x'=x-\frac{1}{4\eta \rho},\qquad t'=t,\qquad z'=-
\frac{1}{4z\rho ^{2}},
\end{equation*}
provided by the invariants of the pseudosymmetry $\bar{Y}$.
\end{example}

%e34 #&#
\begin{example}
\label{exmp34}
We consider here the following modification of Harry-Dym (mHD) equation
%
%e45 #&#
\begin{equation}
\mathcal{E}=\left \{ z_{t}=z^{3}z_{xxx}+\left (-4\,k^{2}z^{3}-
\frac{12a^{2}}{z}+b\right )z_{x}\right \}
\label{eq:mHD}
%%LEAP%%%\label{eq45}
\end{equation}
where $a,b,k\in \mathbb{R}$. This is a new integrable equation that includes,
as a particular instance, the equation recently studied in
\cite{GengLiXue,TanWu}. Here we show how the invariants of a nonlocal pseudosymmetry
$\bar{Y}$ can be used to determine an auto-B\"{a}cklund transformation
of t{(\ref{eq:mHD})}. To this end we consider the following ZCR of t{(\ref{eq:mHD})}
\begin{equation*}
\alpha =\left (
\begin{array}{l@{\quad}l}
{\displaystyle \frac{P}{z^{2}}}\qquad & 4\eta
\vspace{5pt}
\\
-{\displaystyle \frac{\eta}{z^{2}}} & {\displaystyle -\frac{P}{z^{2}}}
\end{array}
\right )dx+\left (
\begin{array}{l@{\quad}l}
-2az_{xx}+8\eta ^{2}z_{x}+{\displaystyle \frac{PQ}{z^{3}}}\qquad & {
\displaystyle \frac{4\eta \left (4azz_{x}+Q\right )}{z}}
\vspace{5pt}
\\
2\eta z_{xx}-4\eta kz_{x}-{\displaystyle \frac{\eta Q}{z^{3}}} & 2az_{xx}-8
\eta ^{2}z_{x}-{\displaystyle \frac{PQ}{z^{3}}}
\end{array}
\right )dt,
\end{equation*}
where $\eta \in \mathbb{R}-\{0\}$ and
\begin{equation*}
P:=k\,z^{2}+a,\qquad Q:=8ak\,z^{2}-16\eta ^{2}z^{2}-8a^{2}+bz.
\end{equation*}
In this case, the procedure described in Section~\ref{subsec:Riccati-type-differentiable-cove}, with
$\rho =v^{2}/v^{1}$, provides the $1$-dimensional differentiable covering
$\mathcal{Y}^{(\infty )}\rightarrow \mathcal{E}^{(\infty )}$ defined by
\begin{equation*}
\mathcal{Y}:\;\left \{
\begin{array}{l}
{\displaystyle \rho _{x}=-4\eta \rho ^{2}+\left (-2k-\frac{2a}{z^{2}}
\right )\rho -\frac{\eta}{z^{2}}},
\vspace{5pt}
\\
{\displaystyle \rho _{t}=\left (\left (-32\eta ak+64\eta ^{3}\right )z-16
\eta a\,z_{x}-4\eta b+\frac{32\eta \,a^{2}}{z}\right )\rho ^{2}}
\vspace{5pt}
\\
\qquad +\left (\left (-16a\,k^{2}+32\eta ^{2}k\right )z-16\eta ^{2}z_{x}+4az_{xx}-2bk+{
\displaystyle \frac{32a\,\eta ^{2}}{z}-\frac{2ab}{z^{2}}+
\frac{16a^{3}}{z^{3}}}\right )\rho
\vspace{5pt}
\\
\qquad -4\eta kz_{x}+2\eta z_{xx}+{\displaystyle
\frac{-8a\eta k+16\eta ^{3}}{z}-\frac{b\eta}{z^{2}}+
\frac{8\eta \,a^{2}}{z^{3}}}
\vspace{5pt}
\\
z_{t}=z^{3}z_{xxx}+\left (-4\,k^{2}z^{3}-{\displaystyle
\frac{12a^{2}}{z}}+b\right )z_{x},
\end{array}
\right .
\end{equation*}
together with the nonlocal conservation law
\begin{eqnarray*}
\mu &= & \left (4\eta \rho +k+\frac{a}{z^{2}}\right )\,dx+\left (
\left (-32\left (-ak+2\eta ^{2}\right )\eta z+16\eta az_{x}+4\eta b-{
\displaystyle \frac{32\eta \,a^{2}}{z}}\right )\rho\right .
\vspace{5pt}
\\
&&\left . {}-8\left (-ak{+}2
\eta ^{2}\right )kz {+}8\eta ^{2}z_{x}-2az_{xx}{+}bk{+}{\displaystyle
\frac{-8a^{2}k-8\left (-ak{+}2\eta ^{2}\right )a}{z}{+}\frac{ab}{z^{2}}-
\frac{8a^{3}}{z^{3}}}\right )\,dt.
\end{eqnarray*}
We are searching for a regular $\mathcal{C}$-morphism from
$\mathcal{Y}$ to $\mathcal{E}'$, with $\mathcal{E}'$ being a copy of
$\mathcal{E}$, defined by a system
$\{x'=\xi ^{1},\;t'=\xi ^{2},\;z'=\nu \}$ such that
$\xi ^{1},\xi ^{2},\nu $ are invariants of a pseudosymmetry $Y$ of
$\mathcal{Y}$.

First, by searching for pseudosymmetries
$\bar{Y}:=\left .Y\right |_{\mathcal{Y}^{(\infty )}}$ of
$\mathcal{Y}$ determined by the pseudoprolongations $Y$ of vector fields
of the form
\begin{equation*}
a_{1}(x,t,z,\rho )\partial _{x}+a_{2}(x,t,z,\rho )\partial _{t}+b^{1}(x,t,z,
\rho )\partial _{z}+b^{2}(x,t,z,\rho )\partial _{\rho},
\end{equation*}
relatively to conservation laws $c\,\mu $, with
$c\in \mathbb{R}-\{0\}$, one finds that the most general pseudosymmetry
of this type with $b^{1}\neq 0$ is a constant multiple of
\begin{equation*}
\bar{Y}=\left (\frac{a\rho}{4\eta ^{2}}+\frac{1}{8\eta}\right )
\partial _{x}+\left (\frac{a\rho ^{2}}{\eta}+\rho +\frac{k}{4\eta}
\right )z\partial _{z}+
\frac{\left (-2a\rho -\eta \right )\left (2\eta \rho +k\right )\rho}{4\eta ^{2}}
\partial _{\rho}+\ldots \,,
\end{equation*}
with $c=2$. Such a pseudosymmetry of $\mathcal{Y}$ can be seen as a nonlocal
pseudosymmetry of $\mathcal{E}$, with respect to the differentiable covering
defined by $\mathcal{Y}$.

From now on, it is convenient to distinguish the case where $k$ is zero
from that where $k$ is non-zero. In fact, up to rescaling $z$ and
$x$, one can always reduce oneself to the cases: \textbf{(i)} $k=0$;
\textbf{(ii)} $k=1$. In particular, case (i) could be further subdivided
into subcases $b=0$ and $b=1$, although we will not do so here.

\textbf{Case (i).} In this case, being
$\{x-1/(4\eta \rho ),\,t,\,\eta z\,\rho ^{2}/(2a\rho +\eta )\}$ a basic
system of zeroth order invariants of $\bar{Y}$, one is naturally lead to
consider $\xi ^{1}=x-1/(4\eta \rho ),\,\xi ^{2}=t$ and taking
$\nu =-(2a\rho +\eta )/(4\eta z\,\rho ^{2})$ one obtains
$z'_{t'}-z'{}^{3}z'_{x'x'x'}+\left (12a^{2}/z'-b\right )z'_{x'}=0$, whenever
$\{z,\rho \}$ is a solution of $\mathcal{Y}$. Hence, one has the auto-B\"{a}cklund
transformation of $\mathcal{E}$
\begin{equation*}
x'=x-\frac{1}{4\eta \rho},\qquad t'=t,\qquad z'=-
\frac{2a\rho +\eta}{4\eta z\,\rho ^{2}},
\end{equation*}
provided by the invariants of the pseudosymmetry $\bar{Y}$, when
$k=0$.

\textbf{Case (ii).} In this case, being
$\{x-\ln ((2\eta \rho +1)/\rho )/2,\,t,\,z(2\eta \rho ^{2}+\rho )/(2a
\rho +\eta )\}$ a basic system of zeroth order invariants of
$\bar{Y}$, one is naturally lead to consider
$\xi ^{1}=x-\ln ((2\eta \rho +1)/\rho )/2,\,\xi ^{2}=t$ and taking
$\nu =-(2a\rho +\eta )/(2z(2\eta \rho ^{2}+\rho ))$ one obtains
$z'_{t'}-z'{}^{3}z'_{x'x'x'}+\left (4z'{}^{3}+12a^{2}/z'-b\right )z'_{x'}=0$,
whenever $\{z,\rho \}$ is a solution of $\mathcal{Y}$. Hence, one has the
auto-B\"{a}cklund transformation of $\mathcal{E}$
\begin{equation*}
x'=x-\frac{1}{2}\ln \left (\frac{2\eta \rho +1}{\rho}\right ),\qquad t'=t,
\qquad z'=-\frac{2a\rho +\eta}{z(4\eta \rho ^{2}+2\rho )},
\end{equation*}
provided by the invariants of the pseudosymmetry $\bar{Y}$, when
$k=1$.
\end{example}

%e35 #&#
\begin{example}
\label{exmp35}
We consider here the coupled Schr\"{o}dinger system
$\mathcal{E}=\{z_{t}^{1}=2\left (z^{1}\right )^{2}z^{2}+z_{xx}^{1},\,z_{t}^{2}=-2z^{1}
\left (z^{2}\right )^{2}-z_{xx}^{2}\}$ and show how the invariants of a
nonlocal pseudosymmetry $\bar{Y}$ can be used to determine an auto-B\"{a}cklund
transformation of $\mathcal{E}$. To this end we consider the following
ZCR
\begin{equation*}
\alpha =\left (
\begin{array}{c@{\quad}c}
\eta \qquad & -z^{2}
\vspace{5pt}
\\
z^{1} & -\eta
\end{array}
\right )dx+\left (
\begin{array}{c@{\quad}c}
-2\eta ^{2}-z^{1}\,z^{2}\qquad & 2\eta z^{2}+z_{x}^{2}
\vspace{5pt}
\\
-2\eta z^{1}+z_{x}^{1} & 2\eta ^{2}+z^{1}z^{2}
\end{array}
\right )dt,\qquad \eta \in \mathbb{R}.
\end{equation*}
In this case, the procedure described in Section~\ref{subsec:Riccati-type-differentiable-cove}, with
$\rho =v^{2}/v^{1}$, provides the $1$-dimensional differentiable covering
defined by the system
\begin{equation*}
\left \{
\begin{array}{l}
\rho _{x}=z^{2}\rho ^{2}-2\eta \rho +z^{1},
\vspace{5pt}
\\
\rho _{t}=\left (-2\eta \,z^{2}-z_{x}^{2}\right )\rho ^{2}+\left (2z^{1}z^{2}+4
\eta ^{2}\right )\rho -2\eta z^{1}+z_{x}^{1},
\vspace{5pt}
\\
z_{t}^{1}=2\left (z^{1}\right )^{2}z^{2}+z_{xx}^{1},
\vspace{5pt}
\\
\,z_{t}^{2}=-2z^{1}\left (z^{2}\right )^{2}-z_{xx}^{2},
\end{array}
\right .
\end{equation*}
together with the nonlocal conservation law
\begin{equation*}
\mu =\left (-z^{2}\rho +\eta \right )\,dx+\left (2\eta z^{2}\rho -z^{1}z^{2}+z_{x}^{2}
\rho -2\eta ^{2}\right )\,dt.
\end{equation*}
Here we double above covering, by considering the covering
$\mathcal{Y}^{(\infty )}\rightarrow \mathcal{E}^{(\infty )}$ defined by
the system
\begin{equation*}
\mathcal{Y}:\;\left \{
\begin{array}{l}
\hat{\rho}_{x}=z^{2}\hat{\rho}^{2}-2\hat{\eta}\hat{\rho}+z^{1},
\vspace{5pt}
\\
\hat{\rho}_{t}=\left (-2\hat{\eta}\,z^{2}-z_{x}^{2}\right )\hat{\rho}^{2}+
\left (2z^{1}z^{2}+4\hat{\eta}^{2}\right )\hat{\rho}-2\hat{\eta}z^{1}+z_{x}^{1},
\vspace{5pt}
\\
\rho _{x}=z^{2}\rho ^{2}-2\eta \rho +z^{1},
\vspace{5pt}
\\
\rho _{t}=\left (-2\eta \,z^{2}-z_{x}^{2}\right )\rho ^{2}+\left (2z^{1}z^{2}+4
\eta ^{2}\right )\rho -2\eta z^{1}+z_{x}^{1},
\vspace{5pt}
\\
z_{t}^{1}=2\left (z^{1}\right )^{2}z^{2}+z_{xx}^{1},
\vspace{5pt}
\\
\,z_{t}^{2}=-2z^{1}\left (z^{2}\right )^{2}-z_{xx}^{2}.
\end{array}
\right .
\end{equation*}
Hence, being $\mu $ still a conservation law for $\mathcal{Y}$, one can
search for pseudosymmetries
$\bar{Y}:=\left .Y\right |_{\mathcal{Y}^{(\infty )}}$ of
$\mathcal{Y}$ determined by the pseudoprolongations $Y$ of vector fields
of the form
\begin{eqnarray*}[ll]
a_{1}(x,t,z^{1},z^{2},\rho ,\hat{\rho})\partial _{x}+a_{2}(x,t,z^{1},z^{2},
\rho ,\hat{\rho})\partial _{t}+b^{11}(x,t,z^{1},z^{2},\rho ,
\hat{\rho})\partial _{z^{1}}\\
\qquad +b^{12}(x,t,z^{1},z^{2},\rho ,\hat{\rho})
\partial _{z^{2}}
+b^{21}(x,t,z^{1},z^{2},\rho ,\hat{\rho})\partial _{\rho},+b^{22}(x,t,z^{1},z^{2},
\rho ,\hat{\rho})\partial _{\hat{\rho}}\;,
\end{eqnarray*}
relatively to $c\,\mu $, with $c\in \mathbb{R}-\{0\}$. In this case, when
$\hat{\eta}\neq \eta $ and $\hat{\rho}\neq \rho $, one finds that the most
general pseudosymmetry of this type with $b^{11}b^{12}\neq 0$ is a constant
multiple of
\begin{equation*}
\bar{Y}=\rho ^{2}\,\partial _{z^{1}}-\partial _{z^{2}}-
\frac{\left (\rho -\hat{\rho}\right )^{2}}{2(\eta -\hat{\eta})}
\partial _{\hat{\rho}}+\ldots \,,
\end{equation*}
with $c=2$. Such a pseudosymmetry of $\mathcal{Y}$ can be seen as a nonlocal
pseudosymmetry of $\mathcal{E}$, with respect to the differentiable covering
defined by $\mathcal{Y}$. Thus, being
\begin{equation*}
\xi ^{1}=x,\quad \xi ^{2}=t,\quad \nu ^{1}=
\frac{((2\eta -2\hat{\eta})\rho \hat{\rho}+z^{1}\left (\rho -\hat{\rho}\right )}{\rho -\hat{\rho}},
\quad \nu ^{2}=
\frac{(\rho -\hat{\rho})z^{2}-2\eta +2\hat{\eta}}{\rho -\hat{\rho}},
\quad \nu ^{3}=\rho ,
\end{equation*}
a basic system of zeroth-order invariants of $\bar{Y}$, one finds that
in this case
%
%e46 #&#
\begin{equation}
x'=x,\qquad t'=t,\qquad z^{1}{}'=
\frac{((2\eta -2\hat{\eta})\rho \,\hat{\rho}+z^{1}\left (\rho -\hat{\rho}\right )}{\rho -\hat{\rho}},
\qquad z^{2}{}'=
\frac{\left (\rho -\hat{\rho}\right )z^{2}-2\eta +2\hat{\eta}}{\rho -\hat{\rho}},
\label{eq:ABT-CH-1-1}
%%LEAP%%%\label{eq46}
\end{equation}
defines a regular $\mathcal{C}$-morphism from $\mathcal{Y}$ to
$\mathcal{E }$.
\end{example}

%e36 #&#
\begin{example}
%%LEAP%%%\label{exmp36}
\label{exa:Tzitzeica2}%
We consider here Tzitzeica equation
$\mathcal{E}=\{z_{xt}=-{{\mathrm{e}}^{-2\,z}}+{{\mathrm{e}}^{z}}\}$. As already shown
in t{Example~\ref{exa:Tzitzeica1}}, $\mathcal{E}$ admits the nonlocal conservation
law
$\mu =\lambda{\rho ^{2}}\,dx+\frac{1}{\lambda}e^{z}{\rho ^{1}}\,dt$ in the $1$-dimensional differentiable covering
$\mathcal{Y}^{(\infty )}\rightarrow \mathcal{E}^{(\infty )}$ defined by
\begin{equation*}
\mathcal{Y}:\;\left \{
\begin{array}{l}
\rho _{x}^{1}=-\lambda \,\rho ^{1}\rho ^{2}-z_{x}\,\rho ^{1}+\lambda ,
\vspace{5pt}
\\
{\displaystyle \rho _{t}^{1}=-\frac{1}{\lambda}{{\mathrm{e}}^{z}}{\left (
\rho ^{1}\right )}^{2}+\frac{1}{\lambda}e^{-2z}\rho ^{2}},
\vspace{5pt}
\\
\rho _{x}^{2}=-\lambda{\left (\rho ^{2}\right )}^{2}+z_{x}{\rho ^{2}}+
\lambda \rho ^{1},
\vspace{5pt}
\\
{\displaystyle \rho _{t}^{2}=-\frac{1}{\lambda}{{\mathrm{e}}^{z}}{\rho ^{1}}{
\rho ^{2}}+\frac{1}{\lambda}e^{z}},
\vspace{5pt}
\\
{\displaystyle z_{xt}=-{{\mathrm{e}}^{-2\,z}}+{{\mathrm{e}}^{z}}}.
\end{array}
\right .
\end{equation*}
Contrary to previous examples, in this case, we did not find any nontrivial
pseudosymmetry, which would be a pseudoprolongation of some vector field
relative to a multiple of $\mu $. However, we found that
$\mathcal{Y}$ admits a $2$-pseudosymmetry given by
\begin{equation*}
\bar{\mathbb{Y}}=\left .\mathbb{Y}\right |_{\mathcal{Y}^{(\infty )}}=
\left [\partial _{z}-\frac{2\rho ^{1}\rho ^{2}-1}{2\rho ^{1}}
\partial _{\rho ^{2}}+\ldots \;,\quad \partial _{\rho ^{1}}-
\frac{\rho ^{2}}{\rho ^{1}}\partial _{\rho ^{2}}+\ldots \right ]^{T},
\end{equation*}
induced on $\mathcal{Y}^{(\infty )}$ by a pseudoprolongation
$\mathbb{Y}$ relative to the horizontal $1$-form
\begin{equation*}
\begin{array}{l}
\gamma {=}\left ({\displaystyle
\begin{array}{c@{\quad}c}
{-}{\displaystyle
\frac{4\lambda \,(\rho ^{1}\rho ^{2})^{2}{+}2\lambda \,(\rho ^{1})^{3}{-}4\lambda \,\rho ^{1}\rho ^{2}{+}\lambda \,}{2\rho ^{2}(\rho ^{1})^{2}{-}\rho ^{1}}}
\vspace{10pt}
& {\displaystyle
\frac{12\lambda \,(\rho ^{1}\rho ^{2})^{2}{+}4\lambda \,(\rho ^{1})^{3}{-}12\lambda \,\rho ^{1}\rho ^{2}{+}3\lambda}{4\rho ^{1}\rho ^{2}{-}2}}
\\
{-}{\displaystyle
\frac{4\lambda \,\rho ^{1}(\rho ^{2})^{2}{+}4\lambda \,(\rho ^{1})^{2}{-}2\lambda \,\rho ^{2}}{2\rho ^{2}(\rho ^{1})^{2}{-}\rho ^{1}}}
\qquad & {\displaystyle
\frac{4\lambda (\rho ^{1})^{2}{+}4\lambda \rho ^{1}(\rho ^{2})^{2}{-}2z_{x}\rho ^{1}\rho ^{2}{-}2\lambda \rho ^{2}{+}z_{x}}{2\rho ^{1}\rho ^{2}{-}1}}
\end{array}
}\right )dx
\vspace{15pt}
\\
\qquad {+}\!\left (\!{\displaystyle
\begin{array}{c@{\quad}c}
{\displaystyle
\frac{4{\mathrm{e}}^{z}(\rho ^{1})^{3}\rho ^{2}{-}2{\mathrm{e}}^{z}(\rho ^{1})^{2}{+}6\rho ^{1}\,(\rho ^{2})^{2}{\mathrm{e}}^{{-}2\,z}{-}2\rho ^{2}{\mathrm{e}}^{{-}2\,z}}{\lambda \,\left (2\rho ^{2}(\rho ^{1})^{2}{-}\rho ^{1}\right )}}
\vspace{10pt}
& {-}{\displaystyle \frac{{\mathrm{e}}^{z}(\rho ^{1})^{2}}{\lambda}{-}3
\frac{\rho ^{2}{\mathrm{e}}^{{-}2z}}{\lambda}{+}
\frac{{\mathrm{e}}^{{-}2z}}{2\lambda \rho ^{1}}}
\\
{\displaystyle
\frac{4{\mathrm{e}}^{z}(\rho ^{1})^{2}\rho ^{2}{-}2\,{\mathrm{e}}^{z}\rho ^{1}{+}4(\rho ^{2})^{2}{\mathrm{e}}^{{-}2z}}{\lambda \,\left (2\rho ^{2}(\rho ^{1})^{2}{-}\rho ^{1}\right )}}
\qquad & {-}{\displaystyle 2\frac{{\mathrm{e}}^{z}\rho ^{1}}{\lambda}{-}
\frac{\rho ^{2}\,{\mathrm{e}}^{{-}2z}}{\lambda \rho ^{1}}}
\end{array}
}\! \right )dt,
\end{array}
\end{equation*}
satisfying
$d_{H}\gamma +\gamma \wedge \gamma =0\mod{\mathcal{Y}^{(\infty )}}$. Such
a $2$-pseudosymmetry of $\mathcal{Y}$ can be seen as a nonlocal $2$-pseudosymmetry
of $\mathcal{E}$, with respect to the differentiable covering defined by
$\mathcal{Y}$.

We are searching for a regular $\mathcal{C}$-morphism from
$\mathcal{Y}$ to $\mathcal{E}'$, with $\mathcal{E}'$ being a copy of
$\mathcal{E}$, defined by a system
$\{x'=\xi ^{1},\;t'=\xi ^{2},\;z'=\nu \}$ such that
$\xi ^{1},\xi ^{2},\nu $ are invariants of the $2$-pseudosymmetry given
by $\bar{\mathbb{Y}}$. Thus, being
$\{x,\,t,\,z+\ln (2\rho ^{1}\rho ^{2}-1)\}$ a basic system of zeroth order
invariants of $\bar{\mathbb{Y}}$, one is naturally lead to consider
$\xi ^{1}=x,\,\xi ^{2}=t$ and taking
$\nu =z+\ln (2\rho ^{1}\rho ^{2}-1)$ one obtains
$z'_{x't'}+{{\mathrm{e}}^{-2\,z'}}-{{\mathrm{e}}^{z'}}=0$, whenever
$\{z,\rho ^{1},\rho ^{2}\}$ is a solution of
$\mathcal{Y }$. Hence, one has the auto-B\"{a}cklund transformation
of $\mathcal{E}$
\begin{equation*}
x'=x,\qquad t'=t,\qquad z'=z+\ln (2\rho ^{1}\rho ^{2}-1),
\end{equation*}
provided by the invariants of the $2$-pseudosymmetry given by
$\bar{\mathbb{Y}}$. This auto-B\"{a}cklund transformation coincides with
that found in \cite{Brez}, by means of a different approach.
\end{example}

%\begin{authordisclosures}

%\section*{CRediT authorship contribution statement}

%\section*{Informed consent}

%\section*{Organ donation}

%\section*{...} %% Any other ethical consent

%\section*{Animal treatment}

%\section*{Declaration of generative AI and AI-assisted technologies in the writing process}

%\section*{Funding}

%\end{authordisclosures}

%DOCI
%
%\begin{conflict} % None declared.
%\end{conflict}

\section{Acknowledgments}
The authors acknowledge the Coordination for the Improvement of Higher Education Personnel - Brazil (CAPES) for financial support of the Article Processing Charge (APC) through the CAPES-Elsevier transformative agreement.

%\begin{appm}
%\def\the...{}
%\reset{}{}
%\appendix{}
%\appendix*{}
%\end{appm}%
\medskip{}
\medskip{}

\medskip{}
\medskip{}
\medskip{}

\end{document}